\documentclass[a4paper,12pt]{amsart}
\usepackage[latin1]{inputenc}
\usepackage[english]{babel}
\usepackage{amsmath, amsthm, amssymb, amsopn, amsfonts, amstext, stmaryrd, enumerate, color, mathtools, hyperref, mathrsfs, enumitem, url}
\usepackage[final]{showkeys}
\usepackage[margin=1.1in]{geometry}
\numberwithin{equation}{section}

\hypersetup{
	colorlinks=true,
	linkcolor=blue,
	citecolor=red,
	urlcolor=black,
	linktoc=all
}

\newcommand{\C}{\mathscr{C}}
\newcommand{\E}{\mathcal{E}}
\newcommand{\F}{\mathcal{F}}
\newcommand{\G}{\mathcal{G}}

\renewcommand{\L}{\mathcal{L}}
\newcommand{\N}{\mathbb{N}}
\newcommand{\R}{\mathbb{R}}

\newcommand{\W}{\mathbb{W}}
\newcommand{\loc}{{\rm loc}}

\newcommand{\dist}{{\mbox{\normalfont dist}}}
\newcommand{\bequ}{\begin{equation}}
\newcommand{\nequ}{\end{equation}}
\newcommand{\PV}{\mbox{\normalfont P.V.}}
\newcommand{\Haus}{\mathcal{H}}

\DeclareMathOperator{\diam}{diam}
\DeclareMathOperator{\supp}{supp}

\DeclareMathOperator{\Tail}{Tail}

\DeclareMathOperator{\DG}{DG}

\def\Xint#1{\mathchoice
{\XXint\displaystyle\textstyle{#1}}%
{\XXint\textstyle\scriptstyle{#1}}%
{\XXint\scriptstyle\scriptscriptstyle{#1}}%
{\XXint\scriptscriptstyle\scriptscriptstyle{#1}}%
\!\int}
\def\XXint#1#2#3{{\setbox0=\hbox{$#1{#2#3}{\int}$ }
\vcenter{\hbox{$#2#3$ }}\kern-.6\wd0}}

\def\dashint{\Xint-}

\theoremstyle{plain}
\newtheorem{definition}{Definition}[section]
\newtheorem{theorem}[definition]{Theorem}
\newtheorem{proposition}[definition]{Proposition}
\newtheorem{lemma}[definition]{Lemma}
\newtheorem{corollary}[definition]{Corollary}

\theoremstyle{definition}

\renewcommand{\le}{\leqslant}

\renewcommand{\ge}{\geqslant}

\def\Xint#1{\mathchoice
{\XXint\displaystyle\textstyle{#1}}%
{\XXint\textstyle\scriptstyle{#1}}%
{\XXint\scriptstyle\scriptscriptstyle{#1}}%
{\XXint\scriptscriptstyle\scriptscriptstyle{#1}}%
\!\int}
\def\XXint#1#2#3{{\setbox0=\hbox{$#1{#2#3}{\int}$ }
\vcenter{\hbox{$#2#3$ }}\kern-.6\wd0}}

\def\dashint{\Xint-}

\title[Regularity results and Harnack inequalities for nonlocal problems]{Regularity results and Harnack inequalities for minimizers and solutions of nonlocal problems:\\
a unified approach via fractional~De~Giorgi classes}

\author[Matteo Cozzi]{}



\begin{document}

\maketitle

\centerline{\scshape Matteo Cozzi}
\medskip
{\footnotesize
 \centerline{BGSMath Barcelona Graduate School of Mathematics}
 \centerline{and}
 \centerline{Universitat Polit\`ecnica de Catalunya}
 \centerline{Departament de Matem\`atiques}
 \centerline{Diagonal 647, E-08028 Barcelona (Spain)}
 \centerline{\href{mailto:matteo.cozzi@upc.edu}{\tt matteo.cozzi@upc.edu}}
}

\bigskip

\begin{abstract}
We study energy functionals obtained by adding a possibly discontinuous potential to an interaction term modeled upon a~Gagliardo-type fractional seminorm. We prove that minimizers of such non-differentiable functionals are locally bounded, H\"older continuous, and that they satisfy a suitable Harnack inequality. Hence, we provide an extension of celebrated results of~M.~Giaquinta and~E.~Giusti to the nonlocal setting. To do this, we introduce a particular class of fractional Sobolev functions, reminiscent of that considered by~E.~De~Giorgi in his seminal paper of~1957. The flexibility of these classes allows us to also establish regularity of solutions to rather general nonlinear integral equations.
\end{abstract}

\tableofcontents

\section{Introduction}

The aim of the present paper is to provide a rather general and comprehensive basic regularity theory for various problems connected to the minimization of nonlocal elliptic energy functionals. The simplest class of problems that we address here can be briefly described as follows.

\subsection{Description of the main results}

Let~$n \in \N$,~$s \in (0, 1)$ and~$p \in (1, +\infty)$. Let~$F: \R \to \R$ be a bounded measurable function. Given a bounded domain~$\Omega \subset \R^n$, we consider a minimizer~$u$ of the functional
\begin{equation} \label{Ffuncdef}
\G_s(v) := (1 - s) [v]_{W^{s, p}(\R^n)}^p + \int_{\R^n} F(v(x)) \, dx,
\end{equation}
within all functions~$v \in W^{s, p}(\R^n)$ that are equal to some given bounded function~$u_0$ outside of~$\Omega$.\footnote{Note that any well-defined notion of minimizer of the functional~$\G_s$ as defined in~\eqref{Ffuncdef} should involve some kind of restriction of~$\G_s$ to the set~$\Omega$, when comparing the energy of~$u$ to those of its competitors. Otherwise, in fact,~$\G_s$ might be always infinite, as a consequence of the prescribed values~$u_0$ outside of~$\Omega$. We will be more precise on the concept of minimizers that we take into account in Section~\ref{mainsec}. For the moment, we take the liberty of being slightly inaccurate and not worrying much about technicalities.} Here,~$[ \, \cdot \,]_{W^{s, p}(\R^n)}$ denotes the standard Gagliardo seminorm of the fractional Sobolev space~$W^{s, p}(\R^n)$. The main question that we positively answer in this work is the following:
$$
\mbox{\itshape Is it true that~$u$ is locally bounded and H\"older continuous inside~$\Omega$?}
$$

In their important paper~\cite{GG82},~M.~Giaquinta and~E.~Giusti studied the regularity of the minimizers of the energy
\begin{equation} \label{G1funcdef}
\mathcal{G}_1(v) := \| \nabla v \|_{L^p(\R^n)}^p + \int_{\R^n} F(v(x)) \, dx,
\end{equation}
and of more general functionals with~$L^p$ gradient structure. In particular, they proved interior~$L^\infty$ and~$C^\alpha$ estimates for the minimizers of~\eqref{G1funcdef}, when~$F$ is bounded (or even subcritical at infinity). One of the most prominent aspects of their work is that the potential term~$F$ may be discontinuous, and thus their theory can be applied to various non-differentiable functionals that arise for instance in connection with minimal surfaces, fluid dynamics, and free boundary problems.

Here, we provide an extension of these results to the nonlocal functional~$\mathcal{G}_s$ displayed in~\eqref{Ffuncdef} and to more general variants. Following an idea of~L.~Caffarelli,~C.~H.~Chan and~A.~Vasseur \cite{CCV11}, we prove that the minimizers of~$\mathcal{G}_s$ satisfy an~\emph{improved} fractional~Caccioppoli inequality and therefore belong to a particular subset of~$W^{s, p}$. We then show that the elements of this subset---which we call a~\emph{fractional~De~Giorgi class}, in honor of the one introduced by~E.~De~Giorgi in~\cite{DeG57}---are locally bounded, H\"older continuous functions that satisfy a Harnack estimate. By interpolating the technique of~\cite{CCV11} with a suitable isoperimetric-type inequality (that is proved to hold in~$W^{s, p}$ for large~$s$), we obtain estimates which are uniform as the differentiability order~$s$ goes to~$1^-$. We stress that our results are new even in the quadratic case~$p = 2$.

\medskip

As said previously, one key point of our analysis is that no differentiability is assumed on the potential~$F$. Instead, when~$F$ is, say, of class~$C^1$, we can differentiate the functional~$\G_s$ and deduce that~$u$ is a weak solution of the Euler-Lagrange equation
\begin{equation} \label{Lsp=F'}
\L_{s, p} u = f(u) \quad \mbox{in } \Omega,
\end{equation}
where~$f = F'$ and~$\L_{s, p}$ is formally defined in the principal value sense by
$$
\L_{s, p} u(x) := (1 - s) \, \PV \int_{\R^n} \frac{|u(x) - u(y)|^{p - 2} (u(x) - u(y))}{|x - y|^{n + sp}} \, dy,
$$
up to an irrelevant positive factor.

The nonlinear singular integral operator~$\L_{s, p}$ is often called~\emph{fractional~p-Laplacian} and has been extensively studied in the last few years. A first,~\emph{localized} version of~$\L_{s, p}$ has been originally introduced by~H.~Ishii and~G.~Nakamura in~\cite{IN10}, where the authors dealt with the solvability (in the viscosity sense) of the associated Dirichlet problem and established a connection with the classical~$p$-Laplace operator in the limit as~$s \rightarrow 1^-$. The first interior H\"older continuity results for~$\L_{s, p}$ have been obtained by~A.~Di~Castro,~T.~Kuusi and~G.~Palatucci in~\cite{DKP16}, concerning weak solutions of equations with vanishing right-hand side, and by~E.~Lindgren in~\cite{Lin16}, for viscosity solutions of equations with a bounded right-hand side. Global regularity for weak solutions of the Dirichlet problem has been then proved by~A.~Iannizzotto,~S.~Mosconi and~M.~Squassina in~\cite{IMS16}. The H\"older continuity of first derivatives is, to the best of our knowledge, still unknown. See the paper~\cite{BL17} by~L.~Brasco and~E.~Lindgren for regularity results in fractional Sobolev spaces of higher order.

Of course, when~$p = 2$ the operator~$\L_{s, p}$ boils down to the well-known fractional Laplacian
$$
(-\Delta)^s u(x) := (1 - s) \, \PV \int_{\R^n} \frac{u(x) - u(y)}{|x - y|^{n + 2 s}} \, dy.
$$
In this case, the theory is much more developed. See e.g.~the works~\cite{BK05,Kas07a,Kas09,Kas11} of~R.~Bass and~M.~Kassmann and~\cite{Sil06,CS09,CS11} of~L.~Caffarelli and~L.~Silvestre for regularity results for weak and viscosity solutions, respectively. See also~\cite{RS14,RS16} by~X.~Ros-Oton and~J.~Serra in relation to the fine boundary behavior of solutions to Dirichlet problems.

With the help of fractional~De~Giorgi classes, we can address here also the regularity of solutions to~\eqref{Lsp=F'} and to equations driven by more general integral operators. Under some mild hypotheses on the growth of~$f$, we show that the solutions of~\eqref{Lsp=F'} are locally bounded, H\"older continuous, and satisfy a Harnack-type inequality. In this way, we provide a full extension of the~De~Giorgi regularity theory to this nonlocal nonlinear setting.

\subsection{Motivations and applications}

Our principal motivation for studying this regularity problem comes from the couple of papers~\cite{CV15,CV16} where~E.~Valdinoci and the author prove the existence of a particular class of minimizers for an energy similar to~\eqref{Ffuncdef}, connected to phase-separation phenomena. There,~$p = 2$ and the term~$F$ is a double-well potential essentially modeled upon the functions
\begin{equation} \label{Fexpl}
F_d(u) := |1 - u^2|^d, \quad \mbox{with } d > 0.
\end{equation}
In local settings, potentials of this kind were considered for instance by~L.~Caffarelli and~A.~Cordoba in~\cite{CC95,CC06}.

The strategy followed in~\cite{CV15,CV16} for the construction of those minimizers relies on several considerations that involve only the energy functional. The corresponding~Euler-Lagrange equation is never used, if not for recovering minimal continuity properties of its solutions, via the already known regularity theory summarized in the previous subsection. This forced us there to stick with differentiable potentials, as such as~$F_d$ with~$d > 1$.

With the aid of the results obtained in the present paper, one could perform the same construction for a wider class of~$F$'s, as for instance the whole range of explicit double-well potentials considered in~\eqref{Fexpl}.\footnote{Of course, the potentials~$F_d$ introduced in~\eqref{Fexpl} are not bounded. However, the arguments presented in~\cite{CV15,CV16} always involve functions that assume values between~$-1$ and~$1$. Therefore, the local boundedness of~$F_d$ is enough in this case.}

As a matter of fact, the most adopted model in the context of phase-separation phenomena is that obtained by taking~$d = 2$ in~\eqref{Fexpl}---which indeed gives a~$C^\infty$ potential. This particular choice leads to the so-called fractional Allen-Cahn equation
$$
- (-\Delta)^s u = u - u^3,
$$
and to questions related to the nonlocal version of a famous conjecture by E. De~Giorgi. See~\cite{CS05,SV09,CC10,CC14,CS15} for progresses in this direction and~\cite{GG98,AC00,Sav09,DKW11,FV11,FV16} for the state-of-the-art results on the conjecture in the classical case, formally obtained by taking~$s = 1$. However, we believe that the analysis of the whole gamut of models given by~\eqref{Fexpl} might still be very interesting. It would be particularly important to understand the way in which the minimizers approach the pure states~$\pm 1$---either asymptotically or with the formation of free-boundaries---in relation to the parameters~$s$ and~$d$. For~$d = 2$ this has been accomplished in~\cite{CS05,PSV13,CS15}. See instead~\cite{CC06} for related results in the local case.

Of course, the regularity results provided here apply to much more general potentials~$F$. Other important classes of examples are given by
\begin{equation} \label{Fexamples}
\begin{aligned}
F^{(1)}(u) & := \chi_{(0, +\infty)}(u),\\
F^{(2)}(u) & := \chi_{(-1, 1)}(u),\\
F^{(3)}(u) & := \lambda_1 \chi_{(-\infty, 0)}(u) + \lambda_2 \chi_{(0, +\infty)}(u),
\end{aligned}
\end{equation}
with~$\lambda_1 \ne \lambda_2$. All these choices lead to nonlocal variants of functionals connected to free boundary problems and often arising in fluid dynamics. In the classical case, they have been diffusely studied by~H.~Alt,~L.~Caffarelli and~A.~Friedman in e.g.~\cite{AC81,ACF84}. In nonlocal settings, functionals of this kind have been considered by~L.~Caffarelli, J.-M.~Roquejoffre and~O.~Savin~\cite{CRS10b}.

\subsection{Strategy of the proof}

Our approach to the proof of the regularity of minimizers of the energy introduced in~\eqref{Ffuncdef} follows, in its most general lines, the one developed by~M.~Giaquinta and~E.~Giusti in~\cite{GG82} for functionals such as~$\mathcal{G}_1$ in~\eqref{G1funcdef}.

With the aid of Widman's hole-filling technique~\cite{Wid71} and of a suitable iteration lemma, the authors showed that any minimizer~$u$ of~$\G_1$ satisfies a family of Caccioppoli inequalities. More precisely, given any~$k \in \R$, the upper truncation~$(u - k)_-$ satisfies
\begin{equation} \label{caccintro}
\begin{aligned}
& \| \nabla (u - k)_- \|_{L^p(B_r(x_0))}^p \\
& \hspace{30pt} \le H \left[ \frac{1}{(R - r)^p} \| (u - k)_- \|_{L^p(B_R(x_0))}^p + d^p \left| B_R(x_0) \cap \{ u < k \} \right| \right],
\end{aligned}
\end{equation}
for any~$x_0 \in \Omega$, any~$0 < r < R < \dist(x_0, \partial \Omega)$ and some constants~$d \ge 0$,~$H \ge 1$. And analogously for the lower truncation~$(u - k)_+$.

A slightly simpler version of~\eqref{caccintro} was first obtained by E.~De~Giorgi in his pioneering work~\cite{DeG57}, where he singled out such inequality as the object encoding all the information about the H\"older continuity of the minimizers. The set of functions satisfying~\eqref{caccintro}---and more general inequalities of the same type---is now typically called a~\emph{De~Giorgi class}, in his honor.

The way one usually proceeds to prove that the elements of De~Giorgi classes are H\"older continuous functions is through the application of a so-called~\emph{growth lemma}. In its most basic formulation, the growth lemma tells that there exists a small universal constant~$\delta > 0$ such that
$$
\begin{dcases}
u \mbox{ satisfies } \eqref{caccintro},\\
u \ge 0 \mbox{ in } B_R(x_0),\\
\dfrac{\left| B_R(x_0) \cap \{ u \ge 1 \} \right|}{|B_R|} \ge \frac{1}{2},\\
d \le \delta,
\end{dcases}
\quad \Longrightarrow \quad u \ge \delta \mbox{ in } B_{\frac{R}{2}}(x_0),
$$
where~$d \ge 0$ is the constant multiplying the last summand on the right-hand side of~\eqref{caccintro}. By scaling and iterating this result on concentric balls of halving radii, at each step one obtains that either the lower bound of~$u$ increases of a small but universal quantity, or, thanks to a completely specular statement, the upper bound decreases of the same quantity. The quantification of this fact yields the desired H\"older continuity of~$u$.

The proof of the growth lemma is usually split into two sublemmata. First, one shows that if the superlevel set~$B_R(x_0) \cap \{ u \ge 2 \delta \}$ occupies a large portion of the ball~$B_R(x_0)$---say, having measure larger than~$(1 - \eta) |B_R|$, with~$\eta$ small and independent of~$\delta$---, then~$u \ge \delta$ in~$B_{R/2}(x_0)$. At a second stage, one checks that the assumption on the size of the measure of the~$2 \delta$-superlevel set is verified, provided~$\delta$ is chosen small enough. This last step is essentially a consequence of the following isoperimetric inequality for level sets of Sobolev functions:
\begin{equation} \label{isoplevset}
\Big[ \left| B_1 \cap \{ v \le 0 \} \right| \left| B_1 \cap \{ v \ge 1 \} \right| \Big]^{\frac{n - 1}{n}} \le C \| \nabla v \|_{L^p(B_1)} \left| B_1 \cap \{ 0 < v < 1 \} \right|^{\frac{p - 1}{p}},
\end{equation}
for some constant~$C \ge 1$ depending only on~$n$ and~$p$. We stress that~\eqref{isoplevset} holds for any~$v \in W^{1, p}(B_1)$, not only for minimizers. This inequality is due to~De~Giorgi and is already contained in~\cite{DeG57}. See also~\cite[Lemma~7.2 or~7.4]{Giu03},~\cite{CV12} or Section~\ref{sDGsec} here for more elementary proofs.\footnote{Notice that the exponents to which the two factors on the left-hand side of~\eqref{isoplevset} are raised may be slightly different in the works~\cite{DeG57,Giu03,CV12}. We refer to Lemma~\ref{DGisolem} here for a proof of~\eqref{isoplevset} (and more general inequalities) in this exact fashion.}

\medskip

In order to adapt this strategy to the minimizers of functional~$\G_s$ in~\eqref{Ffuncdef}, we mainly need to establish two things: a Caccioppoli-type inequality and an isoperimetric lemma such as~\eqref{isoplevset}.

Caccioppoli inequalities for the solutions of~\eqref{Lsp=F'} have already been obtained by many authors (see for instance~\cite{DKP16,KMS15a,BP16}). In our setting, they may be written, e.g.~for lower truncations, as
\begin{equation} \label{1scaccintro}
\begin{aligned}
& (1 - s) [(u - k)_-]_{W^{s, p}(B_r(x_0))}^p \\
& \hspace{20pt} \le H \Bigg[ \frac{R^{(1 - s) p}}{(R - r)^p} \| (u - k)_- \|_{L^p(B_R(x_0))}^p + d^p \left| B_R(x_0) \cap \{ u < k \} \right| \\
& \hspace{20pt} \quad + (1 - s) \frac{R^{n + s p}}{(R - r)^{n + s p}} \| (u - k)_- \|_{L^1(B_R(x_0))} \int_{\R^n \setminus B_r(x_0)} \frac{(u(x) - k)_-^{p - 1}}{|x|^{n + s p}} \, dx \Bigg],
\end{aligned}
\end{equation}
where, as before, we denote by~$[\, \cdot \,]_{W^{s, p}(U)}$ the standard Gagliardo seminorm for fractional Sobolev functions over a measurable set~$U \subseteq \R^n$, that is
\begin{equation} \label{GagspU}
[v]_{W^{s, p}(U)} := \left( \int_{U} \int_{U} \frac{|v(x) - v(y)|^p}{|x - y|^{n + s p}} \, dx dy \right)^{\frac{1}{p}}, \quad \mbox{for } v \in W^{s, p}(U).
\end{equation}
It can be checked that inequality~\eqref{1scaccintro} also holds true for the minimizers of~\eqref{Ffuncdef}. Notice the presence of an additional~\emph{tail} term on the third line of~\eqref{1scaccintro}, that takes into account the nonlocality of the functional~$\G_s$ or of the operator~$\L_{s,p}$. Moreover, the constants~$H$ and~$d$ are independent of~$s$ (the term~$d$ is essentially the~$L^\infty$ norm of~$F$ or~$F'$). Consequently, by the results of e.g.~\cite{BBM01}, we see that inequality~\eqref{1scaccintro} correctly approaches~\eqref{caccintro}, in the limit as~$s \rightarrow 1^-$. We further observe that the idea of deducing regularity properties from a fractional Caccioppoli inequality as~\eqref{1scaccintro} has been first considered by~G.~Mingione in~\cite{M11}. In this work, the author uses such approach in a local setting, and therefore the inequality considered there does not take into account any tail term. Nevertheless, the underlying philosophy is the same. 

In light of these facts, one might be tempted to consider the functions that satisfy~\eqref{1scaccintro} as elements of the fractional analogue of~De~Giorgi classes, and prove their H\"older continuity. This brings us to the second key ingredient: the~De~Giorgi isoperimetric inequality~\eqref{isoplevset}. As observed by~L.~Caffarelli and~A.~Vasseur~\cite{CV12} in the case~$p = 2$, formula~\eqref{isoplevset} \emph{``may be considered as a quantitative version of the fact that a function with jump discontinuity cannot be in~$H^1$''}. But functions with such discontinuity features may well belong to fractional Sobolev spaces: for instance, characteristic functions of sets with smooth boundaries are in~$W^{s, p}$, if~$s p < 1$. In fact, they play a central role in the theory of nonlocal perimeters recently developed by~L.~Caffarelli,~J.-M.~Roquejoffre and~O.~Savin in~\cite{CRS10a}. Therefore, the hopes of finding an appropriate generalization of~\eqref{isoplevset} to fractional Sobolev spaces are low, at least for small~$s$.

\medskip

While for~$s$ close to~$1$ we are able to partially extend inequality~\eqref{isoplevset} to~$W^{s, p}$ (see Proposition~\ref{sDGlemprop}) and therefore reproduce the strategy outlined before with no substantial modifications---thus proving a regularity result which is uniform as~$s \rightarrow 1^-$---, the case of a small~$s$ seems to require a different approach. Indeed, in this situation one needs to find an inequality providing the same information of~\eqref{isoplevset}, and holding at least for the minimizers of~$\G_s$ and the solutions of~\eqref{Lsp=F'}, instead of all functions in a Sobolev space. Inspired by~\cite{CCV11}, where the authors develop a regularity theory for parabolic equations driven by nonlocal operators with general kernels, we propose the following definition: a function~$u$ belongs to a~\emph{fractional~De~Giorgi class} in a bounded domain~$\Omega$ if and only if it satisfies the improved Caccioppoli inequality
\begin{equation} \label{2scaccintro}
\begin{aligned}
& (1 - s) \Bigg[ [(u - k)_-]_{W^{s, p}(B_r(x_0))}^p + \iint_{B_r(x_0)^2} \frac{(u(y) - k)_+^{p - 1} (u(x) - k)_-}{|x - y|^{n + s p}} \, dx dy \Bigg] \\
& \hspace{20pt} \le H \Bigg[ \frac{R^{(1 - s) p}}{(R - r)^p} \| (u - k)_- \|_{L^p(B_R(x_0))}^p + d^p \left| B_R(x_0) \cap \{ u < k \} \right| \\
& \hspace{20pt} \quad + (1 - s) \frac{R^{n + s p}}{(R - r)^{n + s p}} \| (u - k)_- \|_{L^1(B_R(x_0))} \int_{\R^n \setminus B_r(x_0)} \frac{(u(x) - k)_-^{p - 1}}{|x|^{n + s p}} \, dx \Bigg],
\end{aligned}
\end{equation}
for any~$k \in \R$,~$x_0 \in \Omega$,~$0 < r < R < \dist (x_0, \partial \Omega)$, and similarly for the upper truncations of~$u$.

Observe that, unlike in~\eqref{1scaccintro}, we have now two terms on the left-hand side. The newly added quantity
\begin{equation} \label{addquantity}
(1 - s) \iint_{B_r(x_0)^2} \frac{(u(y) - k)_+^{p - 1} (u(x) - k)_-}{|x - y|^{n + s p}} \, dx dy,
\end{equation}
is precisely the one that carries with itself all the information that we are missing not having an inequality as~\eqref{isoplevset} at hand. Indeed, when e.g.~$k = 1/2$, one can bound this term from below by
$$
\frac{1 - s}{C r^{n + s p}} \, |B_r(x_0) \cap \{ u \le 0 \}| |B_r(x_0) \cap \{ u \ge 1 \}|,
$$
with~$C \ge 1$ depending only on~$n$ and~$p$. The above quantity is similar to the one appearing on the left-hand side of~\eqref{isoplevset}. For functions in fractional~De~Giorgi classes, it can be bounded in terms of the right-hand side of~\eqref{2scaccintro}, and this fact is the much needed replacement for inequality~\eqref{isoplevset}.

In addition, it is not hard to show that, for instance when~$u \in W^{1, p}(B_r(x_0))$, the quantity in~\eqref{addquantity} vanishes as~$s \rightarrow 1^-$. Therefore, our definition of fractional~De~Giorgi classes as given by inequality~\eqref{2scaccintro} is consistent with the classical notion, in the limit as~$s \rightarrow 1^-$.

The main goal of the paper is to prove that the elements of fractional~De~Giorgi classes are locally bounded and H\"older continuous functions. Once we have this, the problem of establishing regularity for the minimizers of the functional~$\G_s$ defined in~\eqref{Ffuncdef} and the solutions of equation~\eqref{Lsp=F'} is reduced to show that these critical points belong to those classes.

\medskip

In the next section we give the definitions of the objects that we take into consideration, in their greater generality, and we present the rigorous statements of the results already discussed up to here for the simplified model governed by the energy functional~\eqref{Ffuncdef}.

\section{Precise formulation of the setting and of the main results} \label{mainsec}

Let~$n \in \N$,~$s \in (0, 1)$ and~$p \in (1, +\infty)$ be fixed parameters.

Let~$K: \R^n \times \R^n \to [0, +\infty]$ be a measurable function satisfying
\begin{equation} \label{Ksimm}
K(x, y) = K(y, x) \quad \mbox{for a.a.~} x, y \in \R^n,
\end{equation}
and
\begin{equation} \label{Kell}
\frac{(1 - s) \chi_{B_{r_0}}(x - y)}{\Lambda |x - y|^{n + s p}} \le K(x, y) \le \frac{(1 - s) \Lambda}{|x - y|^{n + s p}} \quad \mbox{for a.a.~} x, y \in \R^n,
\end{equation}
for some~$\Lambda \ge 1$ and~$r_0 > 0$.

Let~$F: \R^n \times \R \to \R$ be a measurable function such that the composition~$F \circ v$ is measurable, for any given measurable function~$v: \R^n \to \R$.\footnote{As it is customary, one can take as~$F$ any Carath\'eodory function, i.e.~such that~$F(\cdot, v)$ is measurable for any~$v \in \R$ and~$F(x, \cdot)$ is continuous for a.a.~$x \in \R^n$. Indeed, when~$F$ is a Carath\'eodory function, then~$F \circ v$ is measurable every time~$v$ is. However, we adopted a slightly broader definition in order to take into account for instance the examples listed in~\eqref{Fexamples} and many more.} For any measurable set~$\Omega \subseteq \R^n$ and measurable function~$u: \R^n \to \R$, we consider the energy functional
\begin{equation} \label{Edef}
\E(u; \Omega) := \frac{1}{2p} \iint_{\C_\Omega} \left| u(x) - u(y) \right|^p K(x, y) \, dx dy + \int_{\Omega} F(x, u(x)) \, dx,
\end{equation}
where
\begin{equation} \label{COmega}
\C_\Omega := \R^{2 n} \setminus \left( \R^n \setminus \Omega \right)^2.
\end{equation}

Note that, if one takes as~$K$ the standard kernel
\begin{equation} \label{K0}
K_0(x, y) := \frac{1 - s}{|x - y|^{n + s p}},
\end{equation}
then~$\E$ is the energy~$\G_s$ considered in the introduction, up to a negligible factor. However, hypothesis~\eqref{Kell} allows for a much richer variety of kernels. Indeed, we can equivalently rewrite~\eqref{Kell} as
$$
K(x, y) = (1 - s) \frac{a(x, y)}{|x - y|^{n + s p}}, \quad \mbox{with } \, \frac{\chi_{B_{r_0}(x - y)}}{\Lambda} \le a(x, y) \le \Lambda.
$$
Then one can choose for instance
\begin{equation} \label{aexamples}
\begin{array}{rll}
& a(x, y) = a_0 \left( \dfrac{x - y}{|x - y|} \right), & \mbox{with } a_0: S^{n - 1} \to \left[ \dfrac{1}{\Lambda}, \Lambda \right],\\
& a(x, y) = a_0 \left( x - y \right), & \mbox{with } a_0: \R \to \left[ \dfrac{1}{\Lambda}, \Lambda \right],\\
\mbox{or} & a(x, y) = \chi_{D}(x - y), & \mbox{with } D \subset \R^n, \mbox{ such that } B_{r_0} \subseteq D \subseteq B_{r_1}, \rule{0pt}{16pt}
\end{array}
\end{equation}
with~$0 < r_0 \le r_1$, and even some more general non-translation-invariant kernels. The last possibility in~\eqref{aexamples}, in particular, yields a truncated kernel.

\medskip

We now introduce the concept of minimizers of~$\E$ that we shall work with. Before, we need a few definitions of the functional spaces involved.

We denote by~$L_s^{p - 1}(\R^n)$ the space composed by the functions~$u: \R^n \to \R$ that satisfy
$$
\int_{\R^n} \frac{|u(x)|^{p - 1}}{\left( 1 + |x| \right)^{n + s p}} \, dx < +\infty.
$$
An object that will play an important role in encoding the behavior of functions in~$L_s^{p - 1}(\R^n)$ at large scales is the so-called~\emph{Tail} defined by
\begin{equation} \label{Tailudef}
\Tail(u; x_0, R) := \left[ (1 - s) R^{s p} \int_{\R^n \setminus B_R(x_0)} \frac{|u(x)|^{p - 1}}{|x - x_0|^{n + s p}} \, dx \right]^{\frac{1}{p - 1}},
\end{equation}
for any fixed~$x_0 \in \R^n$ and~$R > 0$. The above scale-invariant quantity has been introduced in~\cite{DKP14,DKP16} and also considered in several other papers, such as~\cite{BP16,KMS15b}. Observe in particular that it is finite, provided~$u \in L^{p - 1}_s(\R^n)$.

Another functional space that we will need is a modified fractional Sobolev space. Given an open set~$\Omega \subseteq \R^n$, we denote as~$\W^{s, p}(\Omega)$ the space of measurable functions~$u: \R^n \to \R$ such that
$$
u|_\Omega \in L^p(\Omega) \quad \mbox{and} \quad (x, y) \longmapsto \frac{|u(x) - u(y)|^p}{|x - y|^{n + s p}} \in L^1(\C_\Omega),
$$
with~$\C_\Omega$ as in~\eqref{COmega}. In the quadratic case~$p = 2$, spaces of this kind have been for instance considered in~\cite{SV13,SV14}. Note that~$W^{s, p}(\R^n) \subseteq \W^{s, p}(\Omega) \subseteq W^{s, p}(\Omega)$.

We can now proceed to define the minimizers of the energy functional~$\E$ in~\eqref{Edef}.
\begin{definition} \label{mindef}
Let~$\Omega \subset \R^n$ be a bounded open set. A function~$u \in \W^{s, p}(\Omega)$ is said to be a~\emph{minimizer} of~$\E$ in~$\Omega$ if~$F(\cdot, u) \in L^1(\Omega)$ and
$$
\E(u; \Omega) \le \E(v; \Omega),
$$
for any measurable function~$v: \R^n \to \R$ such that~$v = u$ a.e.~in~$\R^n \setminus \Omega$.
\end{definition}

If~$u \in \W^{s, p}(\Omega)$ is such that~$F(\cdot, u) \in L^1(\Omega)$, then the energy~$\E(u; \Omega)$ is finite, thanks to~\eqref{Kell}. Therefore, Definition~\ref{mindef} is meaningful.

On top of this, notice that~$F(\cdot, u)$ is bounded, and thus integrable, whenever~$u$ is bounded and~$F$ is locally bounded in~$u$, uniformly with respect to~$x \in \Omega$. When~$n < s p$, this is true in particular for any~$u \in W^{s, p}(\Omega)$, thanks to the fractional Sobolev embedding (see e.g.~\cite{DPV12}). On the contrary, when~$n > sp$, we will often requre~$F$ to satisfy
\begin{equation} \label{Fbounds}
|F(x, u)| \le d_1 + d_2 |u|^q \quad \mbox{for a.a.~} x \in \Omega \mbox{ and any } u \in \R,
\end{equation}
with~$1 \le q < p^*_s$, where~$p^*_s$ is the fractional Sobolev exponent given by
\begin{equation} \label{p*sdef}
p^*_s := \frac{n p}{n - s p}.
\end{equation}
When~$n = s p$, we simply assume~$F$ to satisfy~\eqref{Fbounds} for some~$1 \le q < +\infty$, that is, we formally set~$p^*_s := +\infty$ in such case. Note that, under~\eqref{Fbounds}, we have~$F(\cdot, u) \in L^1(\Omega)$ for any~$u \in W^{s, p}(\Omega)$.

\medskip

In parallel to the energy~$\E$ considered in~\eqref{Edef} and its minimizers, we take into account weak solutions to equations driven by integral operators with kernel~$K$. For~$K$ satisfying~\eqref{Ksimm} and~\eqref{Kell}, we formally define
\begin{align*}
\L u(x) := & \, \PV \int_{\R^n} |u(x) - u(y)|^{p - 2} (u(x) - u(y)) K(x, y) \, dy \\
= & \, \lim_{\delta \rightarrow 0^+} \int_{\R^n \setminus B_\delta(x)} |u(x) - u(y)|^{p - 2} (u(x) - u(y)) K(x, y) \, dy.
\end{align*}
When~$K$ is of the form~\eqref{K0}, we recover the operator~$\L_{s, p}$ defined in the introduction and, in particular, the fractional Laplacian~$(-\Delta)^s$ if~$p = 2$.

Let~$f: \R^n \times \R \to \R$ be a measurable function such that the composition~$f(\cdot, v)$ is measurable whenever~$v: \R^n \to \R$ is measurable. We are interested in studying weak solutions of the equation
\begin{equation} \label{Lu=f}
- \L u = f(\cdot, u),
\end{equation}
in bounded domains of~$\R^n$. Notice that, when~$F$ is differentiable~$u$, this is the~Euler-Lagrange equation of~$\E$, with~$f = F_u$.

The precise notion of weak solution of~\eqref{Lu=f} that we take into account is as follows.

\begin{definition} \label{soldef}
Let~$\Omega \subset \R^n$ be a bounded open set. A function~$u \in \W^{s, p}(\Omega)$ is said to be a~\emph{weak solution} of~\eqref{Lu=f} in~$\Omega$ if
\begin{multline*}
- \frac{1}{2} \int_{\R^n} \int_{\R^n} |u(x) - u(y)|^{p - 2} \left( u(x) - u(y) \right) \left( \varphi(x) - \varphi(y) \right) K(x, y) \, dx dy \\
= \int_{\R^n} f(x, u(x)) \varphi(x) \, dx,
\end{multline*}
for any~$\varphi \in W^{s, p}(\R^n)$ with~$\supp(\varphi) \subset \subset \Omega$ and such that the right-hand side above is well-defined.
\end{definition}

We have been rather sloppy in our definition of weak solutions, with respect to the choice of test functions that make the right-hand side of the identity written above converge. A sufficient condition to have it well-defined for any~$\varphi \in W^{s, p}(\Omega)$ is that~$f(\cdot, u) \in L^{(p^*_s)'}(\Omega)$ when~$n > sp$, and simply~$f(\cdot, u) \in L^1(\Omega)$ when~$n < sp$, thanks to the fractional Sobolev embeddings. In the case~$n < s p$, in particular, this last requirement on~$f(\cdot, u)$ is fulfilled whenever~$f$ is locally bounded in~$u \in \R$, uniformly w.r.t.~$x \in \Omega$, by employing the Sobolev embedding once again. On the other hand, when~$n \ge s p$ we will almost always ask that
\begin{equation} \label{fbounds}
|f(x, u)| \le d_1 + d_2 |u|^{q - 1} \quad \mbox{for a.a.~} x \in \Omega \mbox{ and any } u \in \R,
\end{equation}
for some~$1 < q < p^*_s$, similarly to what we did in~\eqref{Fbounds} (again, with the understanding that~$p^*_s = +\infty$ if~$n = s p$). We remark that, with this choice,~$f(\cdot, u) \varphi \in L^1(\Omega)$ for any~$\varphi \in W^{s, p}(\Omega)$. We also notice that~$f(\cdot, u) \varphi \in L^1(\Omega)$ for all such~$\varphi$'s if~$f$ is locally bounded in~$u$, uniformly w.r.t.~$x$, and~$u$ is bounded, regardless to the values of~$n$,~$s$ and~$p$.

\bigskip

After all these necessary premises, we can now proceed to state our main contributions to the regularity theory for the minimizers of~$\E$ and the solutions of~\eqref{Lu=f}.

\medskip

First, we have the following result concerning the local boundedness of these critical points. Of course, we can restrict ourselves to take~$n \ge s p$, as otherwise the boundedness is warranted by the fractional Sobolev embedding.

\begin{theorem}[\bfseries Local boundedness] \label{boundmainthm}
Let~$n \in \N$,~$s \in (0, 1)$ and~$p > 1$ be such that~$n \ge s p$. Let~$\Omega$ be an open bounded subset of~$\R^n$. Suppose that the kernel~$K$ satisfies hypotheses~\eqref{Ksimm} and~\eqref{Kell}. Let~$u \in L^{p - 1}_s(\R^n) \cap \W^{s, p}(\Omega)$ be either
\begin{enumerate}[label=$(\alph*)$,leftmargin=*]
\item a minimizer of~$\E$ in~$\Omega$, with~$F$ satisfying~\eqref{Fbounds}, or
\item a weak solution of~\eqref{Lu=f} in~$\Omega$, with~$f$ satisfying~\eqref{fbounds}.
\end{enumerate}
Then,~$u \in L^\infty_\loc(\Omega)$. In particular, for any~$x_0 \in \Omega$ and~$0 < R < \dist(x_0, \partial \Omega) / 2$, it holds
$$
\| u \|_{L^\infty(B_R(x_0))} \le C,
$$
for some constant
$$
C = C \left( n, s, p, q, \Lambda, d_1, d_2, R, r_0, \| u \|_{L^p(B_{2 R}(x_0))}, \Tail(u; x_0, 2 R), \| u \|_{L^\lambda(\Omega)} \right),
$$
and~$\lambda > p$. When~$n > s p$, we can take~$\lambda = p^*_s$, while when~$n > p$, the constant~$C$ does not blow up as~$s \rightarrow 1^-$.
\end{theorem}

The estimate of Theorem~\ref{boundmainthm} is given in terms of an implicit constant~$C$ that is meant to depend~\emph{at most} on the parameters listed. Indeed, in many specific cases (such as when~$d_2 = 0$ or~$1 \le q \le p$) the constant~$C$ may be chosen to depend on fewer quantities. We refer the reader to Theorems~\ref{minboundthm} and~\ref{solboundthm}---respectively for minimizers and solutions---for a more detailed account of the dependencies of~$C$.

We point out that, at least when~$n > p$, the estimate provided in Theorem~\ref{boundmainthm} is independent of~$s$, for~$s$ close to~$1$. This is not surprising at all, in view of the normalization of the kernel~$K$ by means of the factor~$(1 - s)$, implied by~\eqref{Kell}. Indeed, this is consistent with the fact that, for certain choices of kernels, the energy~$\E$ approaches, as~$s \rightarrow 1^-$, a local functional driven by a gradient-type Dirichlet term, such as~$\G_1$ in~\eqref{G1funcdef}. And similarly for the operator~$\L$.

Theorem~\ref{boundmainthm} has been obtained in~\cite{DKP16} for the case of solutions of an equation like~\eqref{Lu=f}. In~\cite{DKP16}, the result is stated assuming the right-hand side~$f$ to be zero, although the techniques displayed there should apply also to more general situations. This is true, since the boundedness of~$u$ may be recovered right from a standard Caccioppoli inequality of the form~\eqref{1scaccintro}, and does not require its improved variant~\eqref{2scaccintro}. We refer the interested reader to the proof of Proposition~\ref{ulocboundprop} for a verification of this fact.

Other results related to Theorem~\ref{boundmainthm} can be found for instance in~\cite{KMS15b,IMS16,LPPS15,BPV15}.

\medskip

Next is the main contribution of the present paper, ensuring the H\"older continuity of the minimizers of~$\E$ and of the solutions of~\eqref{Lu=f}. Again, we only deal with the case~$n \ge s p$. Moreover, by virtue of the boundedness result of Theorem~\ref{boundmainthm}, we can now simply assume the potential~$F$ and the right-hand side~$f$ to be locally bounded functions.

The statement of the H\"older continuity result is as follows.

\begin{theorem}[\bfseries H\"older continuity] \label{holmainthm}
Let~$n \in \N$,~$0 < s_0 \le s < 1$ and~$p > 1$ be such that~$n \ge s p$. Let~$\Omega$ be an open bounded subset of~$\R^n$. Suppose that the kernel~$K$ satisfies hypotheses~\eqref{Ksimm} and~\eqref{Kell}. Let~$u \in L^{p - 1}_s(\R^n) \cap \W^{s, p}(\Omega)$ be either
\begin{enumerate}[label=$(\alph*)$,leftmargin=*]
\item a minimizer of~$\E$ in~$\Omega$, with~$F$ locally bounded in~$u$, uniformly w.r.t.~$x \in \Omega$, or
\item a weak solution of~\eqref{Lu=f} in~$\Omega$, with~$f$ locally bounded in~$u$, uniformly w.r.t.~$x \in \Omega$.
\end{enumerate}
Then,~$u \in C^\alpha_\loc(\Omega)$, for some~$\alpha \in (0, 1)$. In particular, there exists a constant~$C \ge 1$ such that, for any~$x_0 \in \Omega$ and~$0 < R < \min \{ r_0, \dist(x_0, \partial \Omega) \} / 4$, it holds
$$
[u]_{C^\alpha(B_R(x_0))} \le \frac{C}{R^\alpha} \Big( \| u \|_{L^\infty(B_{2 R}(x_0))} + \Tail(u; x_0, 2 R) + \F \Big),
$$
where
\begin{equation} \label{Fdef}
\F := \begin{dcases}
R^s \| F(\cdot, u) \|_{L^\infty(B_{2 R}(x_0))}^{1 / p} \vphantom{\frac{0}{0}} & \quad \mbox{if } (a) \mbox{ is in force}, \\
R^{\frac{s p}{p - 1}} \| f(\cdot, u) \|_{L^\infty(B_{2 R}(x_0))}^{1 / (p - 1)} \vphantom{\frac{0}{0}} & \quad \mbox{if } (b) \mbox{ is in force}.
\end{dcases}
\end{equation}
The constants~$\alpha$ and~$C$ depend only on~$n$,~$s_0$,~$p$ and~$\Lambda$.
\end{theorem}

Analogously to Theorem~\ref{boundmainthm}, the quantities determining the H\"older character of, say, a minimizer of~$\E$ stay bounded as~$s$ goes to~$1$. Again, this is consistent with the local scenario (formally represented by the choice~$s = 1$) where such results were proved in~\cite{GG82}.

In the case of a solution~$u$ of~\eqref{Lu=f}, estimates like the one established in Theorem~\ref{holmainthm} have been obtained in~\cite{DKP16}, for~$f = 0$, and in~\cite{IMS16}, for Dirichlet problems driven by the specific operator~$\L_{s, p}$ defined in the introduction (i.e.~the operator~$\L$ with kernel~$K$ given by~\eqref{K0}). When~$p = 2$, the literature is richer: similar results have been obtained in~\cite{BK05,Kas07a,Kas09,Kas11} and~\cite{Sil06,CS09,CS11}. We also mention~\cite{KMS15b}, where the authors show the continuity of~$u$ under very mild hypotheses on the right-hand side~$f$. Their potential theoretic approach should also yield an H\"older modulus of continuity for~$u$, once~$f$ is chosen sufficiently regular. 

As pointed out in the introduction, Theorems~\ref{boundmainthm} and~\ref{holmainthm} are, to the best of our knowledge, completely new for minimizers of~$\E$, even if~$p = 2$.

The verification of Theorem~\ref{holmainthm} is split between Section~\ref{minsec}---for minimizers---and Section~\ref{solsec}---for solutions. The~$C^\alpha$ estimate in the two different situations is respectively given by Theorem~\ref{minholdthm} in Section~\ref{minsec} and Theorem~\ref{solholdthm} in Section~\ref{solsec}. We remark that the statements of these two results partially differ from that of Theorem~\ref{holmainthm}, with respect to some limitations on the radius~$R$. As for Theorem~\ref{boundmainthm} and Theorems~\ref{minboundthm}-\ref{solboundthm}, the result stated right above can be easily recovered from those of Sections~\ref{minsec}-\ref{solsec} with the help of a straightforward covering argument.

A key ingredient of the proof of Theorem~\ref{holmainthm} is the improved Caccioppoli inequality~\eqref{2scaccintro}. With the aid of this estimate---that holds, in a sometimes slightly weaker form, for both solutions and minimizers---we are able to prove a growth lemma, the crucial step for the H\"older continuity. In particular, we use the bound for the second member on the left-hand side of~\eqref{2scaccintro} to replace the~De~Giorgi isoperimetric-type inequality~\eqref{isoplevset}, which may fail in the context of fractional Sobolev spaces.

Observe once again that the estimate provided in Theorem~\ref{holmainthm} is uniform in~$s$, at least when~$s$ is bounded away from~$0$. On the other hand, the information represented by the upper bound for the second term on the left-hand side of~\eqref{2scaccintro} tends to disappear as~$s$ approaches~$1$. As a result, one would naturally expect H\"older estimates which blow up as~$s \rightarrow 1^-$. To resolve this apparent inconsistency, in Proposition~\ref{sDGlemprop} we obtain an estimate in the spirit of~\eqref{isoplevset} for functions belonging to (large regions of) the fractional Sobolev space~$W^{s, p}$, with~$s$ sufficiently close to~$1$. The interpolation of this result with inequality~\eqref{2scaccintro} yields~$C^\alpha$ estimates uniform in~$s$.

\medskip

The last result that we present in this section is a Harnack-type inequality. Note that here we do not limit ourselves to~$n \ge s p$, as the result is now meaningful for the full range of parameters. On the other hand, in place of~\eqref{Kell}, we take into account the following slightly more restrictive hypothesis on~$K$:
\begin{equation} \label{Kell2}
\frac{1 - s}{\Lambda |x - y|^{n + s p}} \le K(x, y) \le \frac{(1 - s) \Lambda}{|x - y|^{n + s p}} \quad \mbox{for a.a.~} x, y \in \R^n,
\end{equation}
with~$\Lambda \ge 1$. Observe that~\eqref{Kell2} differs from~\eqref{Kell} in that the left-hand inequality---i.e.~the ellipticity assumption on~$K$---is now required to hold everywhere, instead of only in a neighborhood of the diagonal~$\{ x = y\}$. Hypothesis~\eqref{Kell2} formally corresponds to~\eqref{Kell} with~$r_0 = +\infty$.

The statement of the Harnack inequality may now follow.

\begin{theorem}[\bfseries Harnack inequality] \label{harmainthm}
Let~$n \in \N$,~$s \in (0, 1)$ and~$p > 1$. Let~$\Omega \subset \R^n$ be an open bounded subset of~$\R^n$. Suppose that~$K$ satisfies hypotheses~\eqref{Ksimm} and~\eqref{Kell2}. Let~$u$,~$F$ and~$f$ be as in Theorem~\ref{holmainthm}, and assume in addition that~$u \ge 0$ in~$\Omega$. Then, there exists a constant~$C \ge 1$ such that, for any~$x_0 \in \Omega$ and~$0 < R < \dist(x_0, \partial \Omega) / 2$, it holds
$$
\sup_{B_R(x_0)} u \le C \left( \inf_{B_R(x_0)} u + \Tail(u_-; x_0, R) + \F \right),
$$
with~$\F$ as in~\eqref{Fdef}. The constant~$C$ depends only on~$n$,~$s$,~$p$ and~$\Lambda$. When~$n \notin \{ 1, p \}$, the constant~$C$ does not blow up as~$s \rightarrow 1^-$.
\end{theorem}

Harnack inequalities for solutions to integral equations like~\eqref{Lu=f} have been obtained in~\cite{Kas11,DKP14}, both considering the homogeneous case (i.e.~with no right-hand side). For minimizers, this is the first available result in this direction.

When comparing Theorem~\ref{harmainthm} to the classical Harnack inequalities for second-order partial differential equations (see e.g.~\cite{M61,DT84}), we immediately notice the presence here of an additional Tail term. As noted in~\cite{Kas07b,Kas11}, this is the correct formulation of the Harnack inequality for nonlocal operators. Of course, when~$u$ is non-negative on the whole of~$\R^n$, one recovers the classical inequality.

Moreover, as the radius~$R$ of the ball~$B_R(x_0)$ over which the inequality is set can be chosen freely (as long as~$B_{2 R}(x_0)$ is contained in~$\Omega$), this is a global inequality. As a consequence, one can deduce from it a Liouville-type theorem for entire solutions of~$\L u = 0$ which are bounded from above or below.

We do not know whether or not the stronger assumption~\eqref{Kell2}, in place of~\eqref{Kell}, is necessary for the validity of Theorem~\ref{harmainthm}. The only place where we take full advantage of~\eqref{Kell2} is in Theorem~\ref{DGharthm}: it is used to deduce formula~\eqref{Tail+control}, which gives a bound for the Tail of the positive part of~$u$. We believe it to be an interesting problem to understand if a similar Harnack inequality might be obtained for more general kernels. To this aim, it would possibly be convenient to take into account a Tail term tailored on the kernel~$K$ under consideration, rather than the canonical choice~\eqref{Tailudef} given by~\eqref{K0}.

\medskip

We also stress that, coherently with the path traced in~\cite{GG82}, one could be interested in understanding if the minimizers and solutions under consideration belong to fractional~Sobolev spaces of higher integrability order. This has been proved in~\cite{KMS15a} for equations driven by nonlocal operators with linear growth and barely measurable kernels. Quite surprisingly and differently from the local scenario, the authors obtained there that the solutions embed in Sobolev spaces having also higher order of~(fractional) differentiability. As noted at the end of Section~1B in~\cite{KMS15a}, their techniques should extend without too many difficulties to solutions of nonlinear equations with more general growths and, possibly, to minimizers. This might be the object of future works.

\bigskip

As anticipated in the introductory section, we obtain Theorems~\ref{boundmainthm},~\ref{holmainthm} and~\ref{harmainthm} by showing that the minimizers of~$\E$ and the solutions of~\eqref{Lu=f} equally satisfy an improved Caccioppoli inequality of the form~\eqref{2scaccintro}. We consider the set of all functions fulfilling~\eqref{2scaccintro} and more general inequalities---which we call a~\emph{fractional~De~Giorgi class}---and prove that they are locally bounded, H\"{o}lder continuous and that, when non-negative, they satisfy a Harnack inequality.

In light of this, the present paper extends various results and techniques displayed in the classical references~\cite{DeG57,LU68,GG82,DT84,Giu03} to a nonlocal setting.

\bigskip

The remaining part of the paper is organized as follows.

First, in Section~\ref{notsec} we fix some terminology that is often adopted in the paper.

In the preparatory Section~\ref{prepsec} we include a collection of numerical and functional inequalities that will be largely used in the subsequent sections.

Section~\ref{sDGsec} is devoted to the proof of a~De~Giorgi isoperimetric-type inequality for the level sets of functions that belong to fractional Sobolev spaces with large differentiability order~$s$.

In Section~\ref{DGsec} we introduce fractional~De~Giorgi classes in the full generality needed for our applications. There, we also show that their elements are locally bounded, H\"older continuous functions that satisfy an Harnack-type theorem. These three facts are proved in Theorems~\ref{DGboundthm},~\ref{DGholdthm} and~\ref{DGharthm}.

The conclusive Sections~\ref{minsec} and~\ref{solsec} contain the proofs of Theorems~\ref{boundmainthm},~\ref{holmainthm},~\ref{harmainthm}: these results are restated as Theorems~\ref{minboundthm},~\ref{minholdthm},~\ref{minharthm} for minimizers, in Section~\ref{minsec}, and as Theorems~\ref{solboundthm},~\ref{solholdthm},~\ref{solharthm} for solutions, in Section~\ref{solsec}.

\section{Notation} \label{notsec}

In this brief section, we formally specify some of the notation that will be used more frequently in the remainder of the paper.

\smallskip

First of all, the dimension of the space in which we are set is always indicated by~$n$, which is normally meant to be any natural number.

As we did in the two previous sections, we denote by~$B_R(x_0)$ the open Euclidean ball of radius~$R > 0$, centered at~$x_0 \in \R^n$. That is,
$$
B_R(x_0) := \Big\{ x \in \R^n : |x - x_0| < R \Big\}.
$$
When~$x_0$ is the origin, we simply write~$B_R$ in place of~$B_R(0)$.

We use the symbol~$\chi_\Omega$ to indicate the characteristic function of a set~$\Omega \subseteq \R^n$, i.e.
$$
\chi_\Omega(x) := \begin{cases}
1 & \quad \mbox{if } x \in \Omega,\\
0 & \quad \mbox{if } x \in \R^n \setminus \Omega.
\end{cases}
$$

For any two given parameters~$s \in (0, 1)$,~$p > 1$ and any measurable set~$U \subseteq \R^n$, we have already introduced the fractional Sobolev space~$W^{s, p}(U)$ as the subset of~$L^p(U)$ made up by those functions that have finite Gagliardo seminorm~$[\, \cdot \,]_{W^{s, p}(U)}$, as given by~\eqref{GagspU}. For~$n > s p$, the important fractional Sobolev exponent~$p^*_s$ has been defined in~\eqref{p*sdef}. In section~\ref{mainsec}, we also considered the modified Sobolev space~$\W^{s, p}(U)$ and the weighted Lebesgue space~$L^{p - 1}_s(\R^n)$. For~$u \in L^{p - 1}_s(\R^n)$ and any~$x_0 \in \R^n$,~$R > 0$, we saw that the quantity~$\Tail(u; x_0, R)$, as in~\eqref{Tailudef}, is well-defined and finite. For some later purposes, it is convenient to introduce also the related non-scaling-invariant Tail term
\begin{equation} \label{nsiTailudef}
\begin{aligned}
\overline{\Tail}(u; x_0, R) := & \, \left[ (1 - s) \int_{\R^n \setminus B_R(x_0)} \frac{|u(x)|^{p - 1}}{|x - x_0|^{n + s p}} \, dx \right]^{\frac{1}{p - 1}} \\
= & \, R^{- \frac{s p}{p - 1}} \, \Tail(u; x_0, R).
\end{aligned}
\end{equation}

We adopt a short-hand notation for the level sets of functions. Given~$u: \R^n \to \R$ and~$k \in \R$, we denote the superlevel set of~$u$ of level~$k$ as
$$
\{ u > k \} := \Big\{ x \in \R^n : u(x) > k \Big\}.
$$
Similarly, we write~$\{ u < k \}$ for the sublevel set~$\{ x \in \R^n : u(x) < k \}$. The other notations~$\{ u = k \}$,~$\{ u \ge k \}$ and~$\{ u \le k \}$ all have analogous meanings.

As it is customary, the positive and negative parts of a function (or a real number)~$u$ are indicated by~$u_+$ and~$u_-$, respectively. This means that we have~$u_+ := \max \{ u, 0 \}$ and~$u_- := - \min \{ u, 0 \}$.

Of particular interest are also the lower truncation~$(u - k)_+$ and the upper truncation~$(u - k)_-$ of a function~$u$ at level~$k \in \R$. We often refer to their supports as
\begin{equation} \label{A+-def}
\begin{aligned}
A^+(k) := \supp((u - k)_+) = \{ u > k \}, \\
A^-(k) := \supp((u - k)_-) = \{ u < k \}.
\end{aligned}
\end{equation}
The intersections of these sets with the ball~$B_R(x_0)$ are denoted by~$A^+(k, x_0, R)$ and~$A^-(k, x_0, R)$, respectively. As before, we drop reference to~$x_0$ when it is the origin, and simply write~$A^+(k, R)$ and~$A^-(k, R)$.

In Sections~\ref{minsec} and~\ref{solsec}, we will frequently consider the measure element
\begin{equation} \label{dmudef}
d\mu = d\mu_K(x, y) := K(x, y) \, dx dy.
\end{equation}
This terminology is used for the sole purpose of abbreviating several integral formulas.

Finally, we remark that we use several letters (roman or greek characters, in upper or lower cases) to denote constants and parameters. Sometimes---as in Theorem~\ref{boundmainthm} or Proposition~\ref{minareDGprop}---we write the quantities on which some constant depends between round brackets, right after the symbol used for said constant. We always use the letter~$C$ to denote a general constant, greater or equal to~$1$. The value of~$C$ may change within the same statement, proof or even between different lines of the same formula. During proofs, we usually specify on which parameters a certain constant~$C$ depends as soon as it appears in a formula; eventual other occurrences of~$C$ in the same proof are supposed to depend on the same exact parameters, unless otherwise specified. When we need to be more precise on the value of some particular occurrence, we use subscripts, such as~$C_\star, C_\sharp, C_1, C_2$, etc.

\section{Some auxiliary results} \label{prepsec}

Here we present several ancillary lemmata that will be used in the remainder of the paper. For their technical nature and rather general applicability, we preferred to collect them in this separate section.

The first four results are standard numerical inequalities. Most of them are probably well-known to the reader or very easy to be obtained. For the sake of completeness, we include their proofs in full details.

\begin{lemma} \label{numestlem1}
Let~$p \ge 1$ and~$a, b \ge 0$. Then,
\begin{equation} \label{numest1}
(a + b)^p - a^p \ge \theta p a^{p - 1} b + (1 - \theta) b^p,
\end{equation}
for any~$\theta \in [0, 1]$.
\end{lemma}
\begin{proof}
Of course, it is enough to prove~\eqref{numest1} for~$\theta = 0, 1$. Indeed, the case~$\theta = 0$ plainly follows from the standard fact that the~$p$-norm~$|x|_p$ is monotone non-increasing in~$p \in [1, +\infty)$, for any fixed~$x \in \R^2$. On the other hand,
$$
(a + b)^p - a^p = p \int_a^{a + b} t^{p - 1} \, dt \ge p a^{p - 1} b,
$$
which is~\eqref{numest1} for~$\theta = 1$.
\end{proof}

We observe that, when~$p \ge 2$, a simpler and stronger inequality holds true. Essentially, in this case one can replace both coefficients~$\theta$ and~$1 - \theta$ on the right-hand side of~\eqref{numest1} with~$1$. However, for our applications the interpolation inequality of Lemma~\ref{numestlem1} will suffice.

Next are other three lemmata providing numerical estimates.

\begin{lemma} \label{numestlem2}
Let~$p \ge 1$,~$\mu \in [0, 1]$ and~$a, b \ge 0$. Then,
\begin{equation} \label{numest2}
|\mu a - b|^p - |a - b|^p \le p b^{p - 1} a.
\end{equation}
\end{lemma}
\begin{proof}
We consider separately the three possibilities~$b \ge a$,~$\mu a \le b < a$ and~$b < \mu a$. In the first case,
$$
|\mu a - b|^p - |a - b|^p = (b - \mu a)^p - (b - a)^p = p \int_{b - a}^{b - \mu a} t^{p - 1} \, dt \le p b^{p - 1} a.
$$
On the other hand, if~$\mu a \le b < a$, then
\begin{align*}
|\mu a - b|^p - |a - b|^p & = (b - \mu a)^p - (a - b)^p = p \int_{a - b}^{b - \mu a} t^{p - 1} \, dt \\
& \le p (b - \mu a)^{p - 1} (2 b - (1 + \mu) a) \le p b^{p - 1} a.
\end{align*}
Finally, when~$b < \mu a$ the thesis is trivially verified, as the left-hand side of~\eqref{numest2} is negative.
\end{proof}

\begin{lemma} \label{numestlem3}
Let~$p > 1$ and~$a \ge b \ge 0$. Then,
$$
a^p - b^p \le \varepsilon a^p + \left( \frac{p - 1}{\varepsilon} \right)^{p - 1} (a - b)^p,
$$
for any~$\varepsilon > 0$.
\end{lemma}
\begin{proof}
We compute
$$
a^p - b^p = \left( b + (a - b) \right)^p - b^p = p \int_0^{a- b} (b + t)^{p - 1} \, dt \le p a^{p - 1} (a - b).
$$
For a fixed~$\delta > 0$, we use Young's inequality to deduce
$$
a^p - b^p \le p \left( \delta^{\frac{1}{p}} a \right)^{p - 1} \left( \frac{a - b}{\delta^{\frac{p - 1}{p}}} \right) \le (p - 1) \delta a^p + \delta^{1 - p} (a - b)^p,
$$
and the conclusion follows by taking~$\delta = \varepsilon / (p - 1)$.
\end{proof}

\begin{lemma} \label{numestlem4}
Let~$p > 1$,~$a \in \R$ and~$b \ge 0$. Then,
\begin{equation} \label{numest4}
(a - b)_+^{p - 1} \ge \min \{ 1, 2^{2 - p} \} a_+^{p - 1} - b^{p - 1}.
\end{equation}
\end{lemma}
\begin{proof}
We consider separately the three possibilities~$a \ge b$,~$0 \le a < b$ and~$a < 0$. If~$a \ge b$, it easy to see that
$$
(a - b)^{p - 1} + b^{p - 1} \ge \min \{ 1, 2^{2 - p} \} a^{p - 1},
$$
which is~\eqref{numest4}. If~$0 \le a < b$, then
$$
(a - b)_+^{p - 1} - \min \{ 1, 2^{2 - p} \} a_+^{p - 1} = - \min \{ 1, 2^{2 - p} \} a^{p - 1} \ge - b^{p - 1},
$$
and~\eqref{numest4} follows as well. In the case~$a < 0$, inequality~\eqref{numest4} is also trivially true.
\end{proof}

The next six results contain well-know functional inequalities in fractional Sobolev spaces. Our estimates are usually minor modifications of those available in the literature. We write them here in order to keep better track of the values of the constants involved and to have them ready for applications in the subsequent sections.

First is a weighted estimate related to the embeddings of fractional Sobolev spaces as the differentiability order varies.

\begin{lemma} \label{sobinclem0}
Let~$n \in \N$,~$p \ge 1$,~$0 < \sigma \le s < 1$ and~$R > 0$. Then, for any~$u \in W^{s, p}(B_R)$, it holds
$$
[u]_{W^{\sigma, p}(B_R)}^p \le \delta^{(s - \sigma) p} [u]_{W^{s, p}(B_R)}^p + \frac{2^p |B_1|}{\sigma p} \, \chi_{(0, 2 R)}(\delta) \delta^{- \sigma p} \| u \|_{L^p(B_R)}^p,
$$
for any~$\delta > 0$.
\end{lemma}
\begin{proof}
The result is a weighted version of, say,~\cite[Proposition~2.1]{DPV12} and the proof follows the same lines of that presented there. However, we report it here for the reader's convenience.

First of all, it is enough to deal with~$R = 1$, as a scaling argument readily shows. Fix~$\delta > 0$. On the one hand, we have
\begin{align*}
& \int_{B_1} \left[ \int_{B_1 \cap B_\delta(x)} \frac{|u(x) - u(y)|^p}{|x - y|^{n + \sigma p}} \, dy \right] dx \\
& \hspace{40pt} = \frac{1}{\delta^{n + \sigma p}} \int_{B_1} \left[ \int_{B_1 \cap B_\delta(x)} |u(x) - u(y)|^p \left( \frac{\delta}{|x - y|} \right)^{n + \sigma p} \, dy \right] dx \\
& \hspace{40pt} \le \frac{1}{\delta^{n + \sigma p}} \int_{B_1} \left[ \int_{B_1 \cap B_\delta(x)} |u(x) - u(y)|^p \left( \frac{\delta}{|x - y|} \right)^{n + s p} \, dy \right] dx \\
& \hspace{40pt} \le \delta^{(s - \sigma) p} \int_{B_1} \int_{B_1} \frac{|u(x) - u(y)|^p}{|x - y|^{n + s p}} \, dx dy.
\end{align*}
Of course, if~$\delta \ge 2$, we have already proved the claim. On the other hand, by Jensen's inequality
\begin{align*}
\int_{B_1} \left[ \int_{B_1 \setminus B_\delta(x)} \frac{|u(x) - u(y)|^p}{|x - y|^{n + \sigma p}} \, dy \right] dx & \le 2^p \int_{B_1} |u(x)|^p \left[ \int_{\R^n \setminus B_\delta(x)} \frac{dy}{|x - y|^{n + \sigma p}} \right] dx \\
& = \frac{2^p |B_1|}{\sigma p} \, \delta^{- \sigma p} \int_{B_1} |u(x)|^p \, dx.
\end{align*}
These two formulas lead to the desired estimate.
\end{proof}

In the following result we deal with Sobolev spaces having different orders of integrability and differentiability. Unlike in Lemma~\ref{sobinclem0}, this estimate involves only the Gagliardo seminorms of these spaces, and no Lebesgue norms. Moreover, the inequality is stated for more general quantities than the Gagliardo seminorms, allowing for slightly more freedom in the choice of the domains of integration. This small tweak is of some importance for a future application in Lemma~\ref{growthlem}.

\begin{lemma} \label{sobinclem}
Let~$n \in \N$,~$1 \le q < p$ and~$0 < \sigma < s < 1$. Let~$\Omega' \subseteq \Omega \subset \R^n$ be two bounded measurable sets. Then, for any~$u \in W^{s, p}(\Omega)$, it holds
$$
\left[ \int_{\Omega} \int_{\Omega'} \frac{|u(x) - u(y)|^q}{|x - y|^{n + \sigma q}} \, dx dy \right]^{\frac{1}{q}} \le C_0 |\Omega'|^{\frac{p - q}{p q}} \diam(\Omega)^{s - \sigma} \left[ \int_{\Omega} \int_{\Omega'} \frac{|u(x) - u(y)|^p}{|x - y|^{n + s p}} \, dx dy \right]^{\frac{1}{p}},
$$
with
$$
C_0 := \left[ \frac{n (p - q)}{(s - \sigma) p q} |B_1| \right]^{\frac{p - q}{p q}}.
$$
\end{lemma}
\begin{proof}
By H\"older's inequality, we have
\begin{align*}
& \int_{\Omega} \int_{\Omega'} \frac{|u(x) - u(y)|^q}{|x - y|^{n + \sigma q}} \, dx dy \\
& \hspace{50pt} = \int_{\Omega} \int_{\Omega'} \frac{|u(x) - u(y)|^q}{|x - y|^{\frac{q}{p} (n + s p)}} \frac{1}{|x - y|^{\frac{p - q}{p} n - (s - \sigma) q}} \, dx dy \\
& \hspace{50pt} \le \left( \int_{\Omega} \int_{\Omega'} \frac{|u(x) - u(y)|^p}{|x - y|^{n + s p}} \, dx dy \right)^{\frac{q}{p}} \left( \int_{\Omega} \int_{\Omega'} \frac{dx dy}{|x - y|^{n - \frac{(s - \sigma) p q}{p - q}}} \right)^{\frac{p - q}{p}}.
\end{align*}
Then, letting~$d := \diam(\Omega)$ and changing variables appropriately, we compute
\begin{align*}
\int_{\Omega} \int_{\Omega'} \frac{dx dy}{|x - y|^{n - \frac{(s - \sigma) p q}{p - q}}} \le \int_{\Omega'} \left( \int_{B_d} \frac{dz}{|z|^{n - \frac{(s - \sigma) p q}{p - q}}} \right) dx = \frac{n (p - q)}{(s - \sigma) p q} |B_1| |\Omega'| d^{\frac{(s - \sigma) p q}{p - q}},
\end{align*}
and the thesis follows.
\end{proof}

Next is a fractional Poincar\'e inequality for functions having fat zero level sets. The corresponding Poincar\'e-Wirtinger-type inequality for functions with vanishing integral mean is due to~\cite{BBM02,P04}. Notice that the dependence of the constant on the parameter~$s$ is explicit, at least when~$s$ is far from~$0$.

\begin{lemma} \label{poinine}
Let~$n \in \N$,~$p \ge 1$,~$0 < s_0 \le s < 1$ and~$R > 0$. Let~$u \in W^{s, p}(B_R)$ be such that~$u = 0$ a.e.~on a set~$\Omega_0 \subseteq B_R$, with~$|\Omega_0| \ge \gamma |B_R|$, for some~$\gamma \in (0, 1]$. Then,
$$
\int_{B_R} |u(x)|^p \, dx \le C_1 (1 - s) R^{s p} \int_{B_R} \int_{B_R} \frac{|u(x) - u(y)|^p}{|x - y|^{n + s p}} \, dx dy,
$$
for some constant~$C_1 \ge 1$ depending only on~$n$,~$s_0$,~$p$ and~$\gamma$.
\end{lemma}
\begin{proof}
First, we restrict ourselves to take~$R = 1$, as the general estimate follows then by scaling. Moreover, we may as well consider the unit cube~$Q_1 = (-1/2, 1/2)^n$ instead of the ball~$B_1$, by applying a suitable bi-Lipschitz diffeomorphism~$T: \overline{Q}_1 \to \overline{B}_1$. More precisely, if~$v := u \circ T$, then~$v \in W^{s, p}(Q_1)$, with
\begin{equation} \label{Cstar1}
\begin{aligned}
C_\star^{-1} \| u \|_{L^p(B_1)}^p & \le \| v \|_{L^p(Q_1)}^p \le C_\star \| u \|_{L^p(B_1)}^p, \\
C_\star^{-1} [ u ]_{W^{s, p}(B_1)}^p & \le [ v ]_{W^{s, p}(Q_1)}^p \le C_\star [ u ]_{W^{s, p}(B_1)}^p,
\end{aligned}
\end{equation}
and
\begin{equation} \label{Cstar2}
\left| \{ v = 0 \} \cap Q_1 \right| \ge \frac{\gamma}{C_\star} |Q_1|,
\end{equation}
for some dimensional constant~$C_\star \ge 1$.

Applying for instance~\cite[Corollary~2.1]{P04}, we know that there is a constant~$C_\sharp \ge 1$, depending only on~$n$,~$s_0$ and~$p$, such that
\begin{equation} \label{MSpoinine}
\| v - v_{Q_1} \|_{L^p(Q_1)}^p \le C_\sharp (1 - s) \, [v]_{W^{s, p}(Q_1)}^p,
\end{equation}
where
$$
v_{Q_1} := \dashint_{Q_1} v(x) \, dx.
$$
But, by~\eqref{Cstar2},
$$
\| v - v_{Q_1} \|_{L^p(Q_1)}^p \ge \left| \{ v = 0 \} \cap Q_1 \right| |v_Q|^p \ge \frac{\gamma}{C_\star} |v_{Q_1}|^p,
$$
and hence
$$
\| v \|_{L^p(Q_1)} \le \| v - v_{Q_1} \|_{L^p(Q_1)} + |v_{Q_1}| \le \left[ 1 + \left( \frac{C_\star}{\gamma} \right)^{1/p} \right] \| v - v_{Q_1} \|_{L^p(Q_1)}.
$$
This,~\eqref{MSpoinine} and~\eqref{Cstar1} yield the thesis.
\end{proof}

We stress that a Poincar\'e-type inequality of this kind can be obtained with a simpler and more direct computation in the spirit of formula~(4.2) in~\cite{M03}. However, this strategy does not seem to yield a constant with the needed dependence on~$s$.

We now have a couple of fractional Sobolev inequalities in balls. To deduce them, we report here below an analogous result by~\cite{BBM02,MS02}, set in the whole Euclidean space.

\begin{lemma}[{\cite[Theorem~1]{MS02}}] \label{sobinelem}
Let~$n \in \N$,~$p \ge 1$ and~$s \in (0, 1)$ be such that~$n > s p$. Then, for any~$W^{s, p}(\R^n)$, it holds
$$
\left( \int_{\R^n} |u(x)|^{p^*_s} \, dx \right)^{\frac{p}{p^*_s}} \le C_2 \frac{s (1 - s)}{(n - s p)^{p - 1}} \int_{\R^n} \int_{\R^n} \frac{|u(x) - u(y)|^p}{|x - y|^{n + s p}} \, dx dy,
$$
for some constant~$C_2 \ge 1$ depending only on~$n$ and~$p$.
\end{lemma}

We recall that~$p^*_s$ denotes the fractional Sobolev exponent defined in~\eqref{p*sdef}.

As a first corollary of Lemmas~\ref{poinine} and~\ref{sobinelem}, we deduce the following homogeneous fractional Sobolev inequality in a ball.

\begin{corollary} \label{nullsobcor}
Let~$n \in \N$,~$p \ge 1$ and~$0 < s_0 \le s < 1$ be such that~$n > s p$. Let~$u \in W^{s, p}_0(B_R)$ and suppose that~$u = 0$ on a set~$\Omega_0 \subseteq B_R$ with~$|\Omega_0| \ge \gamma |B_R|$, for some~$\gamma \in (0, 1]$. Then,
$$
\left( \int_{B_R} |u(x)|^{p^*_s} \, dx \right)^{\frac{p}{p^*_s}} \le C_3 \frac{1 - s}{(n - s p)^{p - 1}} \int_{B_R} \int_{B_R} \frac{|u(x) - u(y)|^p}{|x - y|^{n + s p}} \, dx dy,
$$
for some constant~$C_3 \ge 1$ depending only on~$n$,~$s_0$,~$p$ and~$\gamma$.
\end{corollary}
\begin{proof}
Of course, we can reduce to the case~$R = 1$, as the inequality is scaling invariant. In view of~\cite[Theorem~5.4]{DPV12}, we know that there exists a function~$\tilde{u} \in W^{s, p}(\R^n)$ such that~$\tilde{u}|_{B_1} = u$ and~$\| \tilde{u} \|_{W^{s, p}(\R^n)} \le C' \| u \|_{W^{s, p}(B_1)}$, for some constant~$C' \ge 1$. A careful inspection of the several estimates leading to the proof of this result shows that, actually, one has
\begin{equation} \label{extine}
[ \tilde{u} ]_{W^{s, p}(\R^n)}^p \le C \left( [u]_{W^{s, p}(B_1)}^p + \frac{\| u \|_{L^p(B_1)}^p}{s(1 - s)} \right),
\end{equation}
with~$C \ge 1$ depending only on~$n$ and~$p$. The conclusion of the corollary now follows by applying this with Lemmas~\ref{poinine} and~\ref{sobinelem}.
\end{proof}

When the zero level set of a function does not occupy a region as large as a fraction of the ball, but, instead, the function is supported well inside of that ball, we can still take advantage of Lemma~\ref{sobinelem} to get the following estimate.

\begin{corollary} \label{sobinecor}
Let~$n \in \N$,~$p \ge 1$ and~$s \in (0, 1)$ be such that~$n > s p$. Let~$u \in W^{s, p}_0(B_R)$ be such that~$\supp(u) \subseteq B_r$, with~$0 < r < R$. Then,
\begin{multline*}
\left( \int_{B_R} |u(x)|^{p^*_s} \, dx \right)^{\frac{p}{p^*_s}} \\
\le C_4 \frac{1 - s}{(n - s p)^{p - 1}} \left[ \int_{B_R} \int_{B_R} \frac{|u(x) - u(y)|^p}{|x - y|^{n + s p}} \, dx dy + \frac{1}{(R - r)^{s p}} \int_{B_r} |u(x)|^p dx \right],
\end{multline*}
for some constant~$C_4 \ge 1$ depending only on~$n$ and~$p$.
\end{corollary}
\begin{proof}
First, we apply Lemma~\ref{sobinelem} to obtain
\begin{align*}
\left( \int_{B_R} |u(x)|^{p^*_s} \, dx \right)^{\frac{p}{p^*_s}} & = \left( \int_{\R^n} |u(x)|^{p^*_s} \, dx \right)^{\frac{p}{p^*_s}}\\
& \le C_2 \frac{s(1 - s)}{(n - s p)^{p - 1}} \int_{\R^n} \int_{\R^n} \frac{|u(x) - u(y)|^p}{|x - y|^{n + s p}} \, dx dy.
\end{align*}
Now, we essentially use~\cite[Lemma~5.1]{DPV12} to deduce a bound for the right-hand side of the inequality above in terms of quantities integrated over the ball~$B_R$ alone. In fact, we redo the computation in order to keep track of the constants involved. By using that~$\supp(u) \subseteq B_r$ and changing variables appropriately, we compute
\begin{align*}
& \int_{\R^n} \int_{\R^n} \frac{|u(x) - u(y)|^p}{|x - y|^{n + s p}} \, dx dy \\
& \hspace{40pt} = \int_{B_R} \int_{B_R} \frac{|u(x) - u(y)|^p}{|x - y|^{n + s p}} \, dx dy + 2 \int_{B_r} |u(x)|^p \left[ \int_{\R^n \setminus B_R} \frac{dy}{|x - y|^{n + s p}} \right] dx \\
& \hspace{40pt} \le \int_{B_R} \int_{B_R} \frac{|u(x) - u(y)|^p}{|x - y|^{n + s p}} \, dx dy + \frac{2 n |B_1|}{s p} \frac{1}{(R - r)^{s p}} \int_{B_r} |u(x)|^p dx,
\end{align*}
from which the thesis follows.
\end{proof}

We conclude the section with an iteration lemma similar to e.g.~\cite[Lemma~1.1]{GG82}. For the reader's convenience, we provide its simple proof in full details.

\begin{lemma} \label{induclem}
Let~$0 < r < R$ and~$\Phi: [r, R] \to [0, +\infty)$ be a bounded function. Suppose that, for any~$r \le \rho < \tau \le R$, it holds
\begin{equation} \label{frho<ftau}
\Phi(\rho) \le \gamma \Phi(\tau) + A + \frac{B}{(\tau - \rho)^\alpha} + \frac{D}{(\tau - \rho)^\beta},
\end{equation}
for some constants~$A, B, D, \alpha, \beta > 0$ and~$\gamma \in (0, 1)$. Then, for any~$0 < r < R$,
\begin{equation} \label{induclemine}
\Phi(r) \le C \left[ A + \frac{B}{(R - r)^\alpha} + \frac{D}{(R - r)^\beta} \right],
\end{equation}
with~$C \ge 1$ depending only on~$\alpha$,~$\beta$ and~$\gamma$.
\end{lemma}
\begin{proof}
Take~$\theta \in (0, 1)$ to be later specified and consider the sequence~$\{ \rho_i \}$ of positive numbers inductively defined by
$$
\begin{cases}
\rho_i - \rho_{i - 1} = (1 - \theta) \theta^{i - 1} (R - r) & \quad \mbox{for } i \in \N \\
\rho_0 = r. &
\end{cases}
$$
This sequence is obviously increasing and~$\rho_i \rightarrow R$ as~$i \rightarrow +\infty$. We claim that
\begin{equation} \label{induclemclaim}
\begin{aligned}
\Phi(r) & \le \gamma^k \Phi(\rho_k) + A \frac{1 - \gamma^k}{1 - \gamma} + \frac{B}{(1 - \theta)^\alpha (R - r)^{\alpha}} \frac{1 - \left( \frac{\gamma}{\theta^\alpha} \right)^k}{1 - \frac{\gamma}{\theta^\alpha}} \\
& \quad + \frac{D}{(1 - \theta)^\beta (R - r)^\beta} \frac{1 - \left( \frac{\gamma}{\theta^\beta} \right)^k}{1 - \frac{\gamma}{\theta^\beta}},
\end{aligned}
\end{equation}
for any~$k \in \N \cup \{ 0 \}$.

We prove~\eqref{induclemclaim} by induction. Clearly, such inequality holds true for~$k = 0$. Thus, we take~$j \in \N$ and assume that~\eqref{induclem} is valid for~$k = j - 1$. With the aid of~\eqref{frho<ftau} we compute
\begin{align*}
\Phi(r) & \le \gamma^{j - 1} \left[ \gamma \Phi(\rho_j) + A + \frac{B}{(\rho_j - \rho_{j - 1})^\alpha} + \frac{D}{(\rho_j - \rho_{j - 1})^\beta} \right] \\
&\quad + A \frac{1 - \gamma^{j - 1}}{1 - \gamma} + \frac{B}{(1 - \theta)^\alpha (R - r)^{\alpha}} \frac{1 - \left( \frac{\gamma}{\theta^\alpha} \right)^{j - 1}}{1 - \frac{\gamma}{\theta^\alpha}} + \frac{D}{(1 - \theta)^\beta (R - r)^\beta} \frac{1 - \left( \frac{\gamma}{\theta^\beta} \right)^{j - 1}}{1 - \frac{\gamma}{\theta^\beta}} \\
& = \gamma^j \Phi(\rho_j) + A \left[ \frac{1 - \gamma^{j - 1}}{1 - \gamma} + \gamma^{j - 1} \right] + \frac{B}{(1 - \theta)^\alpha (R - r)^{\alpha}} \left[ \frac{1 - \left( \frac{\gamma}{\theta^\alpha} \right)^{j - 1}}{1 - \frac{\gamma}{\theta^\alpha}} + \left( \frac{\gamma}{\theta^\alpha} \right)^{j - 1} \right] \\
&  \quad + \frac{D}{(1 - \theta)^\beta (R - r)^{\beta}} \left[ \frac{1 - \left( \frac{\gamma}{\theta^\beta} \right)^{j - 1}}{1 - \frac{\gamma}{\theta^\beta}} + \left( \frac{\gamma}{\theta^\beta} \right)^{j - 1} \right] \\
& = \gamma^j \Phi(\rho_j) + A \frac{1 - \gamma^j}{1 - \gamma} + \frac{B}{(1 - \theta)^\alpha (R - r)^{\alpha}} \frac{1 - \left( \frac{\gamma}{\theta^\alpha} \right)^j}{1 - \frac{\gamma}{\theta^\alpha}} \\
& \quad + \frac{D}{(1 - \theta)^\beta (R - r)^{\beta}} \frac{1 - \left( \frac{\gamma}{\theta^\beta} \right)^j}{1 - \frac{\gamma}{\theta^\beta}}, 
\end{align*}
which is precisely~\eqref{induclemclaim} with~$k = j$. We can therefore conclude that~\eqref{induclemclaim} holds for any~$k \ge 0$. Taking the limit as~$k \rightarrow +\infty$ in~\eqref{induclemclaim}, we are finally led to~\eqref{induclemine}, provided we choose~$\theta$ in such a way that~$\gamma \theta^{-\alpha}$ and~$\gamma \theta^{-\beta}$ are both strictly smaller than~$1$.
\end{proof}

\section{A fractional~De~Giorgi isoperimetric-type inequality} \label{sDGsec}

In this section, we establish an isoperimetric-type inequality for the level sets of functions belonging to~$W^{s, p}$, when~$s$ is close to~$1$. This estimate will turn out to be crucial in the next section, where we use it to obtain~$C^\alpha$ estimates uniform in the parameter~$s$, as~$s \rightarrow 1^-$.

The statement of this result is as follows.

\begin{proposition} \label{sDGlemprop}
Let~$n \ge 2$ be an integer and~$p > 1$. Fix~$M > 0$ and~$\gamma_1, \gamma_2 \in (0, 1)$. Then, there exist two constants~$\bar{s} \in (0, 1)$ and~$C > 0$ such that, given any~$s \in [\bar{s}, 1)$ and~$R > 0$, it holds
\begin{equation} \label{sDGine}
\begin{aligned}
& (k - h) \Big[ |B_R \cap \{ u \le h \}| | B_R \cap \{ u \ge k \} | \Big]^{\frac{n - 1}{n}} \\
& \hspace{60pt} \le C R^{n - 2 + s} (1 - s)^{1 / p} [u]_{W^{s, p}(B_R)} |B_R \cap \{ h < u < k \}|^{\frac{p - 1}{p}},
\end{aligned}
\end{equation}
for any two real numbers~$h < k$ and any~$u \in W^{s, p}(B_R)$ satisfying
\begin{gather*}
\| u \|_{L^p(B_R)}^p + (1 - s) R^{s p} [u]_{W^{s, p}(B_R)}^p \le M R^{n} (k - h)^p, \\
|B_R \cap \{ u \le h \}| \ge \gamma_1 |B_R| \quad \mbox{and} \quad |B_R \cap \{ u \ge k \}| \ge \gamma_2 |B_R|.
\end{gather*}
The constant~$C$ depends only on~$n$ and~$p$, while~$\bar{s}$ may also depend on~$M$,~$\gamma_1$ and~$\gamma_2$.
\end{proposition}

Inequality~\eqref{sDGine} gives a bound from below for the measures of intermediate level sets of functions in fractional Sobolev spaces of differentiability order~$s$ close to~$1$. In particular, it provides a partial fractional counterpart to the classical result by~E.~De~Giorgi that states that functions in~$W^{1, p}$ cannot have jump discontinuities.

The constant~$\bar{s}$, which unfortunately is not explicitly determined, acts as a threshold, separating~\emph{well-behaved} from~\emph{ill-behaved} functions in the scale~$W^{s, p}$, as~$s \in (0, 1)$. As noted in the introduction, characteristic functions may belong to~$W^{s, p}$, if~$s p < 1$. This suggests that~$\bar{s} \ge 1 / p$.

Since~$\bar{s}$ may not depend solely on~$n$ and~$p$, we do not get a clear, global separation between~\emph{good} and~\emph{bad} Sobolev spaces. Instead, we identify a~\emph{transversal} region of, say,~$\cup_{s \in (0, 1)} W^{s, p}(B_1)$, composed by functions that satisfy~\eqref{sDGine} for various parameters~$M$,~$\gamma_1$ and~$\gamma_2$. It would be interesting to understand whether or not~$\bar{s}$ could be chosen independently from~$M$,~$\gamma_1$ or~$\gamma_2$. See the brief discussion following the proof of Lemma~\ref{growthlem} in Section~\ref{DGsec} for some comments on the implications of such possible uniformity on the regularity of functions in fractional~De~Giorgi classes.

\smallskip

We now focus on the proof of Proposition~\ref{sDGlemprop}. Our aim is to deduce it from the already mentioned isoperimetric-type inequality obtained by~E.~De~Giorgi in~\cite{DeG57}. We recall here below this classical result and provide a short proof of it. Our argument follows the strategy outlined in~\cite{Giu03}, which is essentially based on the Poincar\'e-Sobolev inequality.

\begin{lemma} \label{DGisolem}
Let~$n \ge 2$ be an integer and~$p > 1$. Then, for any two real numbers~$\ell < m$ and any~$u \in W^{1, p}(B_1)$, it holds
$$
\Big[ | B_1 \cap \{ u \le \ell \} | | B_1 \cap \{ u \ge m \} | \Big]^{\frac{n - 1}{n}} \le \frac{C_\bullet}{m - \ell} \, \| \nabla u \|_{L^p(B_1)} | B_1 \cap \{ \ell < u < m \} |^{\frac{p - 1}{p}},
$$
for some constant~$C_\bullet \ge 1$ depending only on~$n$ and~$p$.
\end{lemma}

\begin{proof}
Clearly, we can suppose that both sets~$B_1 \cap \{ u \le \ell \}$ and~$B_1 \cap \{ u \ge m \}$ have positive measure, otherwise there is nothing to prove. Define
$$
w(x) := \begin{cases}
m - \ell & \quad \mbox{if } u(x) \ge m, \\
u(x) - \ell & \quad \mbox{if } \ell < u(x) < m, \\
0 & \quad \mbox{if } u(x) \le \ell.
\end{cases}
$$
By applying Poincar\'e-Sobolev inequality (see e.g.~\cite[Theorem~3.16]{Giu03}) to this function, we get
\begin{align*}
(m - \ell) |B_1 \cap \{ u \ge m \}|^{\frac{n - 1}{n}} & = \left( \int_{B_1 \cap \{ u \ge m \}} w(x)^{\frac{n}{n - 1}} \, dx \right)^{\frac{n - 1}{n}} \le \| w \|_{L^{\frac{n}{n - 1}}(B_1)} \\
& \le \frac{C_\bullet}{|B_1 \cap \{ u \le \ell \}|^{\frac{n - 1}{n}}} \, \| \nabla w \|_{L^1(B_1)},
\end{align*}
for some~$C_\bullet \ge 1$ depending only on~$n$ and~$p$. Using then H\"older's inequality,
$$
\| \nabla w \|_{L^1(B_1)} = \int_{B_1 \cap \{ \ell < u < m \}} | \nabla u(x) | \, dx \le \| \nabla u \|_{L^p(B_1)} |B_1 \cap \{ \ell < u < m \}|^{\frac{p - 1}{p}}.
$$
The combination of these two estimates leads to the thesis.
\end{proof}

With this result at hand, we can deduce Proposition~\ref{sDGlemprop} via a contradiction argument. Here follow the details.

\begin{proof}[Proof of Proposition~\ref{sDGlemprop}]
First of all, we suppose that~$R = 1$, as the general case can then be obtained by scaling.

We claim that~\eqref{sDGine} holds true with
$$
C := \frac{4 C_\bullet}{D_*},
$$
where~$C_\bullet$ is as in Lemma~\ref{DGisolem} and~$D_*$ is defined by
\begin{equation} \label{Dstardef}
D_* := \left[ \frac{1}{p} \int_{S^{n - 1}} |e_1 \cdot \sigma|^p \, d\Haus^{n - 1}(\sigma) \right]^{\frac{1}{p}}.
\end{equation}
To prove this fact, we argue by contradiction. Observe that we can limit ourselves to deal with the case of~$h = 0$ and~$k > 0$, since the inequality is invariant under translations in the dependent variable~$u$. Therefore, we suppose that there exist three sequences~$\{ s_j \}_{j \in \N} \subset (0, 1)$,~$\{ k_j \}_{j \in \N} \subset (0, +\infty)$ and~$\{ u_j \}_{j \in \N} \subset L^p(B_1)$, such that
\begin{gather*}
\nonumber \lim_{j \rightarrow +\infty} s_j = 1, \\
u_j \in W^{s_j, p}(B_1), \, \, \mbox{with} \, \, \| u_j \|_{L^p(B_1)}^p + (1 - s_j) [u_j]_{W^{s_j, p}(B_1)}^p \le M k_j^p, \\
| B_1 \cap \{ u_j \le 0 \} | \ge \gamma_1 |B_1|, \, \, | B_1 \cap \{ u_j \ge k_j \} | \ge \gamma_2 |B_1|,
\end{gather*}
and
\begin{align*}
& k_j \Big[ | B_1 \cap \{ u_j \le 0 \} | | B_1 \cap \{ u_j \ge k_j \} | \Big]^{\frac{n - 1}{n}} \\
& \hspace{60pt} > 4 C_\bullet \, \frac{(1 - s_j)^{1 / p} [u_j]_{W^{s_j, p}(B_1)}}{D_*} \, | B_1 \cap \{ 0 < u_j < k_j \} |^{\frac{p - 1}{p}},
\end{align*}
for any~$j \in \N$.

We now normalize the~$u_j$'s over the sequence~$\{ k_j \}$, i.e.~we define~$v_j := u_j / k_j$. Observe that~$v_j \in W^{s_j, p}(B_1)$ satisfies
\begin{gather}
\label{vjleM} \| v_j \|_{L^p(B_1)}^p + (1 - s_j) [v_j]_{W^{s_j, p}(B_1)}^p \le M, \\
\label{vjlev} | B_1 \cap \{ v_j \le 0 \} | \ge \gamma_1 |B_1|, \, \, | B_1 \cap \{ v_j \ge 1 \} | \ge \gamma_2 |B_1|,
\end{gather}
and
\begin{equation} \label{sDGineM2}
\begin{aligned}
& \Big[ | B_1 \cap \{ v_j \le 0 \} | | B_1 \cap \{ v_j \ge 1 \} | \Big]^{\frac{n - 1}{n}} \\
& \hspace{60pt} > 4 C_\bullet \, \frac{(1 - s_j)^{1 / p} [v_j]_{W^{s_j, p}(B_1)}}{D_*} \, | B_1 \cap \{ 0 < v_j < 1 \} |^{\frac{p - 1}{p}},
\end{aligned}
\end{equation}
for any~$j \in \N$.

Thanks to~\eqref{vjleM}, we may apply~\cite[Corollary~7]{BBM01} or~\cite[Theorem~1.2]{P04} and deduce that, up to subsequences,~$v_j$ converges in~$L^p(B_1)$ to some function~$v_\infty \in W^{1, p}(B_1)$, as~$j \rightarrow +\infty$. Up to extracting a further subsequence, we may suppose that~$v_j \rightarrow v_\infty$ a.e.~in~$B_1$ and that~$(1 - s_j)^{1/p} [v_j]_{W^{s_j, p}(B_1)}$ has a limit as~$j \rightarrow +\infty$, which, by~\cite[Theorem~1.2]{P04}, necessarily satisfies
\begin{equation} \label{nablalesemi}
\lim_{j \rightarrow +\infty} (1 - s_j)^{1 / p} [v_j]_{W^{s_j, p}(B_1)} \ge D_* \| \nabla v_\infty \|_{L^p(B_1)},
\end{equation}
with~$D_*$ as in~\eqref{Dstardef}.

Since~$v_\infty \in W^{1, p}(B_1)$, we know that~$|B_1 \cap \{ v_\infty = t \}| = 0$, for a.a.~$t \in \R$. To see this, one could apply e.g.~\cite[Theorem~1.1]{MSZ03} and conclude that almost all level sets of~(a specific representative of)~$v_\infty$ are countable~$(n - 1)$-rectifiable sets, and thus have zero Lebesgue measure. Choose now~$\varepsilon \in (0, 1/4]$ in a way that
\begin{equation} \label{uinftyzerolev}
|B_1 \cap \{ v_\infty = \varepsilon \}| = |B_1 \cap \{ v_\infty = 1 - \varepsilon \}| = 0.
\end{equation}
It is not hard to see that
\begin{equation} \label{limlevsets}
\begin{aligned}
& \lim_{j \rightarrow +\infty} | B_1 \cap \{ v_j < \varepsilon \} | = |B_1 \cap \{ v_\infty < \varepsilon \}|,\\
& \lim_{j \rightarrow +\infty} | B_1 \cap \{ v_j > 1 - \varepsilon \} | = |B_1 \cap \{ v_\infty > 1 - \varepsilon \}|,\\
& \lim_{j \rightarrow +\infty} | B_1 \cap \{ \varepsilon < v_j < 1 - \varepsilon \} | = |B_1 \cap \{ \varepsilon < v_\infty < 1 - \varepsilon \}|.
\end{aligned}
\end{equation}
We prove for instance the validity of the first limit in~\eqref{limlevsets}. Notice that
\begin{equation} \label{chiujtouinfty}
\lim_{j \rightarrow +\infty} \chi_{\{ v_j < \varepsilon \}} = \chi_{\{ v_\infty < \varepsilon \}} \quad \mbox{a.e.~in } B_1.
\end{equation}
Indeed, for a.a.~$x \in \{ v_\infty < \varepsilon \}$, we have that~$v_j(x) < \varepsilon$ for all but a finite number of~$j$'s, since~$v_j \rightarrow v_\infty$ a.e.~in~$B_1$, as~$j \rightarrow +\infty$. Therefore,
$$
\lim_{j \rightarrow +\infty} \chi_{\{ v_j < \varepsilon \}}(x) = 1 = \chi_{\{ v_\infty < \varepsilon \}}(x) \quad \mbox{for a.a.~} x \mbox{ in } B_1 \cap \{ v_\infty < \varepsilon \}.
$$
Similarly,
$$
\lim_{j \rightarrow +\infty} \chi_{\{ v_j < \varepsilon \}}(x) = 0 = \chi_{\{ v_\infty < \varepsilon \}}(x) \quad \mbox{for a.a.~} x \mbox{ in } B_1 \cap \{ v_\infty > \varepsilon \},
$$
and~\eqref{chiujtouinfty} follows, thanks to~\eqref{uinftyzerolev}. By~\eqref{chiujtouinfty}, we may apply Lebesgue's dominated convergence theorem to obtain that
$$
\lim_{j \rightarrow +\infty} \int_{B_1} \chi_{\{ v_j < \varepsilon \}}(x) \, dx = \int_{B_1} \chi_{\{ v_\infty < \varepsilon \}}(x) \, dx.
$$
This gives the first formula in~\eqref{limlevsets}. Analogously, one gets the other two. In particular, we deduce from~\eqref{vjlev} and the first two limits in~\eqref{limlevsets} that
\begin{equation} \label{uinftylev}
|B_1 \cap \{ v_\infty < \varepsilon \}| \ge \gamma_1 |B_1| \quad \mbox{and} \quad |B_1 \cap \{ v_\infty > 1 - \varepsilon \}| \ge \gamma_2 |B_1|.
\end{equation}

In view of~\eqref{nablalesemi} and~\eqref{limlevsets}, by letting~$j \rightarrow +\infty$ in~\eqref{sDGineM2} we immediately see that
\begin{align*}
& \Big[ | B_1 \cap \{ v_\infty < \varepsilon \} | | B_1 \cap \{ v_\infty > 1 - \varepsilon \} | \Big]^{\frac{n - 1}{n}} \\
& \hspace{60pt} \ge 4 C_\bullet \| \nabla v_\infty \|_{L^p(B_1)} | B_1 \cap \{ \varepsilon < v_\infty < 1 - \varepsilon \} |^{\frac{p - 1}{p}}.
\end{align*}
By comparing this with the inequality of Lemma~\ref{DGisolem} and taking advantage of the fact that, by definition of~$\varepsilon$, it holds~$2 (1 - 2 \varepsilon ) \ge 1$, we finally obtain that
\begin{align*}
& \Big[ | B_1 \cap \{ v_\infty < \varepsilon \} | | B_1 \cap \{ v_\infty > 1 - \varepsilon \} | \Big]^{\frac{n - 1}{n}} \\
& \hspace{50pt} \ge 4 (1 - 2 \varepsilon) \Big[ | B_1 \cap \{ v_\infty \le \varepsilon \} | | B_1 \cap \{ v_\infty \ge 1 - \varepsilon \} | \Big]^{\frac{n - 1}{n}} \\
& \hspace{50pt} \ge 2 \Big[ | B_1 \cap \{ v_\infty < \varepsilon \} | | B_1 \cap \{ v_\infty > 1 - \varepsilon \} | \Big]^{\frac{n - 1}{n}}.
\end{align*}
This is clearly a contradiction, since both sides are positive, by~\eqref{uinftylev}. We therefore conclude that~\eqref{sDGine} holds true and the proposition is proved.
\end{proof}

\section{Fractional De~Giorgi classes} \label{DGsec}

In this section we introduce the notion of fractional De~Giorgi classes that we take into consideration and prove that their elements are bounded, H\"older continuous functions. On top of this, we show that where they are non-negative, they also satisfy a nonlocal version of the Harnack inequality.

\subsection{Definition and first properties}

In this first subsection, we give our definition of fractional~De~Giorgi classes and point out some elementary features of them. These classes are composed by functions that satisfy an improved Caccioppoli-type inequality, such as formula~\eqref{2scaccintro} in the introduction. In fact, we consider here a broader family of inequalities, that depend on a number of parameters.

\smallskip

Let~$n \in \N$,~$s \in (0, 1)$ and~$p > 1$. Let~$\Omega$ be an open subset of~$\R^n$.

Also fix~$d \ge 0$,~$H \ge 1$,~$k_0 \in \R$,~$\varepsilon \in (0, sp / n]$,~$\lambda \ge 0$ and~$R_0 \in (0, +\infty]$.

Let~$u \in L_s^{p - 1}(\R^n) \cap W^{s, p}(\Omega)$ be a given function. We say that~$u$ belongs to the~\emph{fractional~De~Giorgi class}~$\DG_+^{s, p}(\Omega; d, H, k_0, \varepsilon, \lambda, R_0)$ if and only if it holds
\begin{equation} \label{DG+def}
\begin{aligned}
& [(u - k)_+]_{W^{s, p}(B_{r}(x_0))}^p + \int_{B_{r}(x_0)} (u(x) - k)_+ \left[ \int_{B_{2 R_0}(x)} \frac{ (u(y) - k)_-^{p - 1}}{|x - y|^{n + s p}} \, dy \right] dx \\
& \hspace{3pt} \le \frac{H}{1 - s} \Bigg[ \left( R^\lambda d^p + \frac{|k|^p}{R^{n \varepsilon}} \right) |A^+(k, x_0, R)|^{1 - \frac{sp}{n} + \varepsilon} + \frac{R^{(1 - s) p}}{(R - r)^p} \| (u - k)_+ \|_{L^p(B_R(x_0))}^p \\
& \hspace{3pt} \quad + \frac{R^{n + s p}}{(R - r)^{n + s p}} \| (u - k)_+ \|_{L^1(B_R(x_0))} \overline{\Tail}((u - k)_+; x_0, r)^{p - 1} \Bigg],
\end{aligned}
\end{equation}
for any point~$x_0 \in \Omega$, radii~$0 < r < R < \min \{ R_0, \dist \left( x_0, \partial \Omega \right) \}$ and~$k \ge k_0$. Recall that the set~$A^+(k, x_0, R)$ has been defined right below~\eqref{A+-def}.

Analogously,~$u \in \DG_-^{s, p}(\Omega; d, H, k_0, \varepsilon, \lambda, R_0)$ if and only if
\begin{equation} \label{DG-def}
\begin{aligned}
& [(u - k)_-]_{W^{s, p}(B_{r}(x_0))}^p + \int_{B_{r}(x_0)} (u(x) - k)_- \left[ \int_{B_{2 R_0}(x)} \frac{ (u(y) - k)_+^{p - 1}}{|x - y|^{n + s p}} \, dy \right] dx \\
& \hspace{3pt} \le \frac{H}{1 - s} \Bigg[ \left( R^\lambda d^p + \frac{|k|^p}{R^{n \varepsilon}} \right) |A^-(k, x_0, R)|^{1 - \frac{sp}{n} + \varepsilon} + \frac{R^{(1 - s) p}}{(R - r)^p} \| (u - k)_- \|_{L^p(B_R(x_0))}^p \\
& \hspace{3pt} \quad + \frac{R^{n + s p}}{(R - r)^{n + s p}} \| (u - k)_- \|_{L^1(B_R(x_0))} \overline{\Tail}((u - k)_-; x_0, r)^{p - 1} \Bigg],
\end{aligned}
\end{equation}
for any~$x_0 \in \Omega$,~$0 < r < R < \min \{ R_0, \dist \left( x_0, \partial \Omega \right) \}$ and~$k \le - k_0$.

Finally, we set
$$
\DG^{s, p}(\Omega; d, H, k_0, \varepsilon, \lambda, R_0) := \DG_+^{s, p}(\Omega; d, H, k_0, \varepsilon, \lambda, R_0) \cap \DG_-^{s, p}(\Omega; d, H, k_0, \varepsilon, \lambda, R_0).
$$

With a slight abuse of notation, we denote by~$\DG_+^{s, p}(\Omega; d, H, -\infty, \varepsilon, \lambda, R_0)$ the class of functions that satisfy~\eqref{DG+def} for any~$k \in \R$, and similarly for the spaces~$\DG_-^{s, p}$ and~$\DG^{s, p}$.

\smallskip

Notice that, with the above definitions, we have
$$
u \in \DG_+^{s, p}(\Omega; d, H, k_0, \varepsilon, \lambda, R_0) \quad \mbox{iff} \quad - u \in \DG_-^{s, p}(\Omega; d, H, k_0, \varepsilon, \lambda, R_0),
$$
Furthermore, it is not hard to see that the following scaling properties hold true:
\begin{align*}
u \in \DG_+^{s, p}(\Omega; d, H, k_0, \varepsilon, \lambda, R_0) \quad & \mbox{iff} \quad u_{z, \rho} \in \DG_+^{s, p} \left( \rho \, \Omega + z; \rho^{- \frac{\lambda + n \varepsilon}{p}} d, H, k_0, \varepsilon, \lambda, \rho R_0 \right),\\
u \in \DG_+^{s, p}(\Omega; d, H, k_0, \varepsilon, \lambda, R_0) \quad & \mbox{iff} \quad u^{(\mu)} \in \DG_+^{s, p}(\Omega; \mu d, H, \mu k_0, \varepsilon, \lambda, R_0),
\end{align*}
where, for any~$z \in \R^n$,~$\rho, \mu > 0$, we set
$$
u_{z, \rho}(x) := u \left(\frac{x - z}{\rho}\right), \quad u^{(\mu)}(x) := \mu u(x),
$$
and, as customary, we wrote~$\rho \, \Omega + z = \{ \rho x + z : x \in \Omega \}$. Analogous statements clearly hold for the spaces~$\DG_-^{s, p}$ and~$\DG^{s, p}$.

\smallskip

We now proceed to inspect the regularity properties of the elements of the just defined classes.

\subsection{Local boundedness}

We prove that the elements of fractional~De~Giorgi classes are locally bounded functions. Observe that here we only consider choices of parameters~$n$,~$s$,~$p$ that satisfy the condition~$n \ge sp$. Indeed, this is not at all a strong limitation, as when~$n < s p$ the boundedness and the H\"older continuity of the functions in~$\DG^{s, p}$---and, more generally, in~$W^{s, p}$---is warranted by the fractional Morrey embedding (see e.g.~\cite{DPV12}).

\smallskip

We begin with the following proposition, that establishes interior upper bounds for~$u \in \DG^{s, p}_+$.

\begin{proposition} \label{ulocboundprop}
Let~$u \in \DG_+^{s, p}(\Omega; d, H, k_0, \varepsilon, \lambda, R_0)$, with~$n \ge s p$,~$k_0 \ge 0$ and~$0 < \varepsilon_0 \le \varepsilon \le s p / n$. Then, there exist~$C \ge 1$ and~$\theta \in (0, \varepsilon_0 / 2]$, such that, for any~$x_0 \in \Omega$ and~$0 < R < \min \{ \dist \left( x_0, \partial \Omega \right), R_0 \} / 2$, it holds
\begin{equation} \label{ulocbound}
\begin{aligned}
\sup_{B_R(x_0)} \, (u - k_0)_+ & \le C \, \frac{\delta^{- \frac{p - 1}{(\varepsilon - \theta) p}}}{(n - s p + n \theta)^{\frac{p - 1}{(\varepsilon - \theta) p}}} \left( \dashint_{B_{2 R}(x_0)} (u(x) - k_0)_+^p \, dx \right)^{\frac{1}{p}} \\
& \quad + \delta \Tail((u - k_0)_+; x_0, R) + \delta^{\frac{p - 1}{p}} \left( R^{\frac{\lambda + n \varepsilon}{p}} d + k_0 \right),
\end{aligned}
\end{equation}
for any~$\delta \in (0, 1]$. The constant~$\theta$ depends on~$n$,~$p$ and~$\varepsilon_0$, while~$C$ also on~$H$. When~$n > s p$, we can even take~$\theta = 0$.
\end{proposition}
\begin{proof}
Suppose without loss of generality that~$x_0 = 0$.

Let~$R \le \rho < \tau \le 2 R$ and consider a cut-off function~$\eta \in C^\infty_0(\R^n)$ such that~$0 \le \eta \le 1$ in~$\R^n$,~$\supp(\eta) \subseteq B_{(\tau + 3 \rho) / 4}$,~$\eta = 1$ in~$B_\rho$ and~$|\nabla \eta| \le 8 / (\tau - \rho)$ in~$\R^n$. Fix~$k \ge k_0 \ge 0$ and set~$w_k := (u - k)_+$,~$v := \eta w_k$. Notice that~$\supp(v) \subseteq B_{(\tau + 3 \rho) / 4}$. Let~$\sigma \in [s - n \varepsilon_0 / (2 p), s]$ be given by
$$
\sigma := \begin{cases}
s & \quad \mbox{if } n > s p, \\
\max \left\{ 2 s - 1, s - \frac{n \varepsilon_0}{2 p}  \right\} & \quad \mbox{if } n = s p.
\end{cases}
$$
Observe that, with this choice,~$n > \sigma p$. Also,~$1 - \sigma \le 2 (1 - s)$. By H\"older's inequality and Corollary~\ref{sobinecor}, we have
\begin{equation} \label{locboundtech1}
\begin{aligned}
\| w_k \|_{L^p(B_\rho)}^p & \le |A^+(k, \rho)|^{\frac{\sigma p}{n}} \| v \|_{L^{p^*_\sigma}(B_{(\tau + \rho) / 2})}^p \\
& \le \frac{C (1 - \sigma)}{(n - \sigma p)^{p - 1}} |A^+(k, \rho)|^{\frac{\sigma p}{n}} \left[ [ v ]_{W^{\sigma, p}(B_{(\tau + \rho) / 2})}^p + \frac{\| v \|_{L^p(B_{(\tau + \rho) / 2})}^p}{(r - \rho)^{\sigma p}} \right] \\
& \le \frac{C (1 - s)}{(n - \sigma p)^{p - 1}} \frac{|A^+(k, \rho)|^{\frac{\sigma p}{n}}}{(\tau - \rho)^{\sigma p}} \left[ (\tau - \rho)^{s p} [ v ]_{W^{s, p}(B_{(\tau + \rho) / 2})}^p + \| w_k \|_{L^p(B_\tau)}^p \right],
\end{aligned}
\end{equation}
with~$C \ge 1$ depending only on~$n$,~$p$ and~$\varepsilon_0$. Notice that, when~$n = s p$ we also took advantage of Lemma~\ref{sobinclem0} (with~$\delta = \tau - \rho$), to deduce the last inequality. Using then Young's inequality and the definition of~$\eta$, we compute
\begin{align*}
[ v ]_{W^{s, p}(B_{(\tau + \rho) / 2})}^p & \le C \left[ [w_k]_{W^{s, p}(B_{(\tau + \rho) / 2})}^p + \int_{B_\tau} w_k(x)^p \left[ \int_{B_\tau} \frac{|\eta(x) - \eta(y)|^p}{|x - y|^{n + s p}} \, dy \right] dx \right] \\
& \le C \left[ [w_k]_{W^{s, p}(B_{(\tau + \rho) / 2})}^p + \frac{1}{(\tau - \rho)^p} \int_{B_\tau} w_k(x)^p \left[ \int_{B_\tau} \frac{dy}{|x - y|^{n - p + s p}} \right] dx \right] \\
& \le C \left[ [w_k]_{W^{s, p}(B_{(\tau + \rho) / 2})}^p + \frac{1}{1 - s} \frac{\tau^{(1 - s) p}}{(\tau - \rho)^p} \| w_k \|_{L^p(B_\tau)}^p \right].
\end{align*}
By combining this with~\eqref{locboundtech1} and~\eqref{DG+def}, we are led to the estimate
\begin{equation} \label{locboundtech2}
\begin{aligned}
\| w_k \|_{L^p(B_\rho)}^p & \le \frac{C}{(n - \sigma p)^{p - 1}}  |A^+(k, \rho)|^{\frac{\sigma p}{n}} \Bigg[ \left( \tau^\lambda d^p + \frac{k^p}{\tau^{n \varepsilon}} \right) \tau^{(s - \sigma) p} |A^+(k, \tau)|^{1 - \frac{sp}{n} + \varepsilon} \\
& \quad + \frac{\tau^{(1 - \sigma) p}}{(\tau - \rho)^p} \| w_k \|_{L^p(B_\tau)}^p + \frac{\tau^{n + s p}}{(\tau - \rho)^{n + \sigma p}} \| w_k \|_{L^1(B_\tau)} \overline{\Tail}(w_k; R)^{p - 1} \Bigg],
\end{aligned}
\end{equation}
where~$C$ now depends on~$H$ too.

Fix~$0 < h < k$. For~$x \in A^+(k) \subseteq A^+(h)$, we have
$$
w_h(x) = u(x) - h \ge k - h,
$$
and
$$
w_h(x) = u(x) - h \ge u(x) - k = w_k(x).
$$
Accordingly, given any~$r > 0$,
\begin{align*}
\| w_h \|_{L^p(B_r)}^p & \ge \int_{A^+(k, r)} w_h(x)^p \, dx \ge (k - h)^p |A^+(k, r)|,\\
\| w_h \|_{L^p(B_r)}^p & \ge \int_{A^+(k, r)} w_k(x)^p \, dx = \| w_k \|_{L^p(B_r)}^p,\\
\| w_h \|_{L^p(B_r)}^p & \ge (k - h)^{p-1} \int_{A^+(k, r)} w_k(x) \, dx = (k - h)^{p-1} \| w_k \|_{L^1(B_r)}.
\end{align*}
That is,
$$
|A^+(k, r)| \le \frac{\| w_h \|_{L^p(B_r)}^p}{(k - h)^p}, \, \, \| w_k \|_{L^p(B_r)}^p \le \| w_h \|_{L^p(B_r)}^p \, \mbox{ and } \,
\| w_k \|_{L^1(B_r)} \le \frac{\| w_h \|_{L^p(B_r)}^p}{(k - h)^{p - 1}}.
$$
With the aid of these estimates, inequality~\eqref{locboundtech2} yields
\begin{align*}
\| w_k \|_{L^p(B_\rho)}^p & \le \frac{C}{(n - \sigma p)^{p - 1}} \Bigg[
\frac{\tau^{\lambda + n \varepsilon} d^p + k^p}{(k - h)^p} \left( \frac{|A^+(k, \tau)|}{|B_\tau|} \right)^{\varepsilon - \frac{(s - \sigma)p}{n}} \\
& \quad + \left( \frac{|A^+(k, \tau)|}{|B_\tau|} \right)^{\frac{\sigma p}{n}} \left( \frac{\tau^p}{(\tau - \rho)^p} + \frac{\tau^{n + (s + \sigma) p} \, \overline{\Tail}(w_k; R)^{p - 1}}{(\tau - \rho)^{n + \sigma p} (k - h)^{p - 1}} \right) \Bigg] \| w_h \|_{L^p(B_\tau)}^p \\
& \le \frac{C}{(n - \sigma p)^{p - 1}} \frac{\tau^{- n \varepsilon_\sigma}}{(k - h)^{\varepsilon_\sigma p}} \Bigg[
\frac{\tau^{\lambda + n \varepsilon} d^p + k^p}{(k - h)^p} + \frac{\tau^{p}}{(\tau - \rho)^p} \\
& \quad + \frac{\tau^{n + (s + \sigma) p} \, \overline{\Tail}(w_k; R)^{p - 1}}{(\tau - \rho)^{n + \sigma p} (k - h)^{p - 1}} \Bigg] \| w_h \|_{L^p(B_\tau)}^{(1 + \varepsilon_\sigma) p},
\end{align*}
with~$\varepsilon_\sigma := \varepsilon - (s - \sigma) p / n \le \sigma p / n$.
Setting
$$
\varphi(\ell, \sigma) := \| w_\ell \|_{L^p(B_\sigma)}^p, \quad \mbox{for any } \ell, \, \sigma > 0,
$$
we get
\begin{equation} \label{locboundtech5}
\begin{aligned}
\varphi(k, \rho) & \le \frac{C}{(n - \sigma p)^{p - 1}} \frac{\tau^{- n \varepsilon_\sigma}}{(k - h)^{\varepsilon_\sigma p}} \left[
\frac{\tau^{\lambda + n \varepsilon} d^p + k^p}{(k - h)^p} + \frac{\tau^{p}}{(\tau - \rho)^p} \right. \\
& \quad \left. + \frac{\tau^{n + (s + \sigma) p} \, \overline{\Tail}(w_k; R)^{p - 1}}{(\tau - \rho)^{n + \sigma p} (k - h)^{p - 1}} \right] \varphi(h, \tau)^{1 + \varepsilon_\sigma}.
\end{aligned}
\end{equation}

Consider now the two sequences of positive numbers~$\{ k_i \}$ and~$\{ \rho_i \}$, defined by
$$
k_i := k_0 + M (1 - 2^{-i}) \quad \mbox{and} \quad \rho_i := (1 + 2^{-i}) R,
$$
for~$i \in \N \cup \{ 0 \}$ and for some~$M > 0$ to be determined. Note that~$\{ k_i \}$ is increasing, while~$\{ \rho_i \}$ is decreasing. Also set~$\varphi_i := \varphi(k_i, \rho_i)$. Recalling definitions~\eqref{Tailudef} and~\eqref{nsiTailudef},
\begin{align*}
\overline{\Tail}(w_{k_{i + 1}}; R)^{p - 1} & = (1 - s) \int_{A(k_{i + 1}) \setminus B_R} \frac{(u(x) - k_{i + 1})^{p - 1}}{|x|^{n + s p}} \, dx \\
& \le (1 - s) \int_{A(k_0) \setminus B_R} \frac{(u(x) - k_0)^{p - 1}}{|x|^{n + s p}} \, dx = R^{- s p} \Tail(w_{k_0}; R)^{p - 1}.
\end{align*}
By this and~\eqref{locboundtech5}, we compute
\begin{align*}
\varphi_{i + 1} & \le \frac{C}{(n - \sigma p)^{p - 1}} \frac{\rho_i^{- n \varepsilon_\sigma}}{(k_{i + 1} - k_i)^{\varepsilon_\sigma p}} \left[ \frac{\rho_i^{\lambda + n \varepsilon} d^p + k_{i + 1}^p}{(k_{i + 1} - k_i)^p} + \frac{\rho_i^{p}}{(\rho_i - \rho_{i + 1})^p} \right. \\
& \quad \left. + \frac{\rho_i^{n + (s + \sigma) p} \, \overline{\Tail}(w_{k_{i + 1}}; R)^{p - 1}}{(\rho_i - \rho_{i + 1})^{n + \sigma p} (k_{i + 1} - k_i)^{p - 1}} \right] \varphi_i^{1 + \varepsilon_\sigma} \\
& \le \frac{C}{(n - \sigma p)^{p - 1}} \frac{2^{(n + 3 p) i}}{R^{n \varepsilon_\sigma} M^{\varepsilon_\sigma p}} \left[ \frac{ R^{\lambda + n \varepsilon} d^p + k_0^p + M^p }{M^p} + 1 + \frac{\Tail(w_{k_0}; R)^{p - 1}}{M^{p - 1}} \right] \varphi_i^{1 + \varepsilon_\sigma} \\
& \le \frac{C \delta^{- 1 + p}}{(n - \sigma p)^{p - 1}} \frac{2^{(n + 3 p) i}}{R^{n \varepsilon_\sigma} M^{\varepsilon_\sigma p}} \, \varphi_i^{1 + \varepsilon_\sigma},
\end{align*}
if we choose
$$
M \ge M_1 := \delta \Tail(w_{k_0}; R) + \delta^{\frac{p - 1}{p}} \left( R^{\frac{\lambda + n \varepsilon}{p}} d + k_0 \right),
$$
for any fixed~$\delta \in (0, 1]$. By applying for instance~\cite[Lemma~7.1]{Giu03}, we infer that~$\varphi_i \rightarrow 0$ as~$i \rightarrow +\infty$, and hence~$u \le k_0 + M$ in~$B_R$, provided it holds
$$
\varphi_0 \le C \delta^{\frac{p - 1}{\varepsilon_\sigma}} (n - \sigma p)^{\frac{p - 1}{\varepsilon_\sigma}} R^{n} M^p,
$$
that is
$$
M \ge M_2 := \frac{C \delta^{- \frac{p - 1}{\varepsilon_\sigma p}}}{(n - \sigma p)^{\frac{p - 1}{\varepsilon_\sigma p}}} \left( \frac{1}{R^n} \int_{B_{2 R}} w_{k_0}(x)^p \, dx \right)^{\frac{1}{p}}.
$$
Estimate~\eqref{ulocbound} then follows by taking e.g.~$M := M_1 + M_2$.
\end{proof}

By applying Proposition~\ref{ulocboundprop} to both~$u$ and~$-u$, we get the desired two-sided boundedness result.

\begin{theorem}[\bfseries Local boundedness of fractional De~Giorgi functions] \label{DGboundthm}
\textcolor{white}{}\\
Let~$u \in \DG^{s, p}(\Omega; d, H, k_0, \varepsilon, \lambda, R_0)$, with~$n \ge s p$,~$k_0 \ge 0$ and~$0 < \varepsilon_0 \le \varepsilon \le s p / n$. Then,~$u \in L^\infty_\loc(\Omega)$. In particular, there exists a constant~$C \ge 1$, such that, for any~$x_0 \in \Omega$ and~$0 < R < \min \{ \dist(x_0, \partial \Omega), R_0 \} / 2$,
$$
\| u \|_{L^\infty(B_R(x_0))} \le C R^{-\frac{n}{p}} \| u \|_{L^p(B_{2 R}(x_0))} + \Tail(u; x_0, R) + d R^{\frac{\lambda + n \varepsilon}{p}} + 2 k_0,
$$
The constant~$C$ depends on~$n$,~$s$,~$p$,~$\varepsilon_0$ and~$H$. When~$n > p$, it does not blow up as~$s \rightarrow 1^-$.
\end{theorem}

We remark that up to now we have not fully exploited inequalities~\eqref{DG+def}-\eqref{DG-def}, that define the fractional~De~Giorgi classes. Indeed, to obtain the previous results we only took advantage of the bounds that those inequalities provide for the Gagliardo seminorm of the truncations~$(u - k)_\pm$, i.e.~the first term appearing on the left-hand sides of~\eqref{DG+def} and~\eqref{DG-def}.

In the next subsection, on the contrary, we will make full use of such defining inequalities.

\subsection{H\"older continuity}

We focus here on establishing H\"older continuity estimates for functions belonging to fractional~De~Giorgi classes. As before, we might well restrict ourselves to the case~$n \ge s p$. However, since we will need some of the results obtained here in the following subsection, we do not make such assumption.

\smallskip

The fundamental step in recovering the H\"older regularity is made in the following~\emph{growth lemma}.

\begin{lemma} \label{growthlem}
Let~$u \in \DG_-^{s, p}(B_{4 R}; d, H, -1, \varepsilon, \lambda, R_0)$, for some values~$R > 0$,~$R_0 \ge 4 R$ and~$0 < \varepsilon_0 \le \varepsilon \le s p / n$. Suppose that
\begin{equation} \label{uge0}
u \ge 0 \quad \mbox{in } B_{4 R},
\end{equation}
and
\begin{equation} \label{1dens}
|B_{2 R} \cap \{ u \ge 1 \}| \ge \gamma |B_{2 R}|,
\end{equation}
for some~$\gamma \in (0, 1)$. There exist a small constant~$\delta \in (0, 1/8]$, such that, if
\begin{equation} \label{udecay}
R^{\frac{\lambda + n \varepsilon}{p}} d + \Tail(u_-; 4 R) \le \delta,
\end{equation}
then,
\begin{equation} \label{ugedelta}
u \ge \delta \quad \mbox{in } B_R.
\end{equation}
The constant~$\delta$ depends only on~$n$,~$p$,~$\varepsilon_0$,~$H$,~$\gamma$ when~$n \ge 2$, and also on~$s$, when~$n = 1$.
\end{lemma}
\begin{proof}
First of all, by scaling, we can restrict ourselves to take~$R = 1$. Let~$\delta \in (0, 1/64]$ and~$\tau \in (0, 2^{- n - 1}]$ to be specified later. Let then~$\delta \le h < k \le 2 \delta$ and~$1 \le \rho < r \le 2$.

We initially suppose that
\begin{equation} \label{deltadens}
|B_2 \cap \{ u < 2 \delta \}| \le \tau |B_2|.
\end{equation}
Under this additional assumption (and if~$\tau$ is sufficiently small), we prove that~\eqref{ugedelta} holds true. Then, at a second stage, we will show that~\eqref{deltadens} is in fact a consequence of the hypotheses made in the statement of the lemma, provided~$\delta$ is chosen small enough.

By~\eqref{deltadens} and the upper bound on~$\tau$, we have
\begin{equation} \label{nullishalf}
\begin{aligned}
| B_\rho \cap \{ (u - k)_- = 0 \}| & = |B_\rho \setminus \{ u < k\}| \ge |B_\rho| - |B_2 \cap \{ u < 2 \delta \}| \\
& \ge \left[ 1 - \tau \left( \frac{2}{\rho} \right)^n \right] |B_\rho| \ge (1 - 2^n \tau) |B_\rho| \\
& \ge \frac{1}{2} \, |B_\rho|.
\end{aligned}
\end{equation}
Let~$\sigma \in (0, s)$ be defined by
$$
\sigma := \max \left\{ 2 s - 1, s - \frac{n \varepsilon_0}{2 p} \right\},
$$
and notice that~$n \ge 1 > \sigma$. Also,~$1 - \sigma \le 2 (1 - s)$ and~$C^{-1} (1 - s) \le s - \sigma \le 1 - s$, for some~$C \ge 1$ depending on~$n$,~$p$ and~$\varepsilon_0$. By using~\eqref{nullishalf}, Corollary~\ref{nullsobcor} and Lemma~\ref{sobinclem}, we find that
\begin{align*}
(k - h) |A^-(h, \rho)|^{\frac{n - \sigma}{n}} & \le \left[ \int_{A^-(h, \rho)} (k - u(x))^{\frac{n}{n - \sigma}} \, dx \right]^{\frac{n - \sigma}{n}} \le \left[ \int_{B_\rho} (u(x) - k)_-^{1^*_\sigma} \, dx \right]^{\frac{1}{1^*_\sigma}} \\
& \le C (1 - \sigma) \int_{B_\rho} \int_{B_\rho} \frac{|(u(x) - k)_- - (u(y) - k)_-|}{|x - y|^{n + \sigma}} \, dx dy \\
& \le C (1 - s) \int_{A^-(k, \rho)} \int_{B_\rho} \frac{|(u(x) - k)_- - (u(y) - k)_-|}{|x - y|^{n + \sigma}} \, dx dy \\
& \le C (1 - s)^{\frac{1}{p}} |A^-(k, \rho)|^{\frac{p - 1}{p}} [(u - k)_-]_{W^{s, p}(B_\rho)}.
\end{align*}
By raising the above inequality to the~$p$ power and applying~\eqref{DG-def}, we then get
\begin{equation} \label{imppostech1}
\begin{aligned}
(k - h)^p |A^-(h, \rho)|^{\frac{n - \sigma}{n} p} & \le C |A^-(k, \rho)|^{p - 1} \Bigg[ (d^p + k^p) |A^-(k, r)|^{1 - \frac{sp}{n} + \varepsilon} \\
& \quad + \frac{\| (u - k)_- \|_{L^p(B_r)}^p}{(r - \rho)^p} \\
& \quad + \frac{\| (u - k)_- \|_{L^1(B_r))} \overline{\Tail}((u - k)_-; \rho)^{p - 1}}{(r - \rho)^{n + s p}} \Bigg],
\end{aligned}
\end{equation}
where~$C$ may now depend on~$H$ too. On the one hand, thanks to~\eqref{uge0},
\begin{equation} \label{Lqcontrol}
\| (u - k)_- \|_{L^q(B_r)}^q = \int_{A^-(k, r)} \left( k - u(x) \right)^q \, dx \le |A^-(k, r)| k^q,
\end{equation}
for any~$q \ge 1$. On the other hand, using~\eqref{udecay}, once again~\eqref{uge0} and that~$k \ge \delta$, we compute
\begin{align*}
\overline{\Tail}((u - k)_-; \rho)^{p - 1} & = (1 - s) \int_{\R^n \setminus B_\rho} \frac{(k - u(x))_+^{p - 1}}{|x|^{n + s p}} \, dx \\
& \le C \left[ k^{p - 1} \int_{\R^n \setminus B_\rho} \frac{dx}{|x|^{n + s p}} + (1 - s) \int_{\R^n \setminus B_4} \frac{u_-(x)^{p - 1}}{|x|^{n + s p}} \, dx \right]
\\
& = C \left[ \frac{\rho^{- s p}}{n \varepsilon_0} \, k^{p - 1} + 4^{- s p} \Tail(u_-; 4)^{p - 1} \right] \le C \left[ k^{p - 1} + \delta^{p - 1} \right] \\\
& \le C k^{p - 1}.
\end{align*}
By exploiting the last two estimates in~\eqref{imppostech1}, together with the fact that, by assumption~\eqref{udecay},~$d \le \delta \le k$, we easily conclude that
$$
(k - h)^p |A^-(h, \rho)|^{\frac{n - \sigma}{n} p} \le C (r - \rho)^{- n - s p} k^p |A^-(k, r)|^{p - \frac{s p}{n} + \varepsilon},
$$
which, thanks to the definition of~$\sigma$, in turn implies that
\begin{equation} \label{imppostech2}
(k - h)^{\frac{n}{n - \sigma}} |A^-(h, \rho)| \le C (r - \rho)^{- \frac{n (n + p)}{(n - 1) p}} k^{\frac{n}{n - \sigma}} |A^-(k, r)|^{1 + \frac{\varepsilon_0}{2 p}},
\end{equation}
at least when~$n \ge 2$.

Consider the sequences~$\{ r_i \}$ and~$\{ k_i \}$ defined by~$r_i := 1 + 2^{-i}$ and~$k_i := (1 + 2^{-i}) \delta$. Also set~$\phi_i := |A^-(k_i, r_i)| / |B_{r_i}|$. By applying~\eqref{imppostech2} with~$h = k_i$,~$k = k_{i - 1}$,~$\rho = r_i$ and~$r = r_{i - 1}$, we obtain
$$
\phi_i \le C 2^{\frac{n (n + 2 p)}{(n - 1) p} i} \phi_{i - 1}^{1 + \frac{\varepsilon_0}{2 p}}.
$$
Note that, by~\eqref{deltadens}, we know that
$$
\phi_0 = \frac{|A^-(2 \delta, 2)|}{|B_2|} \le \tau.
$$
Therefore, we may apply e.g.~\cite[Lemma~7.1]{Giu03} to deduce that~\eqref{ugedelta} holds true, at least if~$\tau$ is chosen sufficiently small, in dependence of~$n$,~$p$,~$\varepsilon_0$ and~$H$ only.

Note that, when~$n = 1$, one can deduce the same fact as above. But in this case~$\tau$ would depend on~$s$ too.

In order to conclude the proof, we now only need to show that the additional assumption~\eqref{deltadens} can be deduced from the hypotheses of lemma, provided~$\delta$ is small enough. To do so, we argue by contradiction and suppose that
\begin{equation} \label{deltadenscontra}
|B_2 \cap \{ u < 2 \delta \}| \ge \tau |B_2|,
\end{equation}
with~$\tau$ fixed as before.

We employ once again inequality~\eqref{DG-def}. Notice that, up to here, we only took advantage of the estimate for the first term on the left-hand side of~\eqref{DG-def}. Now we use it to obtain a bound for the second summand too. By arguing as before, we deduce from~\eqref{DG-def}---applied with~$r = 2$ and~$R = 3$---that, for any~$\ell \in [\delta, 1]$,
\begin{equation} \label{2summbound}
(1 - s) \left[ [(u - \ell)_-]_{W^{s, p}(B_2)}^p + \int_{B_2} \int_{B_2} \frac{(u(x) - \ell)_+^{p - 1} (u(y) - \ell)_-}{|x - y|^{n + s p}} \, dx dy \right] \le C_1 \ell^p,
\end{equation}
with~$C_1 \ge 1$ only depending on~$n$,~$p$,~$\varepsilon_0$ and~$H$.

We start by addressing the case~$n \ge 2$. Let~$\bar{s} \in (0, 1)$ be the parameter found in Proposition~\ref{sDGlemprop}, in correspondence to the choices~$M = 8^p \left( |B_1| + C_1 \right)$,~$\gamma_1 = \tau$ and~$\gamma_2 = \gamma$. Observe that~$\bar{s}$ depends only on~$n$,~$p$,~$\varepsilon_0$,~$H$ and~$\gamma$.

Suppose that~$s \in [\bar{s}, 1)$. Let~$m \ge 5$ be the unique integer for which
\begin{equation} \label{mdeltadef}
2^{- m - 1} \le \delta < 2^{-m}.
\end{equation}
Consider the decreasing sequence~$k_i := 2^{-i}$, for~$i \in \{ 0, \ldots, m \}$. Notice that~$k_i \in (2 \delta, 1]$ for any~$i \in \{ 1, \ldots, m - 1 \}$. Moreover, by~\eqref{1dens},~\eqref{Lqcontrol},~\eqref{deltadenscontra} and~\eqref{2summbound} it is easy to see that
\begin{align*}
|B_2 \cap \{ (u - k_{i - 1})_- \le 2^{-i} \}| & = |B_2 \cap \{ u \ge k_i \}| \ge |B_2 \cap \{ u \ge 1 \}| \ge \gamma |B_2|, \\
|B_2 \cap \{ (u - k_{i - 1})_- \ge 3 \cdot 2^{- i - 1} \}| & = |B_2 \cap \{ u \le k_{i + 1} \}| \ge |B_2 \cap \{ u < 2 \delta \}| \ge \tau |B_2|,
\end{align*}
and
\begin{equation} \label{Gagu-k}
\begin{aligned}
& \| (u - k_{i - 1})_- \|_{L^p(B_2)}^p + (1 - s) [(u - k_{i - 1})_-]_{W^{s, p}(B_2)}^p \\
& \hspace{50pt} \le \left( |A^-(k_{i - 1}, 2)| + C_1 \right) k_{i - 1}^p \le 4^p \left( |B_1| + C_1 \right) 2^n (k_i - k_{i + 1})^p,
\end{aligned}
\end{equation}
for any~$i \in \{ 1, \ldots, m - 2 \}$. Consequently, we can apply Proposition~\ref{sDGlemprop} to the function~$(u - k_{i - 1})_-$, with~$h = k_{i - 1} - k_i = 2^{- i}$ and~$k = k_{i - 1} - k_{i + 1} = 3 \cdot 2^{- i - 1}$. We easily get
\begin{align*}
& | B_2 \cap \{ u \le k_{i + 1} \} |^{\frac{n - 1}{n}} \le C \, 2^i (1 - s)^{1 / p} [(u - k_{i - 1})_-]_{W^{s, p}(B_2)} |B_2 \cap \{ k_{i + 1} < u < k_i \}|^{\frac{p - 1}{p}},
\end{align*}
for some~$C \ge 1$ depending only on~$n$,~$p$ and~$\gamma$. By means of~\eqref{Gagu-k}, we can control the Gagliardo seminorm of~$(u - k_{i - 1})_-$ and deduce that, for any~$i \in \{ 1, \ldots, m - 2 \}$,
$$
| B_2 \cap \{ u < 2 \delta \} |^{\frac{(n - 1) p}{n (p - 1)}} \le | B_2 \cap \{ u \le k_{i + 1} \} |^{\frac{(n - 1) p}{n (p - 1)}} \le C |B_2 \cap \{ k_{i + 1} < u < k_i \}|,
$$
where~$C$ may now depend on~$\varepsilon_0$ and~$H$ too. By adding up the above inequality as~$i$ ranges between~$1$ and~$m - 2$, we find
$$
(m - 2) | B_2 \cap \{ u < 2 \delta \} |^{\frac{(n - 1) p}{n (p - 1)}} \le C \sum_{i = 1}^{m - 2} |B_2 \cap \{ k_{i + 1} < u < k_i \}| \le C,
$$
which in turn yields that
$$
| B_2 \cap \{ u < 2 \delta \} | \le C \, m^{- \frac{n (p - 1)}{(n - 1) p}} \le C |\log \delta|^{- \frac{n (p - 1)}{(n - 1) p}},
$$
thanks to~\eqref{mdeltadef}. But this is in contradiction with~\eqref{deltadenscontra}, if~$\delta$ is sufficiently small.

On the other hand, when~$s \in (0, \bar{s})$, we simply estimate
\begin{align*}
& (1 - s) \int_{B_2} \int_{B_2} \frac{(u(x) - 4 \delta)_+^{p - 1} (u(y) - 4 \delta)_-}{|x - y|^{n + s p}} \, dx dy \\
& \hspace{50pt} \ge \frac{1 - \bar{s}}{4^{n + p}} \int_{B_2 \cap \{ u \ge 1 \}} ( u(x) - 4 \delta)^{p - 1} \, dx \int_{B_2 \cap \{ u < 2 \delta \}} (4 \delta - u(y)) \, dy \\
& \hspace{50pt} \ge \frac{1 - \bar{s}}{4^{n + p}} \frac{2 \delta}{2^{p - 1}} |B_2 \cap \{ u \ge 1 \}| |B_2 \cap \{ u < 2 \delta \}| \\
&  \hspace{50pt} \ge \frac{\delta}{C} \, |B_2 \cap \{ u < 2 \delta \}|,
\end{align*}
where we used~\eqref{uge0},~\eqref{1dens}, that~$\delta \le 1/8$ and the fact that~$|x - y|^{n + s p} \le 4^{n + p}$, for any~$x, y \in B_2$. By comparing this inequality with~\eqref{2summbound}---used here with~$\ell = 4 \delta$---, we readily get
$$
|B_2 \cap \{ u < 2 \delta \}| \le C \delta^{p - 1},
$$
which, again, contradicts~\eqref{deltadenscontra}, provided~$\delta$ is chosen small enough.

The case~$n = 1$ can be treated exactly in the same way as we just did, for~$n \ge 2$ and~$s \in (0, \bar{s})$. Of course, this time~$\delta$ may not be independent of~$s$. The proof is therefore complete.
\end{proof}

We remark that the proof just displayed makes complete use of inequality~\eqref{DG-def} only when~$s$ is smaller than the parameter~$\bar{s}$ found in Proposition~\ref{sDGlemprop} (and when~$n \ge 2$). Indeed, when~$s \ge \bar{s}$, we only needed estimate~\eqref{DG-def} to control the first summand on its left-hand side. If~$\bar{s}$ could be chosen to depend only on~$n$ and~$p$ in Proposition~\ref{sDGlemprop}, then when~$s \ge \bar{s}$ one would be able to prove Lemma~\ref{growthlem}---and thus, H\"older continuity, as we shall see momentarily---for a larger class of functions than~$\DG^{s, p}$. Namely, one could drop the second term on the left-hand side of inequalities~\eqref{DG+def}-\eqref{DG-def} and hence prove regularity for all functions that satisfy a more standard Caccioppoli-type inequality such as~\eqref{1scaccintro}.

Also notice that Proposition~\ref{sDGlemprop} has been used for the sole purpose of having~$\delta$ independent of~$s$. This mainly implies that our~$C^\alpha$ estimates will be independent of~$s$ as well, for~$s$ far from~$0$. On the contrary, if one is not interested in obtaining uniform estimates, then the proof of Lemma~\ref{growthlem} simplifies, as the same argument that we adopted for~$s \le \bar{s}$ can be reproduced in the case of a general~$s \in (0, 1)$.

\smallskip

Thanks to Lemma~\ref{growthlem}, we are now in position to prove the H\"older regularity of functions in fractional~De~Giorgi classes.

\begin{theorem}[\bfseries H\"older continuity of fractional De~Giorgi functions] \label{DGholdthm}
\textcolor{white}{}\\
Let~$u \in \DG^{s, p}(\Omega; d, H, -\infty, \varepsilon, \lambda, R_0)$, with~$0 < \varepsilon_0 \le \varepsilon \le s p / n$. Then~$u \in C^\alpha_\loc(\Omega)$, for some~$\alpha \in (0, 1)$. Moreover, given any~$x_0 \in \Omega$ and~$0 < R < \min \{ \dist \left( x_0, \partial \Omega \right), R_0 \} / 8$, it holds
$$
[ u ]_{C^\alpha(B_R(x_0))} \le \frac{C}{R^\alpha} \Big( \| u \|_{L^\infty(B_{4 R}(x_0))} + \Tail(u; x_0, 4 R) + R^{\frac{\lambda + n \varepsilon}{p}} d \Big),
$$
for some~$C \ge 1$. The constants~$\alpha$ and~$C$ depend only on~$n$,~$p$,~$\varepsilon_0$,~$H$ when~$n \ge 2$, and also on~$s$ when~$n = 1$.
\end{theorem}
\begin{proof}
Assume without loss of generality that~$x_0 = 0$. Let~$\delta \in (0, 1/8]$ be the constant found in Lemma~\ref{growthlem}---with~$\gamma = 1/2$ and~$4^p H$ instead of~$H$. Take
\begin{equation} \label{alphadef}
0 < \alpha \le \min \left\{ \frac{n \varepsilon_0}{2 p} , \log_4 \left( \frac{2}{2 - \delta} \right) \right\},
\end{equation}
in such a way that
\begin{equation} \label{alphacond}
\int_4^{+\infty} \frac{(\rho^\alpha - 1)^{p - 1}}{\rho^{1 + n \varepsilon_0}} \, d\rho \le \frac{\varepsilon_0 \delta}{8^{p + 1} p \max \{ 1, |B_1| \}}.
\end{equation}
Observe that, by Lebesgue's dominated convergence theorem, the integral appearing in~\eqref{alphacond} can be made as small as desired, by taking~$\alpha$ sufficiently small. Set
\begin{equation} \label{j0def}
j_0 := \left\lceil \frac{2}{n \varepsilon_0} \log_4 \left( \frac{8^{p + 1} p \max \{ 1, |B_1| \}}{\varepsilon_0 \delta} \right) \right\rceil.
\end{equation}

We claim that there exist a non-decreasing sequence~$\{ m_i \}$ and a non-increasing sequence~$\{ M_i \}$ of real numbers, such that, for any~$i \in \N \cup \{ 0 \}$,
\begin{equation} \label{miMidef}
\begin{aligned}
m_i \le u \le M_i \mbox{ in } B_{4^{1 - i} R} \quad \mbox{and} \quad M_i - m_i = 4^{- \alpha i} L,
\end{aligned}
\end{equation}
with
\begin{equation} \label{Ldef}
L := 2 \cdot 4^{\frac{n \varepsilon_0}{2 p} j_0} \| u \|_{L^\infty(B_{4 R})} + \Tail(u; 4 R) + R^{\frac{\lambda + n \varepsilon}{p}} d.
\end{equation}

We proceed by induction on the index~$i$. Set~$m_i := - 4^{- \alpha i} L / 2$ and~$M_i := 4^{- \alpha i} L / 2$, for any~$i = 0, \ldots, j_0$. Then,~\eqref{miMidef} holds for these~$i$'s, thanks to~\eqref{alphadef} and~\eqref{Ldef}. Now we fix an integer~$j \ge j_0$ and suppose that the sequences~$\{ m_i \}$ and~$\{ M_i \}$ have been constructed up to~$i = j$. Claim~\eqref{miMidef} will be proved once we construct~$m_{j + 1}$ and~$M_{j + 1}$ appropriately.

Consider the function
\begin{equation} \label{vdef}
v := \frac{2 \cdot 4^{\alpha j}}{L} \left( u - \frac{M_j + m_j}{2} \right).
\end{equation}
By~\eqref{miMidef} and the monotonicity of~$\{ m_j \}$,~$\{ M_j \}$, we have that
\begin{equation} \label{meanjest}
\left| M_j + m_j \right| \le \left( 1 - 4^{- \alpha j} \right) L.
\end{equation}
Then, it is not hard to see that
\begin{equation} \label{vDG}
v \in \DG^{s, p} \left( B_{8 R}; \frac{2 \cdot 4^{\alpha j}}{L} \, d + R^{- \frac{n \varepsilon + \lambda}{p}} \left( 4^{\alpha j} - 1 \right), 2^p H, -\infty, \varepsilon, \lambda, R_0 \right),
\end{equation}
and
\begin{equation} \label{|v|le1}
|v| \le 1 \quad \mbox{in } B_{4^{1 - j} R}.
\end{equation}

Take now~$x \in B_{4 R} \setminus B_{4^{1 - j} R}$ and let~$\ell \in \{ 0, \ldots, j - 1 \}$ be the unique integer for which~$x \in B_{4^{1 - \ell} R} \setminus B_{4^{- \ell} R}$. By~\eqref{miMidef} and the monotonicity of~$\{ m_i \}$ we have
\begin{align*}
v(x) & \le \frac{2 \cdot 4^{\alpha j}}{L} \left[ M_\ell - m_\ell + m_\ell - \frac{M_j + m_j}{2} \right] \le \frac{2 \cdot 4^{\alpha j}}{L} \left[ M_\ell - m_\ell + m_j - \frac{M_j + m_j}{2} \right] \\
& = \frac{2 \cdot 4^{\alpha j}}{L} \left[ M_\ell - m_\ell - \frac{M_j - m_j}{2} \right] = 2 \cdot 4^{\alpha (j - \ell)} - 1 \\
& \le 2 \left( 4^j \frac{|x|}{R} \right)^\alpha - 1.
\end{align*}
Similarly, one checks that~$v(x) \ge - 2 \left( 4^j |x| / R \right)^\alpha + 1$, and hence
\begin{equation} \label{1pmvinsideest}
(1 \pm v(x))_-^{p - 1} \le 2^{p - 1} \left[ \left( 4^j \frac{|x|}{R} \right)^{\alpha} - 1 \right]^{p - 1} \quad \mbox{for a.a.~} x \in B_{4 R} \setminus B_{4^{1 - j} R}.
\end{equation}
On the other hand, using~\eqref{meanjest} we easily get that
\begin{equation} \label{1pmvoutsideest}
(1 \pm v)_-^{p - 1} \le 2^{p - 1} \left[ \left( \frac{2 \cdot 4^{\alpha j}}{L} \right)^{p - 1} |u|^{p - 1} + 4^{\alpha (p - 1) j} \right] \quad \mbox{a.e.~in } \R^n \setminus B_{4 R}.
\end{equation}
With the help of~\eqref{1pmvinsideest},~\eqref{1pmvoutsideest} and changing variables appropriately, we compute
\begin{align*}
& \Tail((1 \pm v)_-; 4^{1 - j} R)^{p - 1} \\
& \hspace{10pt} \le 4^{- j s p + 2 p} R^{s p} \left[ \rule{0pt}{26pt} \int_{\R^n \setminus B_{4^{1 - j} R}} \frac{\left[ \left( 4^j \frac{|x|}{R} \right)^\alpha - 1 \right]^{p - 1}}{|x|^{n + s p}} \, dx \right. \\
& \hspace{10pt} \quad + \left. (1 - s) \left( \frac{4^{\alpha j}}{L} \right)^{p - 1} \int_{\R^n \setminus B_{4 R}} \frac{|u(x)|^{p - 1}}{|x|^{n + s p}} \, dx + 4^{\alpha (p - 1) j} \int_{\R^n \setminus B_{4 R}} \frac{dx}{|x|^{n + s p}} \rule{0pt}{28pt} \right] \\
& \hspace{10pt} \le \frac{8^p p \max \{ 1, |B_1| \}}{\varepsilon_0} \left[ \int_{4}^{+ \infty} \frac{\left( \rho^\alpha - 1 \right)^{p - 1}}{\rho^{1 + n \varepsilon_0}} \, d\rho + 4^{\left( \alpha p - n \varepsilon_0 \right) j} \left( \frac{\Tail(u; 4 R)^{p - 1}}{L^{p - 1}} + 1 \right) \right].
\end{align*}
Recalling~\eqref{alphadef},~\eqref{alphacond},~\eqref{j0def} and~\eqref{Ldef}, we are led to conclude that
\begin{equation} \label{Tailcontrol}
\Tail((1 \pm v)_-; 4^{1 - j} R) \le \frac{\delta}{2}.
\end{equation}

Now, we have that either
\begin{equation} \label{holderdico}
\left| B_{4^{1 - j} R / 2} \cap \{ v \ge 0 \} \right| \ge \frac{1}{2} \left| B_{4^{1 - j} R / 2} \right| \mbox{ or } \left| B_{4^{1 - j} R / 2} \cap \{ v \ge 0 \} \right| < \frac{1}{2} \left| B_{4^{1 - j} R / 2} \right|.
\end{equation}
In the first of the two situations described by~\eqref{holderdico} we set~$w := 1 + v$, while in the second~$w := 1 - v$. In any case, we obtain
$$
\left| B_{4^{1 - j} R / 2} \cap \{ w \ge 1 \} \right| \ge \frac{1}{2} \left| B_{4^{1 - j} R / 2} \right|.
$$
Furthermore,
\begin{align*}
w & \in \DG^{s, p} \left( B_{4^{1 - j} R}; \frac{2 \cdot 4^{\alpha j}}{L} \, d + R^{- \frac{n \varepsilon + \lambda}{p}} 4^{\alpha j}, 4^p H, -\infty, \varepsilon, \lambda, R_0 \right), \\
w & \ge 0 \mbox{ in } B_{4^{1 - j} R},
\end{align*}
and
$$
\left( 4^{- j} R \right)^{\frac{\lambda + n \varepsilon}{p}} \left[ \frac{2 \cdot 4^{\alpha j}}{L} \, d + R^{- \frac{n \varepsilon + \lambda}{p}} 4^{\alpha j} \right] + \Tail(w_-; 4^{1 - j} R) \le \delta,
$$
thanks to~\eqref{vDG},~\eqref{|v|le1},~\eqref{Tailcontrol},~\eqref{alphadef} and~\eqref{j0def}. Therefore, we are in position to apply Lemma~\ref{growthlem} to~$w$. We deduce that
$$
w \ge \delta \quad \mbox{in } B_{4^{- j} R}.
$$
Assume for instance that the first alternative in~\eqref{holderdico} is satisfied. By taking advantage of the above estimate,~\eqref{vdef} and~\eqref{miMidef},
\begin{align*}
u(x) & = \frac{M_j + m_j}{2} + \frac{L}{2 \cdot 4^{\alpha j}} \, v(x) = \frac{M_j + m_j}{2} + \frac{L}{2 \cdot 4^{\alpha j}} \left( w(x) - 1 \right) \\
& \ge M_j - \frac{M_j - m_j}{2} - \frac{L}{2 \cdot 4^{\alpha j}} (1 - \delta) \\
& \ge M_j - \frac{L}{4^{\alpha j}} \frac{2 - \delta}{2},
\end{align*}
for any~$x \in B_{4^{- j} R}$. In view of~\eqref{alphadef}, we finally conclude that
$$
M_j - 4^{ - (j + 1) \alpha} L \le u \le M_j \quad \mbox{in } B_{4^{- j} R},
$$
that is,~\eqref{miMidef} is true for~$i = j + 1$, with~$M_{j + 1} := M_j$ and~$m_{j + 1} := M_{j + 1} - 4^{- (j + 1) \alpha} L$. Of course, if instead the second alternative in~\eqref{holderdico} is valid, an analogous argument leads to the same conclusion, with~$m_{j + 1} := m_j$ and~$M_{j + 1} := m_{j + 1} + 4^{-(j + 1) \alpha} L$.

Claim~\eqref{miMidef} holds therefore for any~$i \in \N \cup \{ 0 \}$, and the H\"older continuity of~$u$ follows in a standard way.
\end{proof}

\subsection{Harnack inequality} \label{harsubsec}

The conclusive part of this section is devoted to establishing a Harnack-type inequality for functions in fractional De~Giorgi classes.

For simplicity of exposition, we restrict ourselves to assume~$n \ge 2$ throughout the whole subsection. In this way, the constants involved in the various propositions are independent of~$s$, if~$s$ is bounded away from~$0$ (at least if~$p \ne n$). When~$n = 1$, all the arguments displayed are still valid, but several estimates may not be uniform in~$s$.

\smallskip

As a first step towards the aforementioned goal, we have the following result, that slightly improves Lemma~\ref{growthlem}.

\begin{lemma} \label{2growthlem}
Let~$t > 0$ and~$u \in \DG_-^{s, p}(B_{4 R}(x_0); d, H, -t, \varepsilon, \lambda, R_0)$, for some~$x_0 \in \R^n$ and~$R > 0$, with~$R_0 \ge 4 R$ and~$0 < \varepsilon_0 \le \varepsilon \le s p / n$. Suppose that
$$
u \ge 0 \quad \mbox{in } B_{4 R}(x_0),
$$
and
$$
|B_R(x_0) \cap \{ u \ge t \}| \ge \gamma |B_R|,
$$
for some~$\gamma \in (0, 1)$. There exists a constant~$\delta > 0$, depending only on~$n$,~$p$,~$\varepsilon_0$,~$H$ and~$\gamma$, such that, if
$$
R^{\frac{\lambda + n \varepsilon}{p}} d + \Tail(u_-; x_0, 4 R) \le \delta t,
$$
then,
$$
u \ge \delta t \quad \mbox{in } B_R(x_0).
$$
\end{lemma}
\begin{proof}
Of course, we can assume~$x_0$ to be the origin. We begin by addressing the case of~$t = 1$. Set
$$
\tilde{\gamma} := 2^{- n} \gamma \in (0, 1),
$$
and let~$\delta$ be the constant found in Lemma~\ref{growthlem}, corresponding to the above defined~$\tilde{\gamma}$. It holds
$$
|B_{2 R} \cap \{ u \ge 1 \}| \ge |B_R \cap \{ u \ge 1 \}| \ge \gamma |B_R| = \tilde{\gamma} |B_{2 R}|.
$$
Hence, we are in position to apply Lemma~\ref{growthlem} and deduce that
$$
u \ge \delta \quad \mbox{in } B_R.
$$
The lemma is therefore proved, for~$t = 1$.

The general case of~$t > 0$ can be then easily deduced. Indeed, let~$v := t^{-1} u$. The function~$v$ belongs to~$\DG_-^{s, p}(B_{4 R}; d/t, H, -1, \varepsilon, R_0)$ and fulfills the hypotheses of the lemma with~$t = 1$ and~$d/t$ in place of~$d$. Thus, from the preceding argument we deduce that~$u = t v \ge t \delta$, and the proof is complete.
\end{proof}

Next, we use Lemma~\ref{2growthlem} to prove a~\emph{weak Harnack inequality}, which, together with Proposition~\ref{ulocboundprop}, will lead to the proper Harnack inequality.

In order to do this, we first recall a classical covering lemma of Krylov and Safonov~\cite{KS80}. We present it here in a version with balls in place of cubes, due to~\cite{KS01}.

\begin{lemma}[{\cite[Lemma~7.2]{KS01}}] \label{KSlem}
Let~$\gamma \in (0, 1)$,~$R > 0$ and~$E \subseteq B_R$ be a measurable set. Define
$$
E_\gamma := \bigcup \Big\{ B_R \cap B_{3 r}(x_0) : x_0 \in B_R, \, r > 0 \mbox{ and } |B_{3 r}(x_0) \cap E| \ge \gamma |B_r| \Big\}.
$$
Then, either~$E_\gamma = B_R$ or
$$
|E_\gamma| \ge \frac{1}{2^n \gamma} |E|.
$$
\end{lemma}

With the aid of Lemma~\ref{KSlem} we can now prove the following result.

\begin{lemma} \label{3growthlem}
Let~$\gamma \in (0, 1)$,~$t > 0$ and~$k \in \N$. Let~$u \in \DG_-^{s, p}(B_{16 R}; d, H, -t, \varepsilon, \lambda, R_0)$, with~$R > 0$,~$R_0 \ge 16 R$ and~$0 < \varepsilon_0 \le \varepsilon \le s p / n$. Suppose that
\begin{equation} \label{uge0B16}
u \ge 0 \quad \mbox{in } B_{16 R},
\end{equation}
and
\begin{equation} \label{supergegammak}
|B_R \cap \{ u \ge t \}| \ge \gamma^k |B_R|.
\end{equation}
There exists a constant~$\delta \in (0, 1/8]$, depending only on~$n$,~$p$,~$\varepsilon_0$,~$H$ and~$\gamma$, such that, if
\begin{equation} \label{Tailledeltak}
R^{\frac{\lambda + n \varepsilon}{p}} d + \Tail(u_-; 16 R) \le \delta^k t,
\end{equation}
then
$$
u \ge \delta^k t \quad \mbox{in } B_{R}.
$$
\end{lemma}
\begin{proof}
Set~$\gamma_1 := 2^{- n} \gamma$ and~$\gamma_2 := 3^{-n} \gamma_1 = 6^{-n} \gamma$. Let~$\delta \in (0, 1/8]$ be the constant found in Lemma~\ref{2growthlem}, in correspondence to~$\gamma_2$. For any~$i \in \N \cap \{ 0 \}$, write
$$
A^i := B_R \cap \left\{ u \ge \delta^i t \right\}.
$$
Clearly,
\begin{equation} \label{Aincl}
A^{i - 1} \subseteq A^i \quad \mbox{for any } i \in \N,
\end{equation}
as~$\delta \le 1$.

Notice that, in order to prove the lemma, it suffices to show that
\begin{equation} \label{Ak-1ge}
|A^{k - 1}| \ge \gamma_2 |B_R|,
\end{equation}
since then an application of Lemma~\ref{2growthlem} would yield the thesis.

Let~$i \in \{ 1, \ldots, k - 1 \}$ be fixed and suppose that, in the notation of Lemma~\ref{KSlem}, it holds~$B_R \cap B_{3 r}(x_0) \subseteq (A^{i - 1})_{\gamma_1}$, for some~$x_0 \in B_R$ and~$r > 0$. This implies that
$$
|B_{3 r}(x_0) \cap \{ u \ge \delta^{i - 1} t \}| \ge |B_{3 r}(x_0) \cap A^{i - 1}| \ge \gamma_1 |B_r| = \gamma_2 |B_{3 r}|.
$$
Moreover, since we may suppose without loss of generality that~$r \le R/3$, we have
\begin{gather*}
u \ge 0 \quad \mbox{in } B_{12 r}(x_0), \\
(3 r)^{\frac{\lambda + n \varepsilon}{p}} d \le R^{\frac{\lambda + n \varepsilon}{p}} d \le \delta^k t \le \frac{\delta^i t}{2}
\end{gather*}
and
$$
\Tail(u_-; x_0, 12 r) = \left( \frac{12 r}{16 R} \right)^{\frac{s p}{p - 1}} \Tail(u_-; 16 R) \le \Tail(u_-; 16 R) \le \delta^k t \le \frac{\delta^i t}{2},
$$
thanks to~\eqref{uge0B16},~\eqref{Tailledeltak} and the fact that~$\delta \le 1/2$. Accordingly, an application of Lemma~\ref{2growthlem} gives that
$$
u \ge \delta^i t \quad \mbox{in } B_{3 r}(x_0).
$$
We have therefore proved that
\begin{equation} \label{AgammainA}
(A^{i - 1})_{\gamma_1} \subseteq A^i \quad \mbox{for any } i \in \{ 1, \ldots, k - 1 \}.
\end{equation}

We now claim that either
\begin{equation} \label{claimA}
\begin{aligned}
& \mbox{there exists }  i \in \{ 1, \ldots, k - 1 \} \mbox{ such that } A^i = B_R, \\
& \mbox{or } |A^i| \ge \frac{1}{\gamma} |A^{i - 1}| \mbox{ for any } i \in \{ 1, \ldots, k - 1 \}.
\end{aligned}
\end{equation}
Indeed, suppose that~$\gamma |A^i| < |A^{i - 1}|$ for some index~$i \in \{ 1, \ldots, k - 1 \}$. By~\eqref{AgammainA}, we deduce that~$(2^n \gamma_1) |(A^{i - 1})_{\gamma_1}| = \gamma |(A^{i - 1})_{\gamma_1}| < |A^{i - 1}|$. But then Lemma~\ref{KSlem} yields that~$(A^{i - 1})_{\gamma_1} = B_R$ and thus~$A^i = B_R$, using once again~\eqref{AgammainA}. Consequently, claim~\eqref{claimA} holds true.

We now show that~\eqref{claimA} implies~\eqref{Ak-1ge}. As noted before, this will conclude the proof. Indeed, if~$A^i = B_R$ for some~$i \in \{ 1, \ldots, k - 1 \}$, then~$A^{k - 1} = B_R$, thanks to~\eqref{Aincl}, and~\eqref{Ak-1ge} follows trivially. On the other hand, if the other option in~\eqref{claimA} is verified, then
$$
|A^{k - 1}| \ge \frac{1}{\gamma} |A^{k - 2}| \ge \frac{1}{\gamma^2} |A^{k - 3}| \ge \ldots \ge \frac{1}{\gamma^{k - 1}} |A^0| \ge \gamma |B_R|,
$$
where the last inequality is true in view of~\eqref{supergegammak}. Hence, we have verified the validity of the bound~\eqref{Ak-1ge} also in this case, since~$\gamma = 6^n \gamma_2 \ge \gamma_2$. Thence, the proof is complete.
\end{proof}

Starting from this result, the derivation of the weak Harnack inequality is rather straightforward.

\begin{proposition} \label{weakHarprop}
Let~$u \in \DG_-^{s, p}(B_{16 R}; d, H, - \infty, \varepsilon, \lambda, R_0)$, with~$R > 0$,~$R_0 \ge 16 R$ and~$0 < \varepsilon_0 \le \varepsilon \le s p / n$, and suppose that~\eqref{uge0B16} holds true. Then, there exist a small~$q \in (0, 1)$ and a large~$C \ge 1$, both depending only on~$n$,~$p$,~$\varepsilon_0$ and~$H$, such that
\begin{equation} \label{weakHarine}
\left( \dashint_{B_R} u(x)^q \, dx \right)^{\frac{1}{q}} \le C \left( \inf_{B_R} u + \Tail(u_-; R) + R^{\frac{\lambda + n \varepsilon}{p}} d \right).
\end{equation}
\end{proposition}
\begin{proof}
Of course, we can assume that~$u$ does not vanish identically on~$B_R$, otherwise~\eqref{weakHarine} is obviously true. Let~$\delta \in (0, 1/8]$ be the constant given by Lemma~\ref{3growthlem}, for~$\gamma = 1/2$. Set
$$
a := \frac{\log \gamma}{\log \delta} = \frac{1}{\log_{\frac{1}{2}} \delta} \in (0, 1).
$$

We claim that
\begin{equation} \label{wHmainclaim}
\inf_{B_R} u + \Tail(u_-; 16 R) + R^{\frac{\lambda + n \varepsilon}{p}} d \ge \delta \left( \frac{|A^+(t, R)|}{|B_R|} \right)^{ \frac{1}{a} } t,
\end{equation}
for any~$t \ge 0$. Notice that, if~$u$ is bounded from above in~$B_R$, then, by~\eqref{uge0B16}, inequality~\eqref{wHmainclaim} holds trivially for any~$t \ge \sup_{B_R} u$. Thus, it suffices to verify~\eqref{wHmainclaim} for any~$t \in [0, u^*)$, where~$u^* \in (0, +\infty]$ denotes the supremum of~$u$ in~$B_R$.

Given~$t \in [0, u^*)$, let~$k = k(t)$ be the smallest integer for which
\begin{equation} \label{wHlevest}
|A^+(t, R)| \ge 2^{-k} |B_R|,
\end{equation}
i.e.,~$k$ is the only integer for which
$$
\log_{\frac{1}{2}} \frac{|A^+(t, R)|}{|B_R|} \le k < 1 + \log_{\frac{1}{2}} \frac{|A^+(t, R)|}{|B_R|}.
$$
Observe that, with this choice, it holds
\begin{equation} \label{wHdeltakbound}
\delta^k \ge \delta \left( \frac{|A^+(t, R)|}{|B_R|} \right)^{\frac{1}{a}}.
\end{equation}
Furthermore,
\begin{equation} \label{wHclaim}
\inf_{B_R} u + \Tail(u_-; 16 R) + R^{\frac{\lambda + n \varepsilon}{p}} d \ge \delta^k t.
\end{equation}
Indeed, if~$\Tail(u_-; 16 R) + R^{\frac{\lambda + n \varepsilon}{p}} d \ge \delta^k t$, then~\eqref{wHclaim} is true, thanks to hypothesis~\eqref{uge0B16}. On the other hand, if~$\Tail(u_-; 16 R) + R^{\frac{\lambda + n \varepsilon}{p}} d < \delta^k t$, then this and~\eqref{wHlevest} enable us to apply Lemma~\ref{3growthlem} and deduce that
$$
u \ge \delta^k t \quad \mbox{in } B_{R}.
$$
Again,~\eqref{wHclaim} follows. Putting together~\eqref{wHclaim} and~\eqref{wHdeltakbound}, we see that~\eqref{wHmainclaim} is valid for any~$t \ge 0$.

Write now
$$
L := \inf_{B_R} u + \Tail(u_-; 16 R) + R^{\frac{\lambda + n \varepsilon}{p}} d,
$$
and note that~\eqref{wHmainclaim} is equivalent to
$$
\frac{|A^+(t, R)|}{|B_R|} \le \left( \frac{L}{\delta t} \right)^a.
$$
Using this inequality and Cavalieri's principle, for any~$q > 0$ we compute
$$
\dashint_{B_R} u(x)^q \, dx = q \int_0^{+\infty} \frac{|A^+(t, R)|}{|B_R|} \, t^{q - 1} \, dt \le q \left[ \int_0^L t^{q - 1} \, dt + \left( \frac{L}{\delta} \right)^a \int_{L}^{+\infty}  t^{q - 1 - a} \, dt \right].
$$
Choosing~$q = a / 2$, we then get
$$
\dashint_{B_R} u(x)^q \, dx \le (1 + \delta^{- a}) L^q,
$$
that is
\begin{align*}
\left( \dashint_{B_R} u(x)^q \, dx \right)^{\frac{1}{q}} & \le (1 + \delta^{- a})^{\frac{2}{a}} \left( \inf_{B_R} u + \Tail(u_-; 16 R) + R^{\frac{\lambda + n \varepsilon}{p}} d \right) \\
& \le (1 + \delta^{- a})^{\frac{2}{a}} \left( \inf_{B_R} u + 16^{\frac{p}{p - 1}} \Tail(u_-; R) + R^{\frac{\lambda + n \varepsilon}{p}} d \right).
\end{align*}
This yields~\eqref{weakHarine}.
\end{proof}

With this, we are now in position to prove the Harnack inequality for fractional De~Giorgi classes with~$r_0 = +\infty$.

\begin{theorem}[\bfseries Harnack inequality for fractional De~Giorgi functions] \label{DGharthm}
\textcolor{white}{}\\
Let~$u \in \DG^{s, p}(\Omega; d, H, -\infty, \varepsilon, \lambda, +\infty)$, with~$0 < \varepsilon_0 \le \varepsilon \le s p / n$ and~$0 \le \lambda \le \lambda_0$. Suppose that~$u \ge 0$ in~$\Omega$. Then, for any~$x_0 \in \Omega$ and~$0 < R < \dist (x_0, \partial \Omega) / 2$, it holds
\begin{equation} \label{Harine}
\sup_{B_R(x_0)} u \le C \left( \inf_{B_R(x_0)} u + \Tail(u_-; x_0, R) + R^{\frac{\lambda + n \varepsilon}{p}} d \right),
\end{equation}
for some~$C \ge 1$ depending on~$n$,~$s$,~$p$,~$\varepsilon_0$,~$\lambda_0$ and~$H$. When~$n \ne p$, the constant~$C$ does not blow up as~$s \rightarrow 1^-$.
\end{theorem}

\begin{proof}
We start supposing that~$x_0$ and~$R$ are such that
\begin{equation} \label{/64}
0 < R < \frac{\dist (x_0, \partial \Omega)}{64}. 
\end{equation}
At the end of the proof we will in fact show that this assumption is unnecessary.

We now prove that~\eqref{Harine} holds true under~\eqref{/64}. Up to a translation, we may assume~$x_0$ to be the origin. As a preliminary observation, we claim that, for any~$z \in B_R$ and~$0 < r \le 2 R$,
\begin{equation} \label{Tail+control}
\Tail(u_+; z, r) \le C \left( \sup_{B_r(z)} u + \Tail(u_-; z, r) + r^{\frac{\lambda + n \varepsilon}{p}} d \right),
\end{equation}
for some constant~$C \ge 1$ depending only on~$n$,~$p$,~$\varepsilon_0$ and~$H$. To check the validity of~\eqref{Tail+control}, we apply~\eqref{DG-def} with~$k := 2 M$ and
$$
M := \max \left\{ \sup_{B_r(z)} u, r^{\frac{\lambda + n \varepsilon}{p}} d \right\}.
$$
Focusing on just the second term on the left-hand side of such inequality, we get
\begin{equation} \label{Tail+tech1}
\begin{aligned}
& \int_{B_{\frac{r}{2}}(z)} (u(x) - 2 M)_- \left[ \int_{\R^n} \frac{(u(y) - 2 M)_+^{p - 1}}{|x - y|^{n + s p}} \, dy \right] dx \\
& \hspace{20pt} \le \frac{C}{1 - s} \bigg[ \left( r^\lambda d^p + \frac{M^p}{r^{n \varepsilon}} \right) |A^-(2 M, z, r)|^{1 - \frac{s p}{n} + \varepsilon} + r^{- s p} \| (u - 2 M)_- \|_{L^p(B_{r}(z))}^p \\
& \hspace{20pt} \quad + \| (u - 2 M)_- \|_{L^1(B_r(z))} \overline{\Tail} ((u - 2 M)_-; z, r/2)^{p - 1} \bigg]
\end{aligned}
\end{equation}
with~$C \ge 1$ depending on~$n$,~$p$ and~$H$.

We begin by dealing with the left-hand side. Observe that~$|x - y| \le 2 |y - z|$, for any~$x \in B_r(z)$ and~$y \in \R^n \setminus B_r(z)$. Moreover, by Lemma~\ref{numestlem4}, we know that
$$
(u(y) - 2 M)_+^{p - 1} \ge \min \{ 1, 2^{2 - p} \} u_+(y)^{p - 1} - 2^{p - 1} M^{p - 1}.
$$
Using these two facts and that~$u \le M$ on~$B_r(z)$, we compute
\begin{align*}
& \int_{B_{\frac{r}{2}}(z)} (u(x) - 2 M)_- \left[ \int_{\R^n} \frac{(u(y) - 2 M)_+^{p - 1}}{|x - y|^{n + s p}} \, dy \right] dx \\
& \hspace{30pt} \ge 2^{- n - s p} M \int_{B_{\frac{r}{2}}(z)} \left[ \int_{\R^n \setminus B_r(z)} \frac{\min \{ 1, 2^{2 - p} \} u_+(y)^{p - 1} - 2^{p - 1} M^{p - 1}}{|y - z|^{n + s p}} \, dy \right] dx \\
& \hspace{30pt} \ge \frac{M r^{n - sp}}{C (1 - s)} \Tail(u_+; z, r)^{p - 1} - C r^{n - s p} M^p,
\end{align*}
where~$C$ now depends on~$\varepsilon_0$ too.

On the other hand, by taking advantage once again of Lemma~\ref{numestlem4} and of the fact that~$u \ge 0$ on~$B_r(z)$, it is not hard to check that the right-hand side of~\eqref{Tail+tech1} can be bounded by
\begin{align*}
\frac{C r^{n - s p}}{1 - s} \Big( M^p + M \Tail(u_-; z, r)^{p - 1} \Big).
\end{align*}
By comparing these last two estimates with~\eqref{Tail+tech1}, we are easily led to~\eqref{Tail+control}.

We now proceed to prove the actual theorem. We consider separately the two cases~$n \ge s p$ and~$n < sp$.

If~$n \ge s p$, we take advantage of the boundedness given by Proposition~\ref{ulocboundprop}. We notice that estimate~\eqref{ulocbound} implies in this setting that, for any~$\delta_1 \in (0, 1]$,
\begin{equation} \label{Hineulocbound-1}
\sup_{B_r(z)} u \le C \, \delta_1^{- \frac{p - 1}{\beta p}} \left( \dashint_{B_{2 r}(z)} u(x)^p \, dx \right)^{\frac{1}{p}} + \delta_1 \Tail(u_+; z, r) + \delta_1^{\frac{p - 1}{p}} r^{\frac{\lambda + n \varepsilon}{p}} d,
\end{equation}
where~$C \ge 1$ and~$\beta \ge \varepsilon_0 / 2$ depend on~$n$,~$s$,~$p$,~$\varepsilon_0$ and~$H$. Both constants do not blow up as~$s \rightarrow 1^-$, when~$n > p$. By~\eqref{Tail+control}, this becomes
$$
\sup_{B_r(z)} u \le C \left[ \delta_1^{ - \frac{p - 1}{\beta p}} \left( \dashint_{B_{2 r}(z)} u(x)^p \, dx \right)^{\frac{1}{p}} + \delta_1^{\frac{p - 1}{p}} \left( \sup_{B_r(z)} u + \Tail(u_-; z, r) + r^{\frac{\lambda + n \varepsilon}{p}} d \right) \right].
$$
Fix now any~$q \in (0, p)$. Using the weighted Young's inequality, we have
\begin{align*}
\left( \dashint_{B_{2 r}(z)} u(x)^p \, dx \right)^{\frac{1}{p}} & \le \left( \sup_{B_{2 r}(z)} u \right)^{\frac{p - q}{p}} \left( \dashint_{B_{2 r}(z)} u(x)^q \, dx \right)^{\frac{1}{p}} \\
& \le \delta_2^{\frac{p}{p - q}} \sup_{B_{2 r}(z)} u + \delta_2^{- \frac{p}{q}} \left( \dashint_{B_{2 r}(z)} u(x)^q \, dx \right)^{\frac{1}{q}},
\end{align*}
for any~$\delta_2 > 0$. By taking~$\delta_1$,~$\delta_2$ sufficiently small, we obtain
\begin{equation} \label{Hinelocbound2}
\sup_{B_r(z)} u \le \frac{1}{2} \sup_{B_{2 r}(z)} u + C \left[ \left( \dashint_{B_{2 r}(z)} u(x)^q \, dx \right)^{\frac{1}{q}} + \Tail(u_-; z, r) + r^{\frac{\lambda + n \varepsilon}{p}} d \right].
\end{equation}

On the other hand, when~$n < s p$, we already know from the fractional Morrey embedding that~$u$ is bounded. In particular, a careful analysis of the proofs of~\cite[Theorem~8.2]{DPV12} and~\cite[Lemma~2.2]{Giu03} reveals, together with the sharp Poincar\'e-Wirtinger-type estimate of~\cite{BBM02,P04} and the extension inequality~\eqref{extine}, that
$$
\sup_{B_r(z)} u^p \le \frac{C}{\left( 2^{\frac{s p - n}{p}} - 1 \right)^p} \left[ (1 - s) \delta_3^{s p - n} r^{s p - n} [u_+]_{W^{s, p}(B_r(z))}^p + \delta_3^{- n} r^{- n} \| u_+ \|_{L^p(B_r(z))}^p \right],
$$
for any~$\delta_3 \in [0, 1]$ and for some constant~$C \ge 1$ depending only on~$n$ and~$p$. To get this estimate, one could prove it first in the case~$r = 1$ and then scale it. The arbitrary parameter~$\delta_3$ is essentially the constant~$R_0$ appearing in~\cite[formula~(8.9)]{DPV12}, while the denominator of the fraction in front of the square brackets comes from the proof of~\cite[Lemma~2.2]{Giu03}. By using~\eqref{DG+def}---with~$k = 0$,~$x_0 = z$ and~$R = 2 r$---to estimate the Gagliardo seminorm of~$u_+$, we are led to
$$
\sup_{B_r(z)} u^p \le C \left[ \delta_3^{- n} \dashint_{B_{2 r}(z)} u(x)^p \, dx + \delta_3^{s p - n} \Bigg( \Tail(u_+; z, r)^p + r^{\lambda + n \varepsilon} d^p \Bigg) \right],
$$
where~$C$ now depends on~$n$,~$s$,~$p$ and~$H$, but does not blow as~$s \rightarrow 1^-$ if~$p > n$ is fixed. By arguing as before, in the case~$n \ge s p$, and starting from this last inequality instead of~\eqref{Hineulocbound-1}, we deduce~\eqref{Hinelocbound2} once again.

We plan to take advantage of~\cite[Lemma~7.1]{Giu03}---or Lemma~\ref{induclem} here---to reabsorb the term~$(1 / 2)\sup_{B_{2 r}(z)} u$ on the left-hand side of~\eqref{Hinelocbound2}. To do it, we first need to perform an easy covering argument. Let~$R\le \rho < \tau \le 2 R$ be fixed. Note that
$$
B_\rho = \bigcup_{z \in B_{2 \rho - \tau}} B_{\tau - \rho}(z) \quad \mbox{and} \quad B_{2 (\tau - \rho)}(z) \subset B_\tau \mbox{ for any } z \in B_{2 \rho - \tau}.
$$
Therefore, by using~\eqref{Hinelocbound2} with~$r = \tau - \rho$, we get
\begin{align*}
\sup_{B_\rho} u & = \sup_{z \in B_{2 \rho - \tau}} \sup_{B_{\tau - \rho}(z)} u \\
& \le \frac{1}{2} \sup_{B_\tau} u + C \sup_{z \in B_{2 \rho - \tau}} \left[ \frac{\| u \|_{L^q(B_{2 (\tau - \rho)}(z))}}{(\tau - \rho)^{\frac{n}{q}}} + \Tail(u_-; z, \tau - \rho) + (\tau - \rho)^{\frac{\lambda + n \varepsilon}{p}} d \right] \\
& \le \frac{1}{2} \sup_{B_\tau} u + C \left[ \frac{\| u \|_{L^q(B_{2 R})}}{(\tau - \rho)^{\frac{n}{q}}} + \Tail(u_-; R) + R^{\frac{\lambda + n \varepsilon}{p}} d \right],
\end{align*}
where we also used that~$u \ge 0$ in~$B_R$. Now we are in position to apply~\cite[Lemma~7.1]{Giu03} and finally deduce that
$$
\sup_{B_R} u \le C \left[ \left( \dashint_{B_{2 R}} u(x)^q \, dx \right)^{\frac{1}{q}} + \Tail(u_-; R) + R^{\frac{\lambda + n \varepsilon}{p}} d \right].
$$
By choosing~$q$ as in Proposition~\ref{weakHarprop} and combining the above estimate with~\eqref{weakHarine}, we conclude that~\eqref{Harine} follows, at least when~\eqref{/64} is in force.

To finish the proof, we only need to show that the additional assumption~\eqref{/64} is not necessary, and that in fact~\eqref{Harine} is true for any~$x_0 \in \Omega$ and~$0 < R < \dist(x_0, \partial \Omega) / 2$. Let~$x_0$ and~$R$ be as such, and take any~$z \in B_R(x_0)$. Obviously,~$B_R(z) \subset B_{2 R}(x_0) \subset \Omega$. Moreover, as~\eqref{Harine} holds under~\eqref{/64} and~$u \ge 0$ in~$B_{2 R}(x_0)$, we know that
\begin{align*}
\sup_{B_{R / 64}(z)} u & \le C \left( \inf_{B_{R / 64}(z)} u + \Tail \left( u_-; z, \frac{R}{64} \right) + \left( \frac{R}{64} \right)^{\frac{\lambda + n \varepsilon}{p}} d \right) \\
& \le C \left( \inf_{B_{R / 64}(z)} u + \Tail \left( u_-; x_0, R \right) + R^{\frac{\lambda + n \varepsilon}{p}} d \right).
\end{align*}
In particular, this implies that, for any~$x', y' \in B_R(x_0)$ such that~$|x' - y'| < R / 32$, we have
\begin{equation} \label{Harx'y'}
u(x') \le C_\star \left( u(y') + \Tail \left( u_-; x_0, R \right) + R^{\frac{\lambda + n \varepsilon}{p}} d \right),
\end{equation}
with~$C_\star$ depending only on~$n$,~$s$,~$p$,~$\varepsilon_0$,~$\lambda_0$ and~$H$. Clearly, we can suppose that~$C_\star \ge 2$.

Let~$x, y \in B_R(x_0)$ be any two points. We now show that~\eqref{Harx'y'} is also true with~$x$,~$y$ in place of~$x'$,~$y'$, up to raising~$C_\star$ to a universal power. Of course, we can suppose that~$|x - y| \ge R / 32$, otherwise~\eqref{Harx'y'} applies to them directly. Denoting by~$\overrightarrow{xy}$ the oriented segment with end points~$x$ and~$y$, we consider~$N$ consecutive points~$\{ x_i \}_{i = 0}^N \subset \overrightarrow{xy}$ such that~$x_0 = x$,~$x_N = y$ and~$R / 64 \le |x_i - x_{i - 1}| < R / 32$ for any~$i \in \{ 1, \ldots, N \}$. It is immediate to check that~$N \le 128$. By applying~\eqref{Harx'y'} with~$x' = x_{i - 1}$ and~$y' = x_i$, we get
$$
u(x_{i - 1}) \le C_\star \left( u(x_i) + \Tail \left( u_-; x_0, R \right) + R^{\frac{\lambda + n \varepsilon}{p}} d \right).
$$
Iterating such estimate over~$i \in \{ 1, \ldots, N \}$, we easily find that
\begin{align*}
u(x) & \le C_\star^N u(y) + \left( \Tail \left( u_-; x_0, R \right) + R^{\frac{\lambda + n \varepsilon}{p}} d \right) \sum_{i = 1}^N C_\star^i \\
& \le C_\star^{129} \left( u(y) + \Tail \left( u_-; x_0, R \right) + R^{\frac{\lambda + n \varepsilon}{p}} d \right).
\end{align*}
This leads to~\eqref{Harine}, in its full generality.
\end{proof}

\section{Applications to minimizers} \label{minsec}

In this section, we show that the minimizers of the energy~$\E$ introduced in~\eqref{Edef} are locally bounded and H\"older continuous functions that satisfy a nonlocal Harnack inequality wherever non-negative. That is, we prove Theorems~\ref{boundmainthm},~\ref{holmainthm} and~\ref{harmainthm} for minimizers. Those three results are rephrased---and sometimes more precisely stated---as follows.

\begin{theorem}[\bfseries Local boundedness of minimizers] \label{minboundthm}
\textcolor{white}{}\\
Let~$n \in \N$,~$s \in (0, 1)$ and~$p > 1$ be such that~$n \ge s p$. Let~$\Omega$ be an open bounded subset of~$\R^n$. Assume that~$K$ and~$F$ respectively satisfy hypotheses~\eqref{Ksimm},~\eqref{Kell} and~\eqref{Fbounds}. If~$u$ is a minimizer of~$\E$ in~$\Omega$, then~$u \in L^\infty_\loc(\Omega)$. In particular, there exist four constants~$C \ge 1$,~$R_0 \in (0, \min \{ 1, r_0 / 2 \}]$,~$\varepsilon \in (0, sp / n]$ and~$\kappa \in \{0, 1\}$ such that, given any~$x_0 \in \Omega$ and~$0 < R \le \min \{ R_0, \dist \left( x_0, \partial \Omega \right) \} / 4$, it holds\footnote{Here and in Proposition~\ref{minareDGprop} we continue to understand~$p^*_s = +\infty$, when~$n = s p$. The same convention will still be valid in Section~\ref{solsec}.}
$$
\| u \|_{L^\infty(B_R(x_0))} \le C R^{- \frac{n}{p}} \| u \|_{L^p(B_{2 R}(x_0))} + \Tail(u; x_0, R) + d_1^{\frac{1}{p}} R^{\frac{n \varepsilon}{p}} + 2 \kappa.
$$
Moreover, the constants can be chosen as follows:
\begin{enumerate}[label={$\bullet$},leftmargin=*]
\item if~$d_2 = 0$, then
$$
C = C(n, s, p, \Lambda), \, \, R_0 = \frac{r_0}{2}, \, \, \varepsilon = \frac{s p}{n} \, \mbox{ and } \, \kappa = 0;
$$
\item if~$d_2 > 0$ and~$1 \le q \le p$, then
$$
C = C(n, s, p, \Lambda, d_2), \, \, R_0 = \min \left\{ 1, \frac{r_0}{2} \right\}, \, \, \varepsilon = \frac{s p}{n} \, \mbox{ and } \, \kappa = 1;
$$
\item if~$d_2 > 0$ and~$p < q < p^*_s$, then
\begin{align*}
& C = C(n, s, p, q, \Lambda, d_2), \, \, R_0 = R_0(n, s, p, q, \Lambda, d_2, r_0, \| u \|_{L^{p^*_\sigma}(\Omega)}), \\
& \varepsilon = 1 - \frac{q}{p^*_\sigma} \, \mbox{ and  } \, \kappa = 0, \, \mbox{ for some } \, \sigma = \sigma(n, s, p, q) \in (0, s).
\end{align*}
\end{enumerate}
When~$n > s p$ we can even take~$\sigma = s$, while when~$n > p$ both constants~$C$ and~$R_0$ do not blow up as~$s \rightarrow 1^-$.
\end{theorem}

\begin{theorem}[\bfseries H\"older continuity of minimizers] \label{minholdthm}
\textcolor{white}{ }\\
Let~$n \in \N$,~$0 < s_0 \le s < 1$ and~$p > 1$ be such that~$n \ge s p$. Let~$\Omega$ be an open bounded subset of~$\R^n$. Assume that~$K$ satisfies hypotheses~\eqref{Ksimm},~\eqref{Kell} and that~$F$ is locally bounded in~$u \in \R$, uniformly w.r.t~$x \in \Omega$. If~$u$ is a minimizer of~$\E$ in~$\Omega$, then,~$u \in C^\alpha_\loc(\Omega)$ for some~$\alpha \in (0, 1)$. In particular, there exists a constant~$C \ge 1$ such that, given any~$x_0 \in \Omega$ and~$0 < R \le \min \{ r_0 / 2, \dist \left(x_0, \partial \Omega \right) \} / 16$, it holds
$$
[ u ]_{C^\alpha(B_R(x_0))} \le \frac{C}{R^\alpha} \left( \| u \|_{L^\infty(B_{4 R}(x_0))} + \Tail(u; x_0, 4 R) + R^s \| F(\cdot, u) \|_{L^\infty(B_{8 R}(x_0))}^{1 / p} \right).
$$
The constants~$\alpha$ and~$C$ depend only on~$n$,~$s_0$,~$p$ and~$\Lambda$.
\end{theorem}

\begin{theorem}[\bfseries Harnack inequality for minimizers] \label{minharthm}
\textcolor{white}{ }\\
Let~$n \in \N$,~$s \in (0, 1)$ and~$p > 1$. Let~$\Omega$ be an open bounded subset of~$\R^n$. Assume that~$K$ satisfies hypotheses~\eqref{Ksimm},~\eqref{Kell2} and that~$F$ is locally bounded in~$u \in \R$, uniformly w.r.t~$x \in \Omega$. Let~$u$ be a minimizer of~$\E$ in~$\Omega$ such that~$u \ge 0$ in~$\Omega$. Then, for any~$x_0 \in \Omega$ and~$0 < R < \dist(x_0, \partial \Omega) / 2$, it holds
$$
\sup_{B_R(x_0)} u \le C \left( \inf_{B_R(x_0)} u + \Tail(u_-; x_0, R) + R^s \| F(\cdot, u) \|_{L^\infty(B_{2 R}(x_0))}^{1 / p} \right),
$$
for some~$C \ge 1$ depending only on~$n$,~$s$,~$p$ and~$\Lambda$. When~$n \notin \{1, p \}$, the constant~$C$ does not blow up as~$s \rightarrow 1^-$.
\end{theorem}

Theorems~\ref{minboundthm}-\ref{minharthm} will be proved by showing that the minimizers of~$\E$ belong to a fractional~De~Giorgi class, so that Theorems~\ref{DGboundthm},\ref{DGholdthm} and~\ref{DGharthm} may be applied to them.

Before proceeding to the proof of this inclusion, we define the notions of sub- and superminimizers as one-sided generalizations of the minimizers introduced in Definition~\ref{mindef}.

\begin{definition}
Let~$\Omega \subset \R^n$ be a bounded measurable set. A function~$u \in \W^{s, p}(\Omega)$ is said to be a~\emph{subminimizer} (\emph{superminimizer}) of~$\E$ in~$\Omega$ if
\begin{equation} \label{minine}
\E(u; \Omega) \le \E(u + \varphi; \Omega),
\end{equation}
for any non-positive (non-negative) measurable function~$\varphi: \R^n \to \R$ supported inside~$\Omega$.
\end{definition}

Recalling Definition~\ref{mindef}, minimizers are, in particular, sub- and superminimizers. Conversely, it is not hard to see that if~$u$ is at the same time a sub- and a superminimizer, then it is a minimizer as well.

Even if not immediately obvious, the minimality property introduced above behaves well with respect to set inclusion. That is, if~$\Omega' \subset \Omega \subset \R^n$ and~$u$ is a sub- or superminimizer in~$\Omega$, then~$u$ is also a sub- or superminimizer in~$\Omega'$ (see e.g.~\cite[Remark~1.2]{CV15} for a more detailed explanation in the case of minimizers).

Furthermore, it is easy to check that~$u$ is a subminimizer (superminimizer) for~$\E$ if and only if~$-u$ is a superminimizer (subminimizer) for the energy having the same interaction term and potential determined by~$\widetilde{F}(x, u) = F(x, -u)$. Noticing that~$\widetilde{F}$ satisfies the same growth assumption~\eqref{Fbounds} of~$F$, this allows us to focus on subminimizers only.

\smallskip

We can now state the following result, where we prove that the minimizers of~$\E$ belong to a suitable fractional De Giorgi class.

Unless otherwise specified, in what follows we assume~$n \in \N$,~$s \in (0, 1)$,~$p > 1$ to be given parameters, and that~$K$ and~$F$ satisfy~\eqref{Ksimm},~\eqref{Kell} and~\eqref{Fbounds}, respectively.

\begin{proposition} \label{minareDGprop}
Let~$u$ be a subminimizer of~$\E$ in a bounded open set~$\Omega \subset \R^n$. Then, there exist four constants~$R_0 \in (0, r_0 / 2]$,~$k_0 \in [-\infty, 1]$,~$H \ge 1$ and~$\varepsilon \in (0, sp / n]$, such that
\begin{equation} \label{minareDGine}
\begin{aligned}
& [(u - k)_+]_{W^{s, p}(B_r(x_0))}^p + \int_{B_r(x_0)} (u(x) - k)_+ \left[ \int_{B_{2 R_0}(x)} \frac{(u(y) - k)_-^{p - 1}}{|x - y|^{n + s p}} \, dy \right] dx \\
& \hspace{15pt} \le \frac{H}{1 - s} \Bigg[ \left( d_1 + \frac{|k|^p}{R^{n \varepsilon}} \right) |A^+(k, x_0, R)|^{1 - \frac{s p}{n} + \varepsilon} + \frac{R^{(1 - s) p}}{(R - r)^p} \| (u - k)_+ \|_{L^p(B_R(x_0))}^p \\
& \hspace{15pt} \quad + \frac{R^{n + s p}}{(R - r)^{n + s p}} \| (u - k)_+ \|_{L^1(B_R(x_0))} \overline{\Tail}((u - k)_+; x_0, r)^{p - 1} \Bigg],
\end{aligned}
\end{equation}
for any~$x_0 \in \Omega$,~$0 < r < R \le \min \{ R_0, \dist (x_0, \partial \Omega) \}$ and~$k \ge k_0$. Consequently,~$u$ belongs to the following fractional De~Giorgi classes:
\begin{enumerate}[label={$\bullet$},leftmargin=*]
\item if~$d_2 = 0$, then
$$
u \in \DG^{s, p}_+ \left( \Omega; d_1^{\frac{1}{p}}, H, -\infty, \frac{ s p }{n}, 0, \frac{r_0}{2} \right),
$$
with~$H = H(n, p, \Lambda)$;
\item if~$d_2 > 0$ and~$1 \le q \le p$, then
$$
u \in \DG^{s, p}_+ \left( \Omega; d_1^{\frac{1}{p}}, H, 1, \frac{ s p }{n}, 0, \min \left\{ 1, \frac{r_0}{2} \right\} \right),
$$
with~$H = H(n, p, \Lambda, d_2)$;
\item if~$d_2 > 0$,~$n \ge s p$ and~$p < q < p^*_s$, then
$$
u \in \DG^{s, p}_+ \left( \Omega; d_1^{\frac{1}{p}}, H, 0, 1 - \frac{q}{p^*_\sigma}, 0, R_0 \right),
$$
with~$H = H(n, p, q, \Lambda, d_2)$,~$R_0 = R_0(n, s, p, q, \Lambda, d_2, r_0, \| u \|_{L^{p^*_\sigma}(\Omega)})$ and for some constant~$\sigma = \sigma(n, s, p, q) \in (0, s)$. When~$n > sp$ we can even take~$\sigma = s$, while when~$n > p$ the constant~$R_0$ does not blow up as~$s \rightarrow 1^-$.
\end{enumerate}
An analogous statement holds for superminimizers and the classes~$\DG_-^{s, p}$.
\end{proposition}
\begin{proof}
Without loss of generality, we suppose~$x_0 = 0$.

Let~$r \le \rho < \tau \le R$ and consider a cut-off function~$\eta \in C^\infty_0(\R^n)$ satisfying~$0 \le \eta \le 1$ in~$\R^n$,~$\supp (\eta) = B_{(\tau + \rho) / 2}$,~$\eta = 1$ in~$B_\rho$ and~$|\nabla \eta| \le 4 / (\tau - \rho)$ in~$\R^n$. For any fixed~$k \ge 0$, we write~$w_\pm := (u - k)_\pm$ and choose~$v := u - \eta w_+$ as a test function in~\eqref{minine}. Notice that~$\supp(\eta w_+) = A^+(k, (\tau + \rho) / 2)$. Hence, and recalling~\eqref{Fbounds} as well as notation~\eqref{dmudef}, we get
\begin{align*}
0 & \le \frac{1}{2 p} \iint_{\C_{B_\tau}} \Big[ |v(x) - v(y)|^p - |u(x) - u(y)|^p \Big] d\mu \\
& \quad + \int_{A^+ \left( k, \frac{\tau + \rho}{2} \right)} \Big[ 2 d_1 + d_2 \left( |u(x)|^q + |v(x)|^q \right) \Big] \, dx.
\end{align*}
Recalling~\eqref{Ksimm} and~\eqref{COmega}, this is equivalent to
\begin{equation} \label{minareDG1}
\begin{aligned}
0 & \le \int_{B_\tau} \int_{B_\tau} \Big[ |v(x) - v(y)|^p - |u(x) - u(y)|^p \Big] \, d\mu \\
& \quad + 2 \int_{B_\tau} \int_{\R^n \setminus B_\tau} \Big[ |v(x) - v(y)|^p - |u(x) - u(y)|^p \Big] \, d\mu \\
& \quad + 2 p \int_{A^+ \left( k, \frac{\tau + \rho}{2} \right)} \Big[ 2 d_1 + d_2 \left( |u(x)|^q + |v(x)|^q \right) \Big] \, dx.
\end{aligned}
\end{equation}

We begin by dealing with the two terms involving double integrals. The following facts hold true:
\begin{align}
\label{minareDGprop-1}
& \mbox{if either } x \notin A^+(k) \mbox{ or } y \notin A^+(k), \mbox{ then } |v(x) - v(y)|^p \le |u(x) - u(y)|^p,\\
\label{minareDGprop1}
& \mbox{if } x, y \in A^+(k, \rho), \mbox{ then } |v(x) - v(y)|^p - |u(x) - u(y)|^p = - |w_+(x) - w_+(y)|^p,\\
\label{minareDGprop3}
\begin{split}
& \mbox{if } x \in A^+(k, \rho) \mbox{ and } y \notin A^+(k), \mbox{ then} \\
& \hspace{10pt} |v(x) - v(y)|^p - |u(x) - u(y)|^p \le - \frac{1}{2} \, |w_+(x) - w_+(y)|^p - \frac{p}{2} \, w_-(y)^{p - 1} w_+(x),
\end{split}\\
\label{minareDGprop2}
\begin{split}
& \mbox{if } x, y \in A^+(k), \mbox{ then} \\
& \hspace{10pt} |v(x) - v(y)|^p \le C \left[ (1 - \eta(x))^p |w_+(x) - w_+(y)|^p + w_+(y)^p \frac{|x - y|^p}{(\tau - \rho)^p} \right],
\end{split}
\end{align}
for some constant~$C \ge 1$ depending only on~$p$.

Indeed,~\eqref{minareDGprop-1} and~\eqref{minareDGprop1} are immediate consequences of the definition of~$v$. On the other hand, for general~$x, y \in A^+(k)$ we may compute
\begin{align*}
|v(x) - v(y)|^p & = |(1 - \eta(x)) w_+(x) - (1 - \eta(y)) w_+(y)|^p \\
& \le 2^{p - 1} \Big[ (1 - \eta(x))^p |w_+(x) - w_+(y)|^p + w_+(y)^p |\eta(x) - \eta(y)|^p \Big],
\end{align*}
and~\eqref{minareDGprop2} follows in view of the properties of~$\eta$. Finally, to get~\eqref{minareDGprop3} we observe that for~$x \in A^+(k, \tau)$ and~$y \notin A^+(k)$,
\begin{align*}
|v(x) - v(y)|^p - |u(x) - u(y)|^p = \left( (1 - \eta(x)) w_+(x) + w_-(y) \right)^p - \left( w_+(x) + w_-(y) \right)^p,
\end{align*}
and then apply Lemma~\ref{numestlem1} (with~$\theta = 1/2$).

Thanks to properties~\eqref{minareDGprop-1}-\eqref{minareDGprop2}, we are in position to estimate the first double integral in~\eqref{minareDG1}. By~\eqref{minareDGprop-1},~\eqref{minareDGprop1},~\eqref{minareDGprop3}, the symmetry assumption~\eqref{Ksimm} on~$K$ and the properties of~$w_\pm$, we get
\begin{align*}
& \int_{B_\rho} \int_{B_\rho} \Big[ |v(x) - v(y)|^p - |u(x) - u(y)|^p \Big] \, d\mu \\
& \hspace{50pt} \le - \frac{1}{2} \int_{B_\rho} \int_{B_\rho} |w_+(x) - w_+(y)|^p \, d\mu - \frac{p}{2} \int_{B_\rho} \int_{B_\rho} w_-(y)^{p - 1} w_+(x) \, d\mu.
\end{align*}
On the other hand, after applying~\eqref{minareDGprop-1},~\eqref{minareDGprop3},~\eqref{minareDGprop2} and~\eqref{Kell}, an immediate computation reveals that
\begin{align*}
& \iint_{B_\tau^2 \setminus B_\rho^2} \Big[ |v(x) - v(y)|^p - |u(x) - u(y)|^p \Big] \, d\mu \\
& \hspace{40pt} \le C \left[ \iint_{B_\tau^2 \setminus B_\rho^2} |w_+(x) - w_+(y)|^p \, d\mu + \frac{R^{(1 - s) p}}{(\tau - \rho)^p} \int_{B_R} w_+(x)^p \, dx \right] \\
& \hspace{40pt} \quad - \frac{p}{2} \int_{B_\rho} w_+(x) \left( \int_{B_\tau \setminus B_\rho} w_-(y)^{p - 1} K(x, y) \, dy \right) dx,
\end{align*}
with~$C$ now depending also on~$n$ and~$\Lambda$. The combination of these last two estimates and another use of~\eqref{Kell} yield
\begin{equation} \label{minareDG2}
\begin{aligned}
&  \int_{B_\tau} \int_{B_\tau} \Big[ |v(x) - v(y)|^p - |u(x) - u(y)|^p \Big] \, d\mu \\
& \hspace{20pt} \le - \frac{1 - s}{C} \left[ [w_+]_{W^{s, p}(B_\rho)}^p  + \int_{B_\rho} w_+(x) \left( \int_{B_\tau} \frac{w_-(y)^{p - 1}}{|x - y|^{n + s p}} \, dy \right) dx \right] \\
& \hspace{20pt} \quad + C \left[ (1 - s) \iint_{B_\tau^2 \setminus B_\rho^2} \frac{|w_+(x) - w_+(y)|^p}{|x - y|^{n + s p}} \, dx dy + \frac{R^{(1 - s) p}}{(\tau - \rho)^p} \| w_+ \|_{L^p(B_R)}^p \right].
\end{aligned}
\end{equation}

Now we deal with the second double integral in~\eqref{minareDG1}. We claim that
\begin{align}
\label{minareDGprop0}
& \mbox{if } x, y \notin A^+ \left( k, \frac{\tau + \rho}{2} \right), \mbox{ then } |v(x) - v(y)|^p = |u(x) - u(y)|^p,\\
\label{minareDGprop4}
\begin{split}
& \mbox{if } x \in A^+(k) \mbox{ and } y \in A^+(k) \setminus B_\tau, \mbox{ then} \\
& \hspace{10pt} |v(x) - v(y)|^p - |u(x) - u(y)|^p \le p w_+(y)^{p - 1} w_+(x).
\end{split}
\end{align}
Since~\eqref{minareDGprop0} is an obvious consequence of the fact that~$\supp(\eta w_+) = A(k, (\tau + \rho) / 2)$, we can concentrate on~\eqref{minareDGprop4}. Letting~$x$ and~$y$ as prescribed, we have
\begin{align*}
|v(x) - v(y)|^p - |u(x) - u(y)|^p = |(1 - \eta(x)) w_+(x) - w_+(y)|^p - |w_+(x) - w_+(y)|^p.
\end{align*}
The inequality stated in~\eqref{minareDGprop4} can then be obtained by employing Lemma~\ref{numestlem2}.

In view of~\eqref{Kell},~\eqref{minareDGprop0},~\eqref{minareDGprop4}, again~\eqref{minareDGprop-1},~\eqref{minareDGprop3} and the fact that
$$
|x - y| \ge |y| - |x| \ge \left( 1 - \frac{\tau + \rho}{2 \tau} \right) |y| = \frac{\tau - \rho}{2 \tau} |y| \ge \frac{\tau - \rho}{2 R} |y|,
$$
for any~$x \in B_{(\tau + \rho) / 2}$ and~$y \notin B_\tau$, it is not hard to see that
%
\begin{align*}
& \int_{B_\tau} \int_{\R^n \setminus B_\tau} \Big[ |v(x) - v(y)|^p - |u(x) - u(y)|^p \Big] \, d\mu \\
& \hspace{50pt} \le C \frac{R^{n + s p}}{(\tau - \rho)^{n + s p}} \| w_+ \|_{L^1(B_R)} \overline{\Tail}(w_+; r)^{p - 1} \\
& \hspace{50pt} \quad - \frac{1 - s}{C} \int_{B_\rho} w_+(x) \left( \int_{B_{r_0}(x) \setminus B_\tau} \frac{w_-(y)^{p - 1}}{|x - y|^{n + s p}} \, dy \right) dx.
\end{align*}
By putting together this last inequality with~\eqref{minareDG1} and~\eqref{minareDG2}, we obtain
\begin{equation} \label{minareDG3}
\begin{aligned}
& [w_+]_{W^{s, p}(B_\rho)}^p + \int_{B_\rho} w_+(x) \left( \int_{B_{r_0}(x)} \frac{w_-(y)^{p - 1}}{|x - y|^{n + s p}} \, dy \right) dx \\
& \hspace{15pt} \le C \iint_{B_\tau^2 \setminus B_\rho^2} \frac{|w_+(x) - w_+(y)|^p}{|x - y|^{n + s p}} \, dx dy \\
& \hspace{15pt} \quad + \frac{C}{1 - s} \Bigg[ \frac{R^{n + s p}}{(\tau - \rho)^{n + s p}} \| w_+ \|_{L^1(B_R)} \overline{\Tail}(w_+; r)^{p - 1} + \frac{R^{(1 - s) p}}{(\tau - \rho)^p} \| w_+ \|_{L^p(B_R)}^p \\
& \hspace{15pt} \quad + d_1 |A^+(k, R)| + d_2 \int_{A^+(k, \tau)} \left( |u(x)|^q + |v(x)|^q \right) \, dx \Bigg],
\end{aligned}
\end{equation}
for some~$C \ge 1$ depending only on~$n$,~$p$ and~$\Lambda$.

We set
\begin{equation} \label{Phidef}
\Phi(t) := [w_+]_{W^{s, p}(B_t)}^p + \int_{B_t} w_+(x) \left( \int_{B_{r_0}(x)} \frac{w_-(y)^{p - 1}}{|x - y|^{n + s p}} \, dy \right) dx + \Psi(t),
\end{equation}
with
$$
\Psi(t) := \begin{dcases}
0 & \mbox{if } d_2 = 0, \mbox{ or } d_2 > 0, \, n \ge sp \mbox{ and } 1 \le q \le p,\\
\frac{d_2}{1 - s} \| u \|_{L^q(A^+(k, t))}^q & \mbox{if } d_2 > 0, \, n \ge s p \mbox{ and } p < q < p^*_s,
\end{dcases}
$$
for any~$0 < t \le R$, and we claim that
\begin{equation} \label{minareDG4}
\begin{aligned}
\Phi(\rho) & \le C_\flat \left[ \Phi(\tau) - \Phi(\rho) \right] + \frac{C_\flat}{1 - s} \Bigg[ \frac{R^{n + s p}}{(\tau - \rho)^{n + s p}} \| w_+ \|_{L^1(B_R)} \overline{\Tail}(w_+; r)^{p - 1} \\
& \quad + (1 + d_2) \frac{R^{(1 - s) p}}{(\tau - \rho)^p} \| w_+ \|_{L^p(B_R)}^p + (d_1 + d_2 k^p R^{- n \varepsilon}) |A^+(k, R)|^{1 - \frac{s p}{n} + \varepsilon} \Bigg],
\end{aligned}
\end{equation}
for some~$C_\flat \ge 1$ possibly depending on~$n$,~$p$,~$q$,~$\Lambda$, and~$\varepsilon \in (0, sp / n]$.

Note that, in the case of~$d_2 = 0$, claim~\eqref{minareDG4} follows immediately from~\eqref{minareDG3}, with~$\varepsilon = s p / n$. To check the validity of~\eqref{minareDG4} when~$d_2 > 0$, we consider separately the two possibilities~$1 \le q \le p$ and~$p < q < p^*_s$ (with~$n \ge sp$ in the latter case).

Suppose that~$1 \le q \le p$. Assuming~$k \ge 1$, for~$x \in A^+(k)$ we have
$$
u(x) > k \ge 1 \quad \mbox{and} \quad v(x) = (1 - \eta(x)) w_+(x) + k \ge k \ge 1.
$$
Therefore,
\begin{align*}
|u(x)|^q + |v(x)|^q & \le |u(x)|^p + |v(x)|^p = |w_+(x) + k|^p + |(1 - \eta(x)) w_+(x) + k|^p \\
& \le 2^p \left( w_+(x)^p + k^p \right).
\end{align*}
By taking~$R \le 1$, it follows that
\begin{align*}
\int_{A^+(k, \tau)} \left( |u(x)|^q + |v(x)|^q \right) \, dx & \le 2^p \int_{A^+(k, \tau)} \left( w_+(x)^p + k^p \right) dx \\
& \le 2^p \left[ \frac{R^{(1 - s) p}}{(\tau - \rho)^p} \| w_+ \|_{L^p(B_R)}^p + k^p R^{- s p} |A^+(k, R)| \right].
\end{align*}
Using this with~\eqref{minareDG3}, we easily deduce~\eqref{minareDG4}, again with~$\varepsilon = sp / n$.

We now deal with the case of~$n \ge s p$ and~$p < q < p^*_s$. Let
\begin{equation} \label{sigmamindef}
\sigma := \begin{cases}
s & \quad \mbox{if } n > s p, \\
\max \left\{ 2 s - 1, \dfrac{(2 q - p) n}{2 p q} \right\} & \quad \mbox{if } n = s p,
\end{cases}
\end{equation}
and notice that~$n > \sigma p$ and~$q < p^*_\sigma$ in both cases~$n > sp$ and~$n = sp$. Also set
\begin{equation} \label{epsigmamindef}
\varepsilon_\sigma := 1 - \frac{q}{p^*_\sigma} \in \left( 0, \frac{s p}{n} \right).
\end{equation}
We add to both sides of~\eqref{minareDG3} the quantity~$d_2 (1 - s)^{-1} \| u \|_{L^q(A^+(k, \rho))}^q$. We obtain
\begin{equation} \label{minareDG5}
\begin{aligned}
\Phi(\rho) & \le \frac{C_1}{1 - s} \Bigg[ (1 - s) \iint_{B_\tau^2 \setminus B_\rho^2} \frac{|w_+(x) - w_+(y)|^p}{|x - y|^{n + s p}} \, dx dy \\
& \quad + \frac{R^{n + s p}}{(\tau - \rho)^{n + s p}} \| w_+ \|_{L^1(B_R)} \overline{\Tail}(w_+; r)^{p - 1} + \frac{R^{(1 - s) p}}{(\tau - \rho)^p} \| w_+ \|_{L^p(B_R)}^p \\
& \quad + d_1 |A^+(k, R)| + d_2 \int_{A^+(k, \tau)} \left( |u(x)|^q + |v(x)|^q \right) \, dx \Bigg],
\end{aligned}
\end{equation}
with~$C_1 \ge 1$ depending on~$n$,~$p$ and~$\Lambda$. We proceed to estimate the last term of~\eqref{minareDG5}. By restricting ourselves to~$k \ge 0$, on~$A^+(k)$ we have
\begin{equation} \label{u+vest}
\begin{aligned}
|u|^q + |v|^q & = |(1 - \eta) u + \eta w_+ + \eta k|^q + |(1 - \eta) u + \eta k|^q \\
& \le 3^q \Big( (1 - \eta)^q |u|^q + (\eta w_+)^q + k^q \Big).
\end{aligned}
\end{equation}
On the one hand, it holds
$$
k^{p^*_{\sigma}} |A^+(k, R)| \le \int_{A^+(k, R)} u(x)^{p^*_\sigma} \, dx \le \int_{\Omega} |u(x)|^{p^*_\sigma} \, dx, 
$$
and thus
\begin{equation} \label{minarekest}
\begin{aligned}
3^q \int_{A^+(k, \tau)} k^q \, dx & = 3^q k^p \left( k^{p^*_\sigma} |A^+(k, R)| \right)^{\frac{q - p}{p^*_\sigma}} |A^+(k, R)|^{1 - \frac{q - p}{p^*_\sigma}} \\
& \le 3^q \| u \|_{L^{p^*_\sigma}(\Omega)}^{q - p} k^p |B_R|^{\frac{(s - \sigma) p}{n}} |A^+(k, R)|^{1 - \frac{s p}{n} + \varepsilon_\sigma} \\
& \le k^p R^{- n \varepsilon_\sigma} |A^+(k, R)|^{1 - \frac{sp}{n} + \varepsilon_\sigma},
\end{aligned}
\end{equation}
provided
$$
R \le R_0 \le \left[ \frac{1}{3^{q} |B_1|^{\frac{(s - \sigma) p}{n}} \| u \|_{L^{p^*_\sigma}(\Omega)}^{q - p}} \right]^{\frac{1}{n \varepsilon_\sigma + (s - \sigma) p}}.
$$
On the other hand, with the help of H\"older's inequality and Corollary~\ref{sobinecor}, we compute
\begin{align*}
\int_{A^+(k, \tau)} (\eta(x) w_+(x))^q \, dx & \le |B_R|^{\varepsilon_\sigma} \left[ \int_{B_\tau} w_+(x)^{p^*_\sigma} \, dx \right]^{\frac{q - p}{p^*_\sigma}} \left[ \int_{B_\tau} (\eta(x) w_+(x))^{p^*_\sigma} \, dx \right]^{\frac{p}{p^*_\sigma}} \\
& \le C_2 \frac{(1 - \sigma) R^{n \varepsilon_\sigma}}{(n - \sigma p)^{p - 1}} \| u \|_{L^{p^*_\sigma}(\Omega)}^{q - p} \left[ [w_+]_{W^{\sigma, p}(B_\tau)}^p + \frac{\| w_+ \|_{L^p(B_R)}^p}{(\tau - \rho)^{\sigma p}} \right],
\end{align*}
with~$C_2 \ge 1$ depending only on~$n$ and~$p$. Notice that~$1 - \sigma \le 2 (1 - s)$, by definition of~$\sigma$. By taking
$$
R \le R_0 \le \min \left\{ \frac{1}{2}, \left[ \frac{(n - \sigma p)^{p - 1}}{4 C_1 C_2 d_2 3^q \| u \|_{L^{p^*_\sigma}(\Omega)}^{q - p}} \right]^{\frac{1}{n \varepsilon_\sigma}} \right\},
$$
in the previous estimate and using Lemma~\ref{sobinclem0} (with~$\delta = 1$), we get
\begin{align*}
\frac{3^q C_1 d_2}{1 - s} \int_{A^+(k, \tau)} (\eta(x) w_+(x))^q \, dx & \le \frac{1}{2} \left( [w_+]_{W^{\sigma, p}(B_\tau)}^p + \frac{\| w_+ \|_{L^p(B_R)}^p}{(\tau - \rho)^{\sigma p}} \right) \\
& \le \frac{1}{2} \left( [w_+]_{W^{s, p}(B_\tau)}^p + \frac{R^{(1 - s) p}}{(\tau - \rho)^{p}} \| w_+ \|_{L^p(B_R)}^p \right).
\end{align*}
By this,~\eqref{u+vest},~\eqref{minarekest} and the definition~\eqref{Phidef} of~$\Phi$, we see that
\begin{align*}
&\frac{C_1 d_2}{1 - s} \int_{A^+(k, \tau)} \left( |u(x)|^q + |v(x)|^q \right) \, dx \\
& \hspace{15pt} \le \frac{1}{2} \left( [w_+]_{W^{s, p}(B_\tau)}^p + \frac{R^{(1 - s) p}}{(\tau - \rho)^{p}} \| w_+ \|_{L^p(B_R)}^p \right) \\
& \hspace{15pt} \quad + \frac{C_1 d_2}{1 - s} \left( 3^{q} \int_{B_\tau \setminus B_\rho} |u(x)|^q \, dx + \frac{k^p}{R^{n \varepsilon_\sigma}} |A^+(k, R)|^{1 - \frac{sp}{n} + \varepsilon_\sigma} \right) \\
& \hspace{15pt} \le \frac{1}{2} \Phi(\rho) + C \left[ \Phi(\tau) - \Phi(\rho) + \frac{R^{(1 - s) p}}{(\tau - \rho)^p} \| w_+ \|_{L^p(B_R)}^p + \frac{d_2 k^p |A^+(k, R)|^{1 - \frac{sp}{n} + \varepsilon_\sigma}}{(1- s) R^{n \varepsilon_\sigma}} \right],
\end{align*}
for some~$C \ge 1$ depending on~$n$,~$p$,~$q$ and~$\Lambda$. This and~\eqref{minareDG5} yield
\begin{align*}
\Phi(\rho) & \le \frac{1}{2} \Phi(\rho) + C \left[ \Phi(\tau) - \Phi(\rho) \right] + \frac{C}{1 - s} \Bigg[ \frac{R^{n + s p}}{(\tau - \rho)^{n + s p}} \| w_+ \|_{L^1(B_R)} \overline{\Tail}(w_+; r)^{p - 1} \\
& \quad + \frac{R^{(1 - s) p}}{(\tau - \rho)^p} \| w_+ \|_{L^p(B_R)}^p + (d_1 + d_2 k^p R^{- n \varepsilon_\sigma}) |A^+(k, R)|^{1 - \frac{s p}{n} + \varepsilon_\sigma} \Bigg],
\end{align*}
at least if
$$
R \le R_0 \le |B_1|^{-\frac{1}{n}},
$$
since~$\varepsilon_\sigma \le sp / n$. After subtracting to both sides the term~$\Phi(\rho) / 2$, we reach~\eqref{minareDG4}, with~$\varepsilon = \varepsilon_\sigma$.

By adding the quantity~$C_\flat \Phi(\rho)$ to both sides of~\eqref{minareDG4} and dividing by~$1 + C_\flat$, we get
\begin{align*}
\Phi(\rho) & \le \gamma \Phi(\tau) + \frac{\gamma}{1 - s} \Bigg[ \frac{R^{n + s p}}{(\tau - \rho)^{n + s p}} \| w_+ \|_{L^1(B_R)} \overline{\Tail}(w_+; r)^{p - 1} \\
& \quad + (1 + d_2) \frac{R^{(1 - s) p}}{(\tau - \rho)^p} \| w_+ \|_{L^p(B_R)}^p + (d_1 + d_2 k^p R^{- n \varepsilon}) |A^+(k, R)|^{1 - \frac{s p}{n} + \varepsilon} \Bigg],
\end{align*}
for any~$0 < r \le \rho < \tau \le R$, with~$\gamma = C_\flat / (1 + C_\flat) \in (0, 1)$. Inequality~\eqref{minareDGine} then follows after an application of Lemma~\ref{induclem}.
\end{proof}

In view of Proposition~\ref{minareDGprop} and the three main results of Section~\ref{DGsec}, the validity of Theorems~\ref{minboundthm}-\ref{minharthm} is easily deduced.

\begin{proof}[Proofs of Theorems~\ref{minboundthm},~\ref{minholdthm} and~\ref{minharthm}]
The local boundedness of~$u$ claimed in Theorem~\ref{minboundthm} is an immediate consequence of Proposition~\ref{minareDGprop} and Theorem~\ref{DGboundthm} combined.

To check that also Theorem~\ref{minholdthm} is true, we first observe that, since we already know that~$u$ is locally bounded, we may view it as a minimizer of~$\E$ in~$B_{8 R}(x_0) \subset \subset \Omega$, with potential~$F$ now fulfilling~\eqref{Fbounds} with~$d_1 = \| F(\cdot, u) \|_{L^\infty(B_{8 R}(x_0))}$ and~$d_2 = 0$. Theorem~\ref{minholdthm} then follows by a straightforward application of Proposition~\ref{minareDGprop} and Theorem~\ref{DGholdthm}.

Notice that the quantities~$\alpha$ and~$C$ are both claimed to be independent of~$s$ in the statement of Theorem~\ref{minholdthm}, even when~$n = 1$. Apparently, this is in contradiction with Theorem~\ref{DGholdthm}. However, going through the proofs of Lemma~\ref{growthlem} it can be checked that~$\alpha$ and~$C$ can be chosen to be uniform in~$s$, as long as~$s$ is far from~$1$. This is the case here, since, by assumption,~$s \le 1 / p < 1$, if~$n = 1$.

In an analogous way, we deduce Theorem~\ref{minharthm} from Theorem~\ref{DGharthm}. For the case~$n = 1$, recall the opening remark of Subsection~\ref{harsubsec}.
\end{proof}

\section{Applications to solutions} \label{solsec}

In the present section, we prove Theorems~\ref{boundmainthm},~\ref{holmainthm} and~\ref{harmainthm} in the case of solutions. Similarly to what we did in Section~\ref{minsec}, we show that weak solutions of the integral equation~\eqref{Lu=f} are in a fractional~De~Giorgi class, so that their boundedness, H\"older continuity and Harnack property are readily deduced from Theorems~\ref{DGboundthm},~\ref{DGholdthm} and~\ref{DGharthm}, respectively. The detailed statements of these results are reported here below.

\begin{theorem}[\bfseries Local boundedness of solutions] \label{solboundthm}
\textcolor{white}{ }\\
Let~$n \in \N$,~$s \in (0, 1)$ and~$p > 1$ be such that~$n \ge s p$. Let~$\Omega$ be an open bounded subset of~$\R^n$. Assume that~$K$ and~$f$ respectively satisfy hypotheses~\eqref{Ksimm},~\eqref{Kell} and~\eqref{fbounds}. If~$u$ is a weak solution of~\eqref{Lu=f} in~$\Omega$, then~$u \in L^\infty_\loc(\Omega)$. In particular, there exist four constants~$C \ge 1$,~$R_0 \in (0, r_0 / 2]$,~$\varepsilon \in (0, sp / n]$ and~$\kappa \in \{ 0, 1 \}$ such that, given any~$x_0 \in \Omega$ and~$0 < R \le \min \{ R_0, \dist \left( x_0, \partial \Omega \right) \} / 4$, it holds
$$
\| u \|_{L^\infty(B_R(x_0))} \le C R^{-\frac{n}{p}} \| u \|_{L^p(B_{2 R}(x_0))} + \Tail(u; x_0, R) + d_1^{\frac{1}{p - 1}} R^{\frac{n \varepsilon}{p} + \frac{s}{p - 1}} + 2 \kappa.
$$
Moreover, the constants can be chosen as follows:
\begin{enumerate}[label={$\bullet$},leftmargin=*]
\item if~$d_2 = 0$, then
$$
C = C(n, s, p, \Lambda), \, \, R_0 = \frac{r_0}{2}, \, \, \varepsilon = \frac{s p}{n} \, \mbox{ and } \, \kappa = 0;
$$
\item if~$d_2 > 0$ and~$1 \le q \le p$, then
$$
C = C(n, s, p, \Lambda, d_2), \, \, R_0 = \min \left\{ 1, \frac{r_0}{2} \right\}, \, \, \varepsilon = \frac{s p}{n} \, \mbox{ and } \, \kappa = 1;
$$
\item if~$d_2 > 0$ and~$p < q < p^*_s$, then
\begin{align*}
& C = C(n, s, p, q, \Lambda, d_2), \, \, R_0 = R_0(n, s, p, q, \Lambda, d_2, \| u \|_{L^{p^*_\sigma}(\Omega)}, r_0),\\
& \varepsilon = 1 - \frac{q}{p^*_\sigma} \, \mbox{ and } \, \kappa = 0, \, \mbox{ for some } \, \sigma = \sigma(n, s, p, q) \in (0, s).
\end{align*}
\end{enumerate}
When~$n > s p$ we can even take~$\sigma = s$, while when~$n > p$ both constants~$C$ and~$R_0$ do not blow up as~$s \rightarrow 1^-$.
\end{theorem}

\begin{theorem}[\bfseries H\"older continuity of solutions] \label{solholdthm}
\textcolor{white}{ }\\
Let~$n \in \N$,~$0 < s_0 \le s < 1$ and~$p > 1$ be such that~$n \ge s p$. Let~$\Omega$ be an open bounded subset of~$\R^n$. Assume that~$K$ satisfies hypotheses~\eqref{Ksimm},~\eqref{Kell} and that~$f$ is locally bounded in~$u \in \R$, uniformly w.r.t.~$x \in \Omega$. If~$u$ is a weak solution of~\eqref{Lu=f} in~$\Omega$, then~$u \in C^\alpha_\loc(\Omega)$ for some~$\alpha \in (0, 1)$. In particular, there exists a constant~$C \ge 1$ such that, given any~$x_0 \in \Omega$ and~$0 < R \le \min \{ r_0 / 2, \dist \left(x_0, \partial \Omega \right) \} / 16$, it holds
$$
[ u ]_{C^\alpha(B_R(x_0))} \le \frac{C}{R^\alpha} \left( \| u \|_{L^\infty(B_{4 R}(x_0))} + \Tail(u; x_0, 4 R) + R^{\frac{s p}{p - 1}} \| f(\cdot, u) \|_{L^\infty(B_{8 R}(x_0))}^{1 / (p - 1)} \right).
$$
The constants~$\alpha$ and~$C$ depend only on~$n$,~$s_0$,~$p$ and~$\Lambda$.
\end{theorem}

\begin{theorem}[\bfseries Harnack inequality for solutions] \label{solharthm}
\textcolor{white}{ }\\
Let~$n \in \N$,~$s \in (0, 1)$ and~$p > 1$. Let~$\Omega$ be an open bounded subset of~$\R^n$. Assume that~$K$ satisfies hypotheses~\eqref{Ksimm},~\eqref{Kell2} and that~$f$ is locally bounded in~$u \in \R$, uniformly w.r.t.~$x \in \Omega$. Let~$u$ be a weak solution of~\eqref{Lu=f} in~$\Omega$ such that~$u \ge 0$ in~$\Omega$. Then, for any~$x_0 \in \Omega$ and~$0 < R < \dist(x_0, \partial \Omega) / 2$, it holds
$$
\sup_{B_R(x_0)} u \le C \left( \inf_{B_R(x_0)} u + \Tail(u_-; x_0, R) + R^{\frac{s p}{p - 1}} \| f(\cdot, u) \|_{L^\infty(B_{2 R}(x_0))}^{1 / (p - 1)} \right),
$$
for some~$C \ge 1$ depending only on~$n$,~$s$,~$p$ and~$\Lambda$. When~$n \notin \{ 1, p \}$, the constant~$C$ does not blow up as~$s \rightarrow 1^-$.
\end{theorem}

Observe that the statements of Theorems~\ref{solboundthm}-\ref{solharthm} are almost completely identical to those of Theorems~\ref{minboundthm}-\ref{minharthm}, in Section~\ref{minsec}. The only notable difference resides in the diverse powers to which the quantities involving the potential and the forcing term---including~$d_1$ in Theorems~\ref{minboundthm} and~\ref{solboundthm}---are raised. Of course, this is coherent with the different homogeneity properties of energy~$\E$ and 
equation~\eqref{Lu=f}.

\smallskip

The notion of weak solutions of equation~\eqref{Lu=f} that we take into account has already been specified in Definition~\ref{soldef}. We now introduce the concepts of weak sub- and supersolutions.

\begin{definition}
Let~$\Omega \subset \R^n$ be a bounded open set and~$u \in \W^{s, p}(\Omega)$. The function~$u$ is said to be a~\emph{weak subsolution} (\emph{supersolution}) of~\eqref{Lu=f} in~$\Omega$ if
\begin{equation} \label{weaksolineq}
\begin{multlined}
\int_{\R^n} \int_{\R^n} |u(x) - u(y)|^{p - 2} \left( u(x) - u(y) \right) \left( \varphi(x) - \varphi(y) \right) K(x, y) \, dx dy \\
\le \int_{\R^n} f(x, u(x)) \varphi(x) \, dx,
\end{multlined}
\end{equation}
for any non-negative (non-positive)~$\varphi \in W^{s, p}(\R^n)$ with~$\supp(\varphi) \subset \subset \Omega$.
\end{definition}

Of course, a function~$u$ is a solution of~\eqref{Lu=f} if and only if it is at the same time a sub- and a supersolution. Moreover,~$u$ is a subsolution (supersolution) of~\eqref{Lu=f} if and only if~$-u$ is a supersolution (subsolution) of the same equation, but with right-hand side given by~$\tilde{f}(x, u) := f(x, -u)$. Hence, we can restrict ourselves to, say, subsolutions, as supersolutions may be dealt with in a specular way.

\smallskip

In the next proposition, we show the crucial step in the proof of our regularity results, namely that solutions of~\eqref{Lu=f} are contained in a fractional De~Giorgi class.

\begin{proposition} \label{solareDGprop}
Let~$u$ be a weak subsolution of~\eqref{Lu=f} in a bounded open set~$\Omega \subset \R^n$. Then, there exist four constants~$R_0 \in (0, r_0 / 2]$,~$k_0 \in [-\infty, 1]$,~$H \ge 1$ and~$\varepsilon \in (0, sp / n]$, such that
\begin{equation} \label{solareDGine}
\begin{aligned}
& [(u - k)_+]_{W^{s, p}(B_r(x_0))}^p + \int_{B_r(x_0)} \left[ (u(x) - k)_+ \int_{B_{2 R_0}(x)} \frac{(u(y) - k)_-^{p - 1}}{|x - y|^{n + s p}} \, dy \right] \\
& \hspace{60pt} \le \frac{H}{1 - s} \Bigg[ \left( R^{\frac{s p}{p - 1}} d_1^{\frac{p}{p - 1}} + \frac{|k|^p}{R^{n \varepsilon}} \right) |A^+(k, x_0, R)|^{1 - \frac{s p}{n} + \varepsilon} \\
& \hspace{60pt} \quad + \frac{R^{(1 - s) p}}{(R - r)^p} \| (u - k)_+ \|_{L^p(B_R(x_0))}^p \\
& \hspace{60pt} \quad + \frac{R^{n + s p}}{(R - r)^{n + s p}} \| (u - k)_+ \|_{L^1(B_R(x_0))} \overline{\Tail}((u - k)_+; x_0, r)^{p - 1} \Bigg],
\end{aligned}
\end{equation}
for any~$x_0 \in \Omega$,~$0 < r < R \le \min \{ R_0, \dist (x_0, \partial \Omega) \}$ and~$k \ge k_0$. Consequently,~$u$ belongs to the following fractional De~Giorgi classes:
\begin{enumerate}[label={$\bullet$},leftmargin=*]
\item if~$d_2 = 0$, then
$$
u \in \DG^{s, p}_+ \left( \Omega; d_1^{\frac{1}{p - 1}}, H, -\infty, \frac{ s p }{n}, \frac{sp}{p - 1}, \frac{r_0}{2} \right),
$$
with~$H = H(n, p, \Lambda)$;
\item if~$d_2 > 0$ and~$1 < q \le p$, then
$$
u \in \DG^{s, p}_+ \left( \Omega; d_1^{\frac{1}{p - 1}}, H, 1, \frac{ s p }{n}, \frac{sp}{p - 1}, \min \left\{ 1, \frac{r_0}{2} \right\} \right),
$$
with~$H = H(n, p, \Lambda, d_2)$;
\item if~$d_2 > 0$,~$n \ge sp$ and~$p < q < p^*_s$, then
$$
u \in \DG^{s, p}_+ \left( \Omega; d_1^{\frac{1}{p - 1}}, H, 0, 1 - \frac{q}{p^*_\sigma}, \frac{sp}{p - 1}, R_0 \right),
$$
with~$H = H(n, p, q, \Lambda, d_2)$,~$R_0 = R_0(n, s, p, q, \Lambda, d_2, r_0, \| u \|_{L^{p^*_\sigma}(\Omega)})$ and for some constant~$\sigma = \sigma(n, s, p, q) \in (0, s)$. When~$n > sp$ we can even take~$\sigma = s$, while when~$n > p$ the constant~$R_0$ does not blow up as~$s \rightarrow 1^-$.
\end{enumerate}
An analogous statement holds for weak supersolutions and the classes~$\DG_-^{s, p}$.
\end{proposition}
\begin{proof}
The argument is similar to that presented in Proposition~\ref{minareDGprop} for minimizers and the computations are simpler. Nevertheless, we report all the details, for the reader's convenience.

Clearly, we can suppose~$x_0 = 0$. Take~$r \le \rho < \tau \le R$ and let~$\eta \in C_0^\infty(\R^n)$ be a cut-off function satisfying~$0 \le \eta \le 1$ in~$\R^n$,~$\supp(\eta) = B_{(\tau + \rho) / 2}$,~$\eta = 1$ in~$B_\rho$ and~$|\nabla \eta| \le 4 / (\tau - \rho)$ in~$\R^n$. Fix~$k \in \R$ and write~$w_\pm := (u - k)_\pm$. By testing formulation~\eqref{weaksolineq} with~$\varphi := \eta^p w_+$ and recalling notation~\eqref{dmudef}, we obtain
\begin{equation} \label{solareDG1}
\iint_{\C_{B_\tau}} |u(x) - u(y)|^{p - 2} \left( u(x) - u(y) \right) \left( \varphi(x) - \varphi(y) \right) d\mu \le \int_{B_\tau} f(x, u(x)) \varphi(x) \, dx.
\end{equation}

We begin to study the term on the left-hand side of~\eqref{solareDG1}. In particular, we now deal with the contributions to the double integral coming from the set~$B_\tau \times B_\tau$. We claim that
\begin{align}
\label{solareDGprop1}
& \mbox{if } x, y \notin A^+(k), \mbox{ then } |u(x) - u(y)|^{p - 2} \left( u(x) - u(y) \right) \left( \varphi(x) - \varphi(y) \right) = 0,\\
\label{solareDGprop2}
\begin{split}
& \mbox{if } x \in A^+(k, \tau) \mbox{ and } y \in B_\tau \setminus A^+(k, \tau), \mbox{ then} \\
& \hspace{10pt} |u(x) - u(y)|^{p - 2} \left( u(x) - u(y) \right) \left( \varphi(x) - \varphi(y) \right) \\
& \hspace{15pt} \ge \min \{ 2^{p - 2}, 1 \} \Big[ |w_+(x) - w_+(y)|^p + w_-(y)^{p - 1} w_+(x) \Big] \eta(x)^p,
\end{split} \\
\label{solareDGprop3}
\begin{split}
& \mbox{if } x, y \in A^+(k, \tau), \mbox{ then} \\
& \hspace{10pt} |u(x) - u(y)|^{p - 2} \left( u(x) - u(y) \right) \left( \varphi(x) - \varphi(y) \right) \\
& \hspace{15pt} \ge \frac{1}{2} |w_+(x) - w_+(y)|^p \max \{ \eta(x), \eta(y) \}^p \\
& \hspace{15pt} \quad - C \max\{ w_+(x), w_+(y) \}^p |\eta(x) - \eta(y)|^p,
\end{split}
\end{align}
for some~$C \ge 1$ depending only on~$p$. Indeed,~\eqref{solareDGprop1} is a consequence of the fact that~$\supp(\varphi) \subset A^+(k)$. Inequality~\eqref{solareDGprop2} is also almost immediate, since, if~$x \in A^+(k, \tau)$ and~$y \in B_\tau \setminus A^+(k, \tau)$,
$$
|u(x) - u(y)|^{p - 2} \left( u(x) - u(y) \right) \left( \varphi(x) - \varphi(y) \right) = \eta(x)^p (w_+(x) + w_-(y))^{p - 1} w_+(x),
$$
and the conclusion follows by e.g.~Lemma~\ref{numestlem1} (with~$\theta = 0$) when~$p \ge 2$, and Jensen's inequality when~$p \in (1, 2)$. Finally, to prove~\eqref{solareDGprop3} we assume without loss of generality that~$u(x) \ge u(y)$. As~$x, y \in A^+(k, \tau)$, we have
\begin{align*}
& |u(x) - u(y)|^{p - 2} \left( u(x) - u(y) \right) \left( \varphi(x) - \varphi(y) \right) \\
& \hspace{50pt} = \left( w_+(x) - w_+(y) \right)^{p - 1} \left( \eta(x)^p w_+(x) - \eta(y)^p w_+(y) \right).
\end{align*}
Notice that, if~$\eta(x) \ge \eta(y)$, then
$$
|u(x) - u(y)|^{p - 2} \left( u(x) - u(y) \right) \left( \varphi(x) - \varphi(y) \right) \ge \left( w_+(x) - w_+(y) \right)^p \eta(x)^p,
$$
and~\eqref{solareDGprop3} trivially follows. On the other hand, if~$\eta(x) < \eta(y)$, we further compute
\begin{equation} \label{solareDGtech1}
\begin{aligned}
& |u(x) - u(y)|^{p - 2} \left( u(x) - u(y) \right) \left( \varphi(x) - \varphi(y) \right) \\
& \hspace{20pt} = \left( w_+(x) - w_+(y) \right)^p \eta(y)^p - \left( w_+(x) - w_+(y) \right)^{p - 1} w_+(x) \left( \eta(y)^p - \eta(x)^p \right).
\end{aligned}
\end{equation}
Then, we apply Lemma~\ref{numestlem3} with~$a = \eta(y)$,~$b = \eta(x)$ and
$$
\varepsilon = \frac{1}{2} \frac{w_+(x) - w_+(y)}{w_+(x)},
$$
to obtain that
\begin{align*}
& \left( w_+(x) - w_+(y) \right)^{p - 1} w_+(x) \left( \eta(y)^p - \eta(x)^p \right) \\
& \hspace{50pt} \le \frac{1}{2} \left( w_+(x) - w_+(y) \right)^p \eta(y)^p + \left[ 2 (p - 1) \right]^{p - 1} w_+(x)^p (\eta(y) - \eta(x))^p.
\end{align*}
This and~\eqref{solareDGtech1} lead to~\eqref{solareDGprop3} also when~$\eta(x) < \eta(y)$.

By virtue of~\eqref{solareDGprop1},~\eqref{solareDGprop2},~\eqref{solareDGprop3} and~\eqref{Kell}, we estimate
\begin{equation} \label{solareDGtech2}
\begin{aligned}
& \int_{B_\tau} \int_{B_\tau} |u(x) - u(y)|^{p - 2} \left( u(x) - u(y) \right) \left( \varphi(x) - \varphi(y) \right) d\mu \\
& \hspace{40pt} \ge \frac{1 - s}{C} \left[ [w_+]_{W^{s, p}(B_\rho)}^p + \int_{B_\rho} w_+(x) \left( \int_{B_\tau} \frac{w_-(y)^{p - 1}}{|x - y|^{n + s p}} \, dy \right) dx \right] \\
& \hspace{40pt} \quad - C (1 - s) \int_{B_\tau} \int_{B_\tau} \max \{ w_+(x), w_+(y) \}^p \, \frac{|\eta(x) - \eta(y)|^p}{|x - y|^{n + s p}} \, dx dy,
\end{aligned}
\end{equation}
for some~$C \ge 1$ depending on~$p$ and~$\Lambda$. Then, recalling the properties of~$\eta$, we get
\begin{align*}
& (1 - s) \int_{B_\tau} \int_{B_\tau} \max \{ w_+(x), w_+(y) \}^p \, \frac{|\eta(x) - \eta(y)|^p}{|x - y|^{n + s p}} \, dx dy \\
& \hspace{60pt} \le 2 (1 - s)\int_{B_\tau} w_+(x)^p \left( \int_{B_\tau} \frac{|\eta(x) - \eta(y)|^p}{|x - y|^{n + s p}} \, dy \right) dx \\
& \hspace{60pt} \le \frac{4^{1 + p} (1 - s)}{(\tau - \rho)^p} \int_{B_\tau} w_+(x)^p \left( \int_{B_\tau} \frac{dy}{|x - y|^{n - (1 - s) p}} \right) dx \\
& \hspace{60pt} \le C \frac{R^{(1 - s) p}}{(\tau - \rho)^p} \| w_+ \|_{L^p(B_R)}^p,
\end{align*}
with~$C$ now depending also on~$n$. By this, inequality~\eqref{solareDGtech2} becomes
\begin{equation} \label{solareDG2}
\begin{aligned}
& \int_{B_\tau} \int_{B_\tau} |u(x) - u(y)|^{p - 2} \left( u(x) - u(y) \right) \left( \varphi(x) - \varphi(y) \right) d\mu \\
& \hspace{20pt} \ge \frac{1 - s}{C} \left[ [w_+]_{W^{s, p}(B_\rho)}^p + \int_{B_\rho} w_+(x) \left( \int_{B_\tau} \frac{w_-(y)^{p - 1}}{|x - y|^{n + s p}} \, dy \right) dx \right] \\
& \hspace{20pt} \quad - C \frac{R^{(1 - s) p}}{(\tau - \rho)^p} \| w_+ \|_{L^p(B_R)}^p.
\end{aligned}
\end{equation}

We now turn our attention to the term integrated over~$\C_{B_\tau} \setminus B_\tau^2$ on the left-hand side of~\eqref{solareDG1}. By~\eqref{Ksimm},~\eqref{Kell} and the properties of~$\eta$, we have
\begin{equation} \label{solaretech2.1}
\begin{aligned}
& \iint_{\C_{B_\tau} \setminus B_\tau^2} |u(x) - u(y)|^{p - 2} \left( u(x) - u(y) \right) \left( \varphi(x) - \varphi(y) \right) d\mu \\
& \hspace{10pt} = 2 \int_{B_\tau} \eta(x)^p w_+(x) \left[ \int_{\R^n \setminus B_\tau} |u(x) - u(y)|^{p - 2} \left( u(x) - u(y) \right) K(x, y) \, dy \right] dx \\
& \hspace{10pt} \ge \frac{2 (1 - s)}{\Lambda} \int_{A^+(k, \rho)} w_+(x) \left[ \int_{ \left( B_{r_0}(x) \cap \{ u(x) \ge u(y) \} \right) \setminus B_\tau} \frac{\left( u(x) - u(y) \right)^{p - 1}}{|x - y|^{n + s p}} \, dy \right] dx \\
& \hspace{10pt} \quad - 2 (1 - s) \Lambda \int_{A^+ \left( k, \frac{\tau + \rho}{2} \right)} w_+(x) \left[ \int_{\{ u(y) > u(x) \} \setminus B_\tau} \frac{\left( u(y) - u(x) \right)^{p - 1}}{|x - y|^{n + s p}} \, dy \right] dx.
\end{aligned}
\end{equation}
Now, on the one hand
\begin{equation} \label{solaretech2.2}
\begin{aligned}
& \int_{A^+(k, \rho)} w_+(x) \left[ \int_{ \left( B_{r_0}(x) \cap \{ u(x) \ge u(y) \} \right) \setminus B_\tau} \frac{\left( u(x) - u(y) \right)^{p - 1}}{|x - y|^{n + s p}} \, dy \right] dx \\
& \hspace{50pt} \ge \int_{B_\rho} w_+(x) \left[ \int_{ \left( B_{r_0}(x) \cap A^-(k) \right) \setminus B_\tau} \frac{\left( w_+(x) + w_-(y) \right)^{p - 1}}{|x - y|^{n + s p}} \, dy \right] dx \\
& \hspace{50pt} \ge \int_{B_\rho} w_+(x) \left( \int_{ B_{r_0}(x) \setminus B_\tau} \frac{w_-(y) ^{p - 1}}{|x - y|^{n + s p}} \, dy \right) dx.
\end{aligned}
\end{equation}
On the other hand, for~$x \in B_{(\tau + \rho) / 2}$ and~$y \in \R^n \setminus B_\tau$, it holds 
$$
|x - y| \ge |y| - |x| \ge |y| - \frac{\tau + \rho}{2 \tau} |y| = \frac{\tau - \rho}{2 \tau} |y| \ge \frac{\tau - \rho}{2 R} |y|,
$$
and therefore, recalling definition~\eqref{nsiTailudef},
\begin{align*}
& \int_{A^+ \left( k, \frac{\tau + \rho}{2} \right)} w_+(x) \left[ \int_{\{ u(y) > u(x) \} \setminus B_\tau} \frac{\left( u(y) - u(x) \right)^{p - 1}}{|x - y|^{n + s p}} \, dy \right] dx \\
& \hspace{70pt} \le \left( \frac{2 R}{\tau - \rho} \right)^{n + s p} \int_{B_{\frac{\tau + \rho}{2}}} w_+(x) \left[ \int_{\R^n \setminus B_\tau} \frac{w_+(y)^{p - 1}}{|y|^{n + s p}} \, dy \right] dx \\
& \hspace{70pt} \le \frac{C}{1 - s} \frac{R^{n + s p}}{(\tau - \rho)^{n + s p}} \| w_+ \|_{L^1(B_R)} \overline{\Tail}(w_+; r)^{p - 1}.
\end{align*}
By this,~\eqref{solaretech2.1} and~\eqref{solaretech2.2}, it follows that
\begin{equation} \label{solareDG3}
\begin{aligned}
& \iint_{\C_{B_\tau} \setminus B_\tau^2} |u(x) - u(y)|^{p - 2} \left( u(x) - u(y) \right) \left( \varphi(x) - \varphi(y) \right) d\mu \\
& \hspace{70pt} \ge \frac{1 - s}{C} \int_{B_\rho} w_+(x) \left( \int_{ B_{r_0}(x) \setminus B_\tau} \frac{w_-(y) ^{p - 1}}{|x - y|^{n + s p}} \, dy \right) dx \\
& \hspace{70pt} \quad - C \frac{R^{n + s p}}{(\tau - \rho)^{n + s p}} \| w_+ \|_{L^1(B_R)} \overline{\Tail}(w_+; r)^{p - 1}.
\end{aligned}
\end{equation}

By putting together~\eqref{solareDG1},~\eqref{solareDG2} and~\eqref{solareDG3}, we find that
\begin{equation} \label{solareDG4}
\begin{aligned}
& [w_+] _{W^{s, p}(B_\rho)}^p + \int_{B_\rho} w_+(x)  \left( \int_{B_{r_0}(x)} \frac{w_-(y)^{p - 1}}{|x - y|^{n + s p}} \, dy \right) dx \\
& \hspace{20pt} \le \frac{C}{1 - s} \left[ \frac{R^{(1 - s) p}}{(\tau - \rho)^p} \| w_+ \|_{L^p(B_R)}^p + \frac{R^{n + s p}}{(\tau - \rho)^{n + s p}} \| w_+ \|_{L^1(B_R)} \overline{\Tail}(w_+; r)^{p - 1} \right] \\
& \hspace{20pt} \quad + \frac{C}{1 - s} \int_{\R^n} f(x, u(x)) \varphi(x) \, dx.
\end{aligned}
\end{equation}

To finish the proof, we now only need to control the term appearing on the third line of~\eqref{solareDG4}. By~\eqref{fbounds} and the properties of~$\eta$, we have
\begin{equation} \label{fest}
\int_{\R^n} f(x, u(x)) \varphi(x) \, dx \le \int_{A^+ \left( k, \frac{\tau + \rho}{2} \right)} \Big( d_1 + d_2 |u(x)|^{q - 1} \Big) \eta(x)^p w_+(x) \, dx.
\end{equation}
To estimate the term involving~$d_1$, we simply apply weighted Young's inequality to get
$$
\int_{A^+ \left( k, \frac{\tau + \rho}{2} \right)} d_1 \eta(x)^p w_+(x) \, dx \le C \left( (\delta d_1)^{\frac{p}{p - 1}} |A^+(k, R)| + \frac{\| w_+ \|_{L^p(B_R)}^p}{\delta^p} \right),
$$
for any~$\delta > 0$. By choosing~$\delta := (\tau - \rho) R^{s - 1} \le R^s$, this yields in turn
\begin{equation} \label{d1est}
\begin{aligned}
& \int_{A^+ \left( k, \frac{\tau + \rho}{2} \right)} d_1 \eta(x)^p w_+(x) \, dx \\
& \hspace{50pt} \le C \left( R^{\frac{s p}{p - 1}} d_1^{\frac{p}{p - 1}} |A^+(k, R)| + \frac{R^{(1 - s) p}}{(\tau - \rho)^{p}} \| w_+ \|_{L^p(B_R)}^p \right).
\end{aligned}
\end{equation}
Note that, when~$d_2 = 0$ we are already led to~\eqref{solareDGine}, with~$\varepsilon = s p / n$. On the other hand, when~$d_2 > 0$ the proof of~\eqref{solareDGine} is more involved. We consider separately the two possibilities~$1 < q \le p$ and~$p < q < p^*_s$.

Suppose that~$1 < q \le p$. In this case, we take~$k \ge 1$. For~$x \in A^+(k)$ we have that~$u(x) > k \ge 1$, and thus
$$
|u(x)|^{q - 1} \le |u(x)|^{p - 1} = |w_+(x) + k|^{p - 1} \le \max \left\{ 1, 2^{p - 2} \right\} \left( w_+(x)^{p - 1} + k^{p - 1} \right).
$$
Accordingly, by Young's inequality
\begin{align*}
\int_{A^+ \left( k, \frac{\tau + \rho}{2} \right)} d_2 |u(x)|^{q - 1} \eta(x)^p w_+(x) \, dx & \le C \int_{A^+ \left( k, \frac{\tau + \rho}{2} \right)} \Big( w_+(x)^p + k^{p - 1} w_+(x) \Big) \, dx \\
& \le C \left( \| w_+ \|_{L^p(B_R)}^p + k^p |A^+(k, R)| \right) \\
& \le C \left[ \frac{R^{(1 - s) p}}{(\tau - \rho)^p} \| w_+ \|_{L^p(B_R)}^p + \frac{k^p}{R^{s p}} |A^+(k, R)| \right],
\end{align*}
provided~$R \le 1$ and with~$C$ depending also on~$d_2$. This and estimates~\eqref{solareDG4},~\eqref{fest} yield~\eqref{solareDGine} when~$1 < q \le p$, again with~$\varepsilon = s p / n$.

Finally, we deal with the second case, when~$p < q < p^*_s$, with~$n \ge s p$. By adding the quantity~$d_2 (1 - s)^{- 1} \| u \|_{L^q(A^+(k, \rho))}^q$ to both sides of~\eqref{solareDG4} and recalling~\eqref{fest},~\eqref{d1est}, we get
\begin{equation} \label{solareDG5}
\begin{aligned}
& [w_+]_{W^{s, p}(B_\rho)}^p + \frac{d_2}{1 - s} \| u \|_{L^q(A^+(k, \rho))}^q + \int_{B_\rho} w_+(x) \left( \int_{B_{r_0}(x)} \frac{w_-(y)^{p - 1}}{|x - y|^{n + s p}} \, dy \right) dx \\
& \hspace{20pt} \le \frac{C_1}{1 - s} \Bigg[ \frac{R^{(1 - s) p}}{(\tau - \rho)^p} \| w_+ \|_{L^p(B_R)}^p + \frac{R^{n + s p}}{(\tau - \rho)^{n + s p}} \| w_+ \|_{L^1(B_R)} \overline{\Tail}(w_+; r)^{p - 1} \\
& \hspace{20pt} \quad + d_1^{\frac{p}{p - 1}} |A^+(k, R)| + d_2 \int_{A^+ \left( k, \frac{\tau + \rho}{2} \right)} |u(x)|^{q - 1} \Big( |u(x)| + \eta(x) w_+(x) \Big) dx \Bigg],
\end{aligned}
\end{equation}
for some~$C_1 \ge 1$ depending only on~$n$,~$p$ and~$\Lambda$. Then, we observe that, if~$k \ge 0$,
\begin{align*}
|u(x)|^q & = |(1 - \eta(x)) u(x) + \eta(x) w_+(x) + \eta(x) k|^q \\
& \le 3^{q - 1} \Big( (1 - \eta(x))^q |u(x)|^q + \left( \eta(x) w_+(x) \right)^q + k^q \Big),
\end{align*}
for any~$x \in A^+(k)$. Hence, using once again Young's inequality,
\begin{align*}
& d_2 \int_{A^+ \left( k, \frac{\tau + \rho}{2} \right)} |u(x)|^{q - 1} \Big( |u(x)| + \eta(x) w_+(x) \Big) dx \\
& \hspace{50pt} \le 2 d_2 \int_{A^+ \left( k, \frac{\tau + \rho}{2} \right)} \Big[ |u(x)|^q + \left( \eta(x) w_+(x) \right)^q \Big] dx \\
& \hspace{50pt} \le 4^q d_2 \int_{A^+(k, \tau)} \Big[ (1 - \eta(x))^q |u(x)|^q + \left( \eta(x) w_+(x) \right)^q + k^q \Big] \, dx.
\end{align*}
Consider now the two quantities~$\sigma \in (0, s]$ and~$\varepsilon_\sigma \in (0, sp / n)$ defined in~\eqref{sigmamindef} and~\eqref{epsigmamindef}, respectively. By arguing as in the last part of the proof of Proposition~\ref{minareDGprop}, we deduce from the estimate above that
\begin{align*}
& \frac{C_1 d_2}{1 - s} \int_{A^+ \left( k, \frac{\tau + \rho}{2} \right)} |u(x)|^{q - 1} \Big( |u(x)| + \eta(x) w_+(x) \Big) dx \\
& \hspace{40pt} \le \frac{1}{2} [w_+]_{W^{s, p}(B_\tau)}^p + \frac{C d_2}{1 - s} \| u \|_{L^q(B_\tau \setminus B_\rho)}^q \\
& \hspace{40pt} \quad + \frac{C}{1 - s} \left[ \frac{R^{(1 - s) p}}{(\tau - \rho)^p} \| w_+ \|_{L^p(B_R)}^p + d_2 k^p R^{- n \varepsilon_\sigma} |A^+(k, R)|^{1 - \frac{s p}{n} + \varepsilon_\sigma} \right],
\end{align*}
where~$C$ may now depend on~$q$ as well, and provided~$R$ is smaller than a quantity~$R_0$ depending only on~$n$,~$s$,~$p$,~$q$,~$\Lambda$,~$d_2$,~$\| u \|_{L^{p^*_\sigma}}$. Combining this with~\eqref{solareDG5} and rearranging appropriately the summands, we find that
\begin{equation} \label{solareDG6}
\begin{aligned}
\Phi(\rho) & \le C_\flat \left[ \Phi(\tau) - \Phi(\rho) \right] + \frac{C_\flat}{1 - s} \Bigg[ \left( d_1^{\frac{p}{p - 1}} + d_2 k^p R^{- n \varepsilon_\sigma} \right) |A^+(k, R)|^{1 - \frac{s p}{n} + \varepsilon_\sigma} \\
& \quad + \frac{R^{(1 - s) p}}{(\tau - \rho)^p} \| w_+ \|_{L^p(B_R)}^p + \frac{R^{n + s p}}{(\tau - \rho)^{n + s p}} \| w_+ \|_{L^1(B_R)} \overline{\Tail}(w_+; r)^{p - 1} \Bigg],
\end{aligned}
\end{equation}
for some~$C_\flat \ge 1$ depending on~$n$,~$p$,~$q$,~$\Lambda$, and with
$$
\Phi(t) := [w_+]_{W^{s, p}(B_t)}^p + \frac{d_2}{1 - s} \| u \|_{L^q(B_t)}^q + \int_{B_t} w_+(x) \left( \int_{B_{r_0}(x)} \frac{w_-(y)^{p - 1}}{|x - y|^{n + s p}} \, dy \right) dx,
$$
for any~$0 < t \le R$. Estimate~\eqref{solareDGine} (with~$\varepsilon = \varepsilon_\sigma$) now follows by adding the quantity~$C_\flat \Phi(\rho)$ to both sides of~\eqref{solareDG6}, dividing by~$1 + C_\flat$ and applying Lemma~\ref{induclem}.
\end{proof}

With the aid of Proposition~\ref{solareDGprop} and Theorems~\ref{DGboundthm},~\ref{DGholdthm}, the boundedness and H\"older continuity of the solutions of~\eqref{Lu=f} are readily established. Furthermore, the Harnack inequality for non-negative solutions follows by Theorem~\ref{DGharthm}. Thus, Theorems~\ref{solboundthm}-\ref{solharthm} are proved. For more details, see the proofs of the analogous Theorems~\ref{minboundthm}-\ref{minharthm} for minimizers, at the end of Section~\ref{minsec}.

\section*{Acknowledgments}

The author is supported by the~``Mar\'ia de Maeztu'' MINECO grant~MDM-2014-0445 and the MINECO grant~MTM2014-52402-C3-1-P.

The author also warmly thanks Xavier Cabr\'e, Moritz Kassmann and Enrico Valdinoci for several inspiring conversations on the subject of this paper.

Finally, the author wishes to express his sincere gratitude to the anonymous referee for his keen comments and observations.

\end{document}